\documentclass[reqno, 12pt]{amsart}

\usepackage{amsmath} \usepackage{amsthm} \usepackage{amsfonts}
\usepackage{amssymb} \usepackage{fullpage} \usepackage{placeins}
\usepackage{graphicx}

\newcommand{\mb}{\mathbf} \newcommand{\mc}{\mathcal}
 \renewcommand{\Re}{\mathrm{Re}\,}
\renewcommand{\Im}{\mathrm{Im}\,} \newcommand{\rg}{\mathrm{rg}\,}
\newcommand{\N}{\mathbb{N}} \newcommand{\R}{\mathbb{R}}
\newcommand{\C}{\mathbb{C}} \newcommand{\Z}{\mathbb{Z}}
\newcommand{\B}{\mathbb{B}} \renewcommand{\S}{\mathbb{S}}
\newcommand{\supp}{\mathrm{supp}\,} 
 \newcommand{\Dh}{\mathcal{D}}
\newcommand{\const}{\mathrm{const}}

 \newtheorem{lemma}{Lemma}[section]
\newtheorem{theorem}[lemma]{Theorem}
\newtheorem{corollary}[lemma]{Corollary}
\newtheorem{proposition}[lemma]{Proposition}
 \theoremstyle{remark}
\newtheorem{remark}[lemma]{Remark}
 \theoremstyle{definition}
\newtheorem{definition}[lemma]{Definition}

\numberwithin{equation}{section}

\title[Blowup for wave maps]{Hyperboloidal similarity coordinates and
a globally stable blowup profile for supercritical wave maps}

\author{Pawe\l {} Biernat} \address{Rheinische Friedrich-Wilhelms-Universit\"at
  Bonn, Life \& Medical Sciences Institute,
Carl-Troll-Strasse 31,
D-53115 Bonn, Germany}
\email{pawel.biernat@gmail.com}

\author{Roland Donninger} \address{Rheinische
  Friedrich-Wilhelms-Universit\"at Bonn, Mathematisches Institut,
  Endenicher Allee 60, D-53115 Bonn, Germany}
\email{donninge@math.uni-bonn.de}

\address{Universit\"at Wien, Fakult\"at f\"ur Mathematik,
  Oskar-Morgenstern-Platz 1, A-1090 Vienna, Austria}
\email{roland.donninger@univie.ac.at}

\author{Birgit Sch\"orkhuber} \address{Universit\"at Wien, Fakult\"at
  f\"ur Mathematik, Oskar-Morgenstern-Platz 1, A-1090 Vienna, Austria}
\email{birgit.schoerkhuber@univie.ac.at}

\address{Karlsruhe Institute of Technology, Department of Mathematics,
  Institute for Analysis,
  Englerstra{\ss}e 2, 76131 Karlsruhe, Germany}
\email{birgit.schoerkhuber@kit.edu}

\thanks{Roland Donninger is supported by the Alexander von Humboldt Foundation via
a Sofja Kovalevskaja Award endowed by the German Federal Ministry of Education
and Research. Birgit Sch\"orkhuber is supported by the Austrian Science Fund
(FWF) via the Hertha Firnberg Program, Project Nr. T 739-N25. Partial support by the Deutsche Forschungsgemeinschaft 
(DFG), CRC 1060 'The Mathematics of Emergent Effects', is also gratefully acknowledged.}

\begin{document}

\begin{abstract}
  We consider co-rotational wave maps from (1+3)-dimensional Minkowski
  space into the three-sphere. This model exhibits an explicit blowup
  solution and we prove the asymptotic nonlinear stability of this
  solution in the whole space under small perturbations of the initial
  data. The key ingredient is the introduction of a novel coordinate
  system that allows one to track the evolution past the blowup time
  and almost up to the Cauchy horizon of the singularity. As a
  consequence, we also obtain a result on continuation beyond blowup.
\end{abstract}

\maketitle

\section{Introduction}
\noindent Wave maps $U: \R^{1,3}\to\S^3$ from $(1+3)$-dimensional
Minkowski space into the three-sphere are defined as critical points
of the action functional
\begin{equation}
  \label{eq:wmaction} \int_{\R^{1,3}}\partial^\mu U^a\partial_\mu U^b
  g_{ab}\circ U, 
\end{equation}
where $g$ is the standard round metric
on $\S^3$ and Einstein's summation convention is in force.  By
choosing spherical coordinates $(t,r,\theta,\varphi)$ on Minkowski
space and hyperspherical coordinates on the three-sphere, one may
restrict oneself to so-called co-rotational maps which take the form
$U(t,r,\theta,\varphi)=(\psi(t,r),\theta,\varphi)$.  Under this
symmetry reduction, the Euler-Lagrange equations associated to the
action \eqref{eq:wmaction} reduce to the single scalar wave equation
\begin{equation}
  \label{eq:main}
  \left (\partial_t^2-\partial_r^2-\frac{2}{r}\partial_r \right )\psi(t,r)+\frac{\sin(2\psi(t,r))}{r^2}=0
\end{equation}
for the angle $\psi$.  Note that the singularity at the center $r=0$
enforces the boundary condition $\psi(t,0)=\frac{m}{2}\pi$ for $m\in
\Z$. To begin with, we restrict ourselves to $m=0$. By
testing Eq.~\eqref{eq:main} with $\partial_t \psi(t,r)$, we obtain the
conserved energy
\begin{equation}
  \label{eq:energy}
  \int_0^\infty \left
    [\tfrac12 |\partial_t\psi(t,r)|^2+  \tfrac12 |\partial_r\psi(t,r)|^2+\frac{\sin^2(\psi(t,r))}{r^2}\right
  ]r^2 dr
\end{equation}
and finiteness of the latter requires
$\lim_{r\to\infty}\psi(t,r)=n\pi$ for $n\in \Z$.  

Despite the existence of a positive definite energy,
Eq.~\eqref{eq:main} develops singularities in finite time. This was
first demonstrated by Shatah \cite{Sha88} who constructed a
self-similar solution $\psi_T(t,r)=f_0(\frac{r}{T-t})$ to
Eq.~\eqref{eq:main} by a variational argument. Here, $T>0$ is a free
parameter (the \emph{blowup time}). In fact, $f_0(\rho)=2\arctan\rho$,
as was observed later \cite{TurSpe90}. The solution $\psi_T(t,r)$ is
perfectly smooth for $t<T$ but develops a gradient blowup at the
spacetime point $(t,r)=(T,0)$. In \cite{Don11, DonSchAic12,
  CosDonXia16, CosDonGlo17} it is shown that the blowup solution
$\psi_T(t,r)=f_0(\frac{r}{T-t})$ is asymptotically stable in the
backward lightcone of the blowup point $(T,0)$ under small perturbations of the
initial data.  This result leaves open two major questions which shall
be addressed in the present paper:
\begin{itemize}
\item How does the solution behave outside the backward lightcone?
\item Is it possible to continue the solution beyond the singularity
  in a well-defined way?
\end{itemize}

As for the second question, we note that $\psi_T(t,r)$ is defined for
$t<T$ only.  However, $\psi_T$ is closely related to the principal
value of the argument function $\arg$ in complex analysis. 
More precisely, we have $\psi_T(t,r)=2\arg(T-t+ir)$ and this
suggests that there exists a natural continuation of $\psi_T$ beyond
the blowup time $t=T$.  Indeed, the tangent half-angle formula yields
the representation
\[ \arg z= 2\arctan\left (\frac{\Im z}{\Re z+\sqrt{(\Re z)^2 +(\Im
      z)^2}} \right ),
\]
valid if $\Im z\not= 0$ or $\Re z>0$, and this leads to the more
general blowup solution
\[ \psi_T^*(t,r):=4\arctan\left
  (\frac{r}{T-t+\sqrt{(T-t)^2+r^2}}\right ).
\]
The skeptical reader may check by a direct computation that $\psi_T^*$
is indeed a solution to Eq.~\eqref{eq:main} for all $t\in \R$ and
$r>0$.  The point is that $\psi_T^*$ is smooth everywhere away from
the center and thus, $\psi_T^*$ extends $\psi_T$ smoothly beyond the
blowup time $t=T$. Furthermore,
\[ \lim_{r\to 0+}\psi(t,r)=2\pi \]
if $t>T$, whereas $\psi(t,0)=0$ for $t<T$.  Consequently, the blowup
is accompanied by a change of the boundary condition at the center.
After the blowup, the solution $\psi_T^*$ settles down to the constant
function $2\pi$, i.e., for any $r>0$ we have $\lim_{t\to\infty}\psi_T^*(t,r)=2\pi$.

\subsection{Statement of the main result}
In view of the boundary condition $\psi(t,0)=0$ it is
natural to switch to the new variable
\[ \widehat u(t,r):=\frac{\psi(t,r)}{r}. \]
In terms of $\widehat u$, Eq.~\eqref{eq:main} reads
\begin{equation}
  \label{eq:mainuhat}
  \left (\partial_t^2-\partial_r^2-\frac{4}{r}\partial_r \right
  )\widehat u(t,r)+\frac{\sin(2r\widehat u(t,r))-2r\widehat u(t,r)}{r^3}=0,
\end{equation}
which is a radial, semilinear wave equation in 5 space dimensions. For
notational purposes it is convenient to rewrite
Eq.~\eqref{eq:mainuhat} as
\begin{equation}
  \label{eq:mainu}
  \left (\partial_t^2-\Delta_x\right )u(t,x)=\frac{2|x|u(t,x)-\sin(2|x|u(t,x))}{|x|^3}
\end{equation}
for $u: \R\times \R^5\to \R$ given by $u(t,x)=\widehat u(t,|x|)$.  By
the above, Eq.~\eqref{eq:mainu} has the explicit one-parameter family
$\{u_T^*: T\in \R\}$ of blowup solutions given by
\[ u_T^*(t,x):=\frac{4}{|x|}\arctan\left
  (\frac{|x|}{T-t+\sqrt{(T-t)^2+|x|^2}}\right ). \]
We introduce the following spacetime region, depicted in
Fig.~\ref{fig:Omega}.
\begin{definition}
  For $T,b\in \R$ we set
  \[ \Omega_{T,b}:=\{(t,x)\in \R\times \R^5: 0\leq t<T+b|x|\}. \]
\end{definition}

\begin{figure}[ht]
  \centering
  \includegraphics{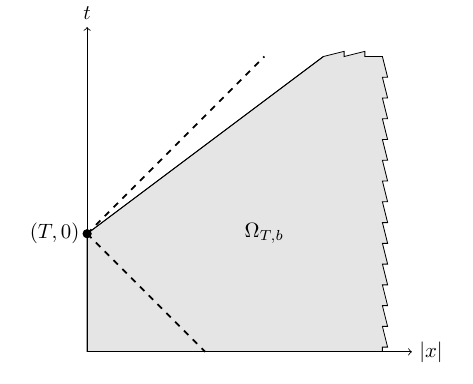}
  \caption{A spacetime diagram depicting the region
    $\Omega_{T,b}$. The dashed lines are the boundaries of the forward
    and backward lightcones of the point $(T,0)$. The solid line
    emerging from $(T,0)$ has slope $b$ and marks the upper boundary
    of the shaded region $\Omega_{T,b}$.}
  \label{fig:Omega}
\end{figure}

\FloatBarrier

Note that $u_T^*\in C^\infty(\Omega_{T,b})$. 
Our main result
establishes the stability of $u_T^*$ under small perturbations of the
initial data.

\begin{theorem}
  \label{thm:main}
  Fix $b\in (0,1)$ and $m\in \N$, $m\geq 8$. Then there exist positive
  constants $\delta,\epsilon,M,\omega_0$ such that the following
  holds.
  \begin{enumerate}
  \item For any pair of radial functions
    $(f,g)\in C^\infty(\R^5)\times C^\infty(\R^5)$, supported in the
    ball $\B^5_\epsilon$ and satisfying
    \[ \|(f,g)\|_{H^m(\R^5)\times H^{m-1}(\R^5)}\leq
    \tfrac{\delta}{M}, \]
    there exists a $T\in [1-\delta,1+\delta]$ and a unique function
    $u\in C^\infty(\Omega_{T,b})$ that satisfies Eq.~\eqref{eq:mainu}
    for all $(t,x)\in \Omega_{T,b}$ and
    \begin{align*}
      u(0,x)&=u_1^*(0,x)+f(x) \\
      \partial_0 u(0,x)&=\partial_0 u_1^*(0,x)+g(x)
    \end{align*}
    for all $x\in \R^5$.

  \item The solution $u$ converges to $u_T^*$ in the sense
    that\footnote{Note that $\|(u_T^*\circ
      \eta_T)(s,\cdot)\|_{H^{m-3}(\B^5_R)}\simeq e^s$ and
$\|\partial_s (u_T^*\circ
      \eta_T)(s,\cdot)\|_{H^{m-4}(\B^5_R)}\simeq e^s$. This motivates
      the normalization factors $e^{-s}$ on the left-hand sides of the estimates.}

    \begin{align*}
      e^{-s}\|(u\circ \eta_T)(s,\cdot)-(u_T^* \circ
      \eta_T)(s,\cdot)\|_{H^{m-3}(\B^5_R)}
      &\leq \delta e^{-\omega_0s} \\
      e^{-s}\|\partial_s
      (u\circ\eta_T)(s,\cdot)-\partial_s (u_T^*\circ
      \eta_T)(s,\cdot)\|_{H^{m-4}(\B^5_R)}&\leq \delta e^{-\omega_0 s}
    \end{align*}
    for all $s\geq 0$, where
    \begin{align*}
      \eta_T(s,y)&=(T+e^{-s}h(y), e^{-s}y),\qquad h(y)=\sqrt{2+|y|^2}-2 \\
      R&=\frac{2b+\sqrt{2(1+b^2)}}{1-b^2}.
    \end{align*}

  \item In the domain
    $\Omega_{T,b}\setminus \eta_T([s_0,\infty)\times \B_R^5)$, where $s_0=\log(-\frac{h(0)}{1+2\epsilon})$, we have
    $u=u_1^*$.
  \end{enumerate}
\end{theorem}

\begin{remark}
  It seems appropriate to comment on the comparatively high degree of
  regularity $m\geq 8$ in Theorem \ref{thm:main}.
 The proof of
 Theorem \ref{thm:main}
 rests on the analysis of the evolution in a novel coordinate
  system which uses a hyperboloidal foliation of spacetime, see
  below. Therefore, it is necessary to first
  transport the data given at $t=0$ to the initial hyperboloid. This
  is done via the standard Cauchy evolution and in order to evaluate
  the solution on the hyperboloid, we use the Sobolev embedding which
  requires a sufficiently high degree of regularity. We are generous
  and assume $m\geq 8$. 
\end{remark}

\begin{remark}
 To avoid any possible confusion, we note that the parameter $b$ in
 Theorem \ref{thm:main} is fixed and only enters via $R$, i.e., it
 will not show up in the proofs below.
\end{remark}

\subsection{Discussion}
Theorem \ref{thm:main} gives a complete description of the evolution
up to the blowup time $t=T$. In particular, Theorem \ref{thm:main}
shows that the solution does not develop singularities outside the
backward lightcone of $(T,0)$ at some time $t<T$, a scenario which
could not be ruled out by the results in \cite{Don11, DonSchAic12}.
Furthermore, causally separated from the blowup point $(T,0)$, the
evolution is controlled even beyond the blowup time and the whole
region $\Omega_{T,b}$ is free of singularities.  In other words, we
also obtain some partial information on the evolution \emph{after} the blowup.
In fact, by taking $b$ close to $1$, we approach the
\emph{Cauchy horizon} of the singularity, that is, the boundary of the
future lightcone of the point $(T,0)$, see Fig.~\ref{fig:Omega}. The
solution is therefore controlled everywhere outside the \emph{future}
lightcone of the blowup point $(T,0)$.

The key ingredient for the proof of Theorem \ref{thm:main} is the
introduction of a novel coordinate system $(s,y)$ which we call
``hyperboloidal similarity coordinates'' (HSC). The coordinates are
defined by the function $\eta_T$ in Theorem \ref{thm:main}, i.e.,
\[ (t,x)=\eta_T(s,y)=(T+e^{-s}h(y),e^{-s}y), \]
and depicted in Fig.~\ref{fig:HSC}.
\begin{figure}[ht]
  \centering
  \includegraphics{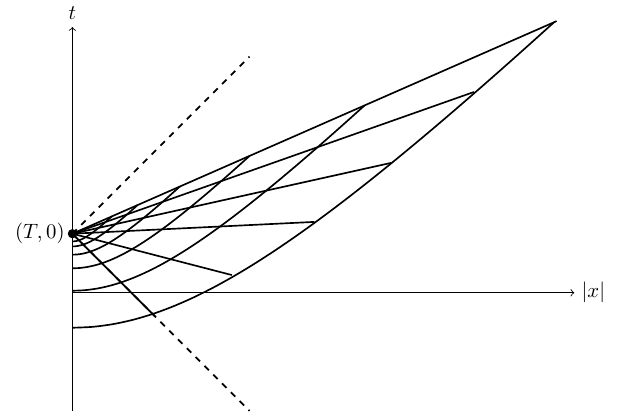}
  \caption{The hyperboloidal similarity coordinates. The hyperboloids
    are the lines $s=\const$ and the straight lines emerging radially
    from the blowup point $(T,0)$ correspond to $y=\const$. The dashed
    lines are the boundaries of the forward and backward lightcones of
    the singularity.}
  \label{fig:HSC}
\end{figure}
The coordinate system is hyperboloidal in the sense of \cite{Fri83,
  Zen11} but at the same time compatible with self-similarity, that is
to say, the fraction $\frac{x}{T-t}=-\frac{y}{h(y)}$ is independent of
the new time coordinate $s$.  The hyperboloidal similarity coordinates
are a generalization of the standard similarity coordinates
$(\tau,\xi)=(-\log(T-t), \frac{x}{T-t})$ which are traditionally used
in the study of self-similar blowup \cite{MerZaa03, MerZaa05}. By
their very definition, the coordinates $(\tau,\xi)$ are restricted to
$t<T$. This limitation is not present in the HSC. More precisely, the
point is that the slices of constant $s$ are curved, as opposed to the
constant $\tau$ slices. As a consequence, the coordinate system
$(s,y)$ covers a much larger portion of spacetime than the traditional
similarity coordinates $(\tau,\xi)$, see Fig.~\ref{fig:HSC}.

The bulk of the paper is concerned with the development of a nonlinear
perturbation theory in the coordinates $(s,y)$ that is capable of
controlling the wave maps flow near the blowup solution $u_T^*$. The
approach is similar in spirit to the earlier works \cite{Don11,
  DonSchAic12}, where the standard similarity coordinates $(\tau,\xi)$
are used, and based on semigroup methods, nonself-adjoint spectral
theory, and ideas from infinite-dimensional dynamical systems.

\subsection{Related work}

The problem of finite-time blowup for wave maps attracted a lot of
interest in the recent past. The bulk of the literature focuses on the
two-dimensional case which is energy-critical. The existence of
finite-time blowup for energy-critical wave maps into the two-sphere
has first been observed numerically in the work of Bizo\'n-Chmaj-Tabor
\cite{BizChmTab01}. Rigorously, the existence of blowup solutions was
proved by Krieger-Schlag-Tataru \cite{KriSchTat08},
Rodnianski-Sterbenz \cite{RodSte10}, and Rapha\"el-Rodnianski
\cite{RapRod12}, see also \cite{Sha16, CanKri15}. 
We remark that the blowup in the energy-critical case is of type II
and proceeds by dynamical rescaling of a soliton, cf.~\cite{Str03}.
In fact, there are by now powerful nonperturbative techniques for
energy-critical equations which allow one to prove versions of the
celebrated \emph{soliton resolution conjecture}, see the work by C\^ote \cite{Cot15},
C\^ote-Kenig-Lawrie-Schlag \cite{CotKenLawSch15a, CotKenLawSch15b},
and, very recently, Jia-Kenig \cite{JiaKen15},
Duyckaerts-Jia-Kenig-Merle \cite{DuyJiaKenMer16}, see also
\cite{Gri17, JenLaw17}. Large-data global well-posedness and
scattering is addressed in the fundamental work by 
Tataru-Sterbenz \cite{SteTat10a, SteTat10b} and Krieger-Schlag
\cite{KriSch12}.

The present paper deals with energy-supercritical wave maps and much
less is known in this case. In the equivariant setting, local
well-posedness at critical regularity was settled by
Shatah-Tahvildar-Zahdeh \cite{ShaTah94} and the general case is
treated in the papers by Tataru \cite{Tat01}, Tao \cite{Tao01a,
  Tao01b}, Klainerman-Rodnianski \cite{KlaRod01}, Shatah-Struwe
\cite{ShaStr02}, Nahmod-Stefanov-Uhlenbeck \cite{NahSteUhl03}, and
Krieger \cite{Kri03}, see also \cite{Tat05}. For energy-supercritical
equations the existence of self-similar solutions is typical and in
fact, the model under investigation has plenty of them
\cite{Biz00}. As already mentioned, the local stability of the
``ground-state'' self-similar solution in the backward lightcone was
established in \cite{Don11, DonSchAic12, CosDonXia16, CosDonGlo17},
see \cite{DonSch12, DonSch14, Don14, DonSch16, DonSch16a} for other
equations. The ``ground state'' actually exists in any dimension
\cite{BizBie15, BieBizMal16} and its stability in the backward
lightcone was recently
established in \cite{CosDonGlo17, ChaDonGlo17}.  Furthermore, Dodson-Lawrie
\cite{DodLaw15} showed that type II blowup is impossible.
This, however, does not mean that every blowup is self-similar.
Indeed, a novel blowup mechanism in high dimensions was discovered
recently by
Ghoul-Ibrahim-Nguyen \cite{GhoIbrNgu17}, building on the work by
Merle-Rapha\"el-Rodnianski \cite{MerRapRod15} on the supercritical
nonlinear Schr\"odinger equation.
To avoid confusion, we remark that this
type of nonself-similar blowup is also called ``type II'' in
\cite{GhoIbrNgu17}
but is
different from the notion of type II blowup used by Dodson-Lawrie
\cite{DodLaw15}.
Furthermore, Germain \cite{Ger08,
  Ger09} studied self-similar wave maps, Widmayer considered the
question of uniqueness of weak wave maps \cite{Wid15} and
Chiodaroli-Krieger \cite{ChiKri17} constructed large global
solutions. Finally, we remark that stable self-similar blowup also
exists for wave maps with negatively curved targets \cite{CazShaTah98,
  DonGlo17}.

\subsection{Notation}

Most of the notation we use is standard in the field or
self-explanatory. We write $\B_R^d(x_0):=\{x\in \R^d: |x-x_0|<R\}$ and
abbreviate $\B_R^d:=\B_R^d(0)$ as well as $\B^d:=\B_1^d(0)$. 
The symbol $\N:=\{1,2,3,\dots\}$ denotes the natural numbers and we
set $\N_0:=\{0\}\cup \N$.
We denote by
$H^k(\R^d)$, $k\in \N_0$, the completion of the Schwartz space $\mc S(\R^d)$ with
respect to the Sobolev norm 
\[ \|f\|_{H^k(\R^d)}^2:=\sum_{|\alpha|\leq k}\|\partial^\alpha
f\|_{L^2(\R^d)}^2. \]
Here, we employ the usual multi-index notation
\[ \partial^\alpha
f:=\partial_1^{\alpha_1}\partial_2^{\alpha_2}\dots\partial_d^{\alpha_d}f \]
for $\alpha=(\alpha_1,\alpha_2,\dots,\alpha_d)\in \N_0^d$ and
$|\alpha|:=\alpha_1+\alpha_2+\dots+\alpha_d$.
The homogeneous Sobolev space $\dot H^k(\R^d)$ is defined analogously
but with the homogeneous Sobolev norm
\[ \|f\|_{\dot H^k(\R^d)}:=\sum_{|\alpha|=k}\|\partial^\alpha
f\|_{L^2(\R^d)}^2. \]
Similarly, we define $H^k(\B_R^d(x_0))$
as the corresponding completion of
$C^\infty(\overline{\B_R^d(x_0)})$.
If $k>\frac{d}{2}$,
we have $H^k(\R^d)\hookrightarrow
C(\R^d)$ and we denote by
$H^k_{\mathrm{rad}}(\R^d)$
and $H^k_\mathrm{rad}(\B_R^d)$
the subsets of $H^k(\R^d)$
and $H^k(\B_R^d)$, respectively, that consist of radial functions.

As usual, $A\lesssim B$ means that there exists a constant $C>0$ such
that $A\leq CB$. Possible dependencies of the implicit constant $C$ on
additional parameters follow from the context. 
We also write $A\simeq B$ if $A\lesssim B$
and $B\lesssim A$. In general, the letter $C$ is used to denote a
constant that may change its value at each occurrence. For the sake of
clarity we sometimes indicate dependencies on additional parameters by subscripts.

We follow the tradition in relativity and number the slots of functions
defined on Minkowski space $\R^{1,d}$ starting at $0$, i.e.,
$\partial_0 u(t,x)=\partial_t u(t,x)$. In general, Greek indices run
from $0$ to $d$ whereas Latin indices run from $1$ to $d$ and
Einstein's summation convention is in force. For the signature of the
Minkowski metric we use the convention that spacelike vectors have
positive lengths. 

For a linear operator $\mb L$ on a Banach space we denote by $\mc
D(\mb L)$, $\sigma(\mb L)$, and $\sigma_p(\mb L)$ its domain,
spectrum, and point spectrum, respectively. Furthermore, for
$\lambda\in \rho(\mb L):=\C\setminus \sigma(\mb L)$, we set $\mb
R_{\mb L}(\lambda):=(\lambda\mb I-\mb L)^{-1}$.
We use boldface lowercase Latin letters to denote 2-component
functions, e.g.~$\mathbf f=(f_1,f_2)$ and we also use the notation
$[\mb f]_j:=f_j$ to extract the components. 

Finally,
$e_1:=(1,0,0,\dots,0)\in \R^d$ is the first unit vector in $\R^d$,
where the dimension $d$ follows from the context.

\section{Review of the standard Cauchy theory}

\noindent The proof of Theorem \ref{thm:main} relies on the
formulation of the problem in adapted hyperboloidal coordinates. In
order to construct data on the initial hyperboloid, we employ some
elementary results on the standard Cauchy theory which are reviewed in
the following. For simplicity we restrict ourselves to spatial dimensions
$d\geq 3$. Furthermore, we only consider wave evolution to the
future starting at $t=0$. By time translation and reflection this 
is in fact already the most general situation. 

\subsection{Wave propagators}
Recall that the solution of the Cauchy problem
\begin{equation}
  \label{eq:wave}
  \left \{ \begin{array}{l}
             (\partial_t^2-\Delta_x)u(t,x)=0,\quad (t,x)\in \R^{1,d} \\
             u(0,\cdot)=f,\qquad \partial_0 u(0,\cdot)=g
           \end{array} \right . 
       \end{equation}
       for $f,g\in \mc S(\R^d)$, say, is given by
       \begin{equation}
         \label{eq:wavesol} u(t,\cdot)=\cos(t|\nabla|)f+\frac{\sin(t|\nabla|)}{|\nabla|}g, 
       \end{equation}
       where $\phi(|\nabla|)f:=\mc F_d^{-1}(\phi(|\cdot|)\mc F_d f)$
       for $\phi \in C(\R)$ and $\mc F_d$ is the Fourier transform
       \[ \mc F_df(\xi):=\int_{\R^d}e^{-i \xi x}f(x)d x .\]
       The \emph{wave propagators} $\cos(t|\nabla|)$ and
       $\frac{\sin(t|\nabla|)}{|\nabla|}$ extend by continuity to
       rough data, e.g.~$(f,g)\in \dot H^1(\R^d)\times L^2(\R^d)$.
       This yields a canonical notion of strong solutions, i.e., one
       says that $u$ solves Eq.~\eqref{eq:wave} if
       Eq.~\eqref{eq:wavesol} holds.  Note further that for any fixed
       $t\in \R$, the wave propagators map $\mc S(\R^d)$ to itself
because the symbols involved are smooth and bounded along with their
derivatives.

       \subsection{Finite speed of propagation}

       The wave equation enjoys finite speed of propagation in the
       following sense.

       \begin{proposition}
         \label{prop:finite}
         Let $x_0\in \R^d$ and $d\geq 3$. Then there exists a
         continuous function $\gamma_d: [0,\infty)\to [1,\infty)$ such
         that
         \begin{align*}
           \|\partial_t^\ell \cos(t|\nabla|)f\|_{\dot
           H^k(\B_{T-t}^d(x_0))}
           &\leq \|f\|_{\dot H^{k+\ell}(\B^d_{T}(x_0))} \\
           \left \|\partial_t^\ell \frac{\sin(t|\nabla|)}{|\nabla|}f
           \right \|_{\dot H^k(\B^d_{T-t}(x_0))}
           &\leq \|f\|_{\dot H^{k+\ell-1}(\B^d_{T}(x_0))} 
         \end{align*}
         as well as
         \begin{align*}
           \|\partial_t^\ell \cos(t|\nabla|)f\|_{L^2(\B_{T-t}^d(x_0))}
           &\leq \gamma_d(T)\|f\|_{H^{1+\ell}(\B^d_{T}(x_0))} \\
           \left \|\partial_t^\ell \frac{\sin(t|\nabla|)}{|\nabla|}f
           \right \|_{L^2(\B^d_{T-t}(x_0))}
           &\leq \gamma_d(T) \|f\|_{H^\ell(\B^d_{T}(x_0))}
         \end{align*}
         for all $f\in \mc S(\R^d)$, $T>0$, $t\in [0,T)$, $k\in \N$, and
         $\ell\in \N_0$.
       \end{proposition}

       The bounds in homogeneous Sobolev spaces $\dot H^k$ follow
       directly from the energy identity. The $L^2$-bounds are
       slightly more involved and in order to prove them, we need the
       following result which gives us control on the $L^2$-norm in
       balls in terms of the $\dot H^1$-norm and a boundary term.

       \begin{lemma}
         \label{lem:embed}
         Let $x_0\in \R^d$ and $d\geq 3$. Then we have 
         \[ \|f\|_{L^2(\B^d_R(x_0))}^2\leq R^2 \|\nabla
         f\|_{L^2(\B^d_R(x_0))}^2+\tfrac{d-1}{2}R \|f\|_{L^2(\partial
           \B^d_R(x_0))}^2 \]
         for all $f\in C^1(\overline{\B^d_R(x_0)})$ and $R>0$.
       \end{lemma}

\begin{proof}
  By translation we may assume $x_0=0$.  Introducing polar coordinates
  $r=|x|$ and $\omega=\frac{x}{|x|}$, we compute
  \begin{align*}
    r^{\frac{d-1}{2}} f(r\omega)=\int_0^r \partial_s \left [s^{\frac{d-1}{2}}f(s\omega)\right ]d s
    =\int_0^r \left [s^{\frac{d-1}{2}}\partial_s f(s\omega)+\tfrac{d-1}{2}s^{\frac{d-3}{2}}f(s\omega) \right ]d s
  \end{align*}
  and Cauchy-Schwarz yields
  \[ r^{d-1}|f(r\omega)|^2 \leq R \int_0^R \left |
    s^{\frac{d-1}{2}}\partial_s
    f(s\omega)+\tfrac{d-1}{2}s^{\frac{d-3}{2}}f(s\omega) \right |^2
  ds \] for all $r\in [0,R]$.  Expanding the square, we find
  \begin{align*}
    \tfrac{1}{R}r^{d-1}|f(r\omega)|^2&\leq \int_0^R |\partial_s f(s\omega)|^2 s^{d-1}ds
                                       +(d-1)\int_0^R \underbrace{\partial_s f(s\omega)f(s\omega)}_{\frac12 \partial_s [f(s\omega)]^2}s^{d-2}ds \\
                                     &\quad +\left (\frac{d-1}{2}\right )^2  \int_0^R f(s\omega)^2s^{d-3}d s \\
                                     &=\int_0^R |\partial_s f(s\omega)|^2 s^{d-1}ds+\tfrac{d-1}{2}R^{d-2}f(R\omega)^2 \\
                                     &\quad -\frac{(d-1)(d-3)}{4}\int_0^R f(s\omega)^2 s^{d-3}ds \\
                                     &\leq \int_0^R |\omega^j \partial_j f(s\omega)|^2 s^{d-1}ds+\tfrac{d-1}{2}R^{d-2}f(R\omega)^2. 
  \end{align*}
  Integrating this inequality yields
  \begin{align*} \int_0^R \int_{\S^{d-1}}f(r\omega)^2 d\sigma(\omega)r^{d-1}dr&\leq R^2 \int_0^R \int_{\S^{d-1}} |\nabla f(r\omega)|^2d\sigma(\omega)r^{d-1}dr \\
                                                                              &\quad +\tfrac{d-1}{2}R^d \int_{\S^{d-1}}f(R\omega)^2d\sigma(\omega) \\
                                                                              &=R^2
                                                                                \|\nabla
                                                                                f\|_{L^2(\B^d_R)}^2+\tfrac{d-1}{2}R\|f\|_{L^2(\partial\B^d_R)}^2,
  \end{align*}
  which is the claim.
\end{proof}

\begin{proof}[Proof of Proposition \ref{prop:finite}]
  Let $u(t,\cdot)=\cos(t|\nabla|)f$. Then
  $u(t,\cdot), \partial_t u(t,\cdot)\in \mc S(\R^d)$ for all
  $t\in \R$, $u\in C^\infty(\R^{1,d})$, $u(0,\cdot)=f$,
  $\partial_0 u(0,\cdot)=0$, and $(\partial_t^2 -\Delta_x)u(t,x)=0$.
  Since $\partial^\alpha u(t,\cdot)$ for any multi-index
  $\alpha\in \N_0^d$ satisfies the same equation, it is sufficient to
  consider the case $k=1$.  Furthermore, by translation invariance we
  may assume $x_0=0$.  We start with the case $\ell=0$.  A
  straightforward computation yields
  \[ \frac{d}{dt}\left [\int_{\B^d_{T-t}}\left (|\nabla_x
      u(t,x)|^2+|\partial_t u(t,x)|^2 \right )dx \right ] \leq 0, \]
cf.~the proof of Lemma \ref{lem:apxenid},
  and thus,
  \[ \|u(t,\cdot)\|_{\dot H^1(\B^d_{T-t})}^2=\|\nabla
  u(t,\cdot)\|_{L^2(\B^d_{T-t})}^2\leq \|\nabla
  u(0,\cdot)\|_{L^2(\B^d_T)}^2+\|\partial_0 u(0,\cdot)\|_{L^2(\B^d_T)}^2
  =\|f\|_{\dot H^1(\B^d_T)}^2 \] since $\partial_0 u(0,\cdot)=0$.

  For the $L^2$-bound we appeal to Lemma \ref{lem:embed} and note that
  the energy may be augmented by a boundary term that does not destroy
  the monotonicity. Indeed, we have
  \[ \frac{d}{dt}\left [\int_{\B^d_{T-t}}\left (|\nabla_x
      u(t,x)|^2+|\partial_t u(t,x)|^2 \right )dx +
    \frac{1}{T-t}\int_{\partial \B_{T-t}^d}u(t,\omega)^2 d\sigma(\omega)
  \right ] \leq 0, \] 
see Lemma \ref{lem:apxenid}.
Consequently, Lemma \ref{lem:embed} implies
  \begin{align*}
    \|u(t,\cdot)\|_{L^2(\B^d_{T-t})}^2
    &\leq (T-t)^2\left [ \|\nabla u(t,\cdot)\|_{L^2(\B^d_{T-t})}^2 
      +\tfrac{d-1}{2}(T-t)^{-1}\|u(t,\cdot)\|_{L^2(\partial\B^d_{T-t})}^2\right ] \\
    &\leq \tfrac{d-1}{2}(T-t)^2 \left [\|\nabla
      u(0,\cdot)\|_{L^2(\B^d_T)}^2
      +\|\partial_0 u(0,\cdot)\|_{L^2(\B^d_T)}^2+T^{-1}\|u(0,\cdot)\|_{L^2(\partial\B^d_T)}^2 \right ] \\
    &\leq \widetilde \gamma_d(T) \|u(0,\cdot)\|_{H^1(\B^d_T)}^2
  \end{align*}
for a continuous function $\widetilde\gamma_d: [0,\infty)\to [1,\infty)$,
  where the last step follows from the trace theorem.  For
  $\ell\geq 1$ we repeat the above arguments with $u$ replaced by
  $\partial_0^\ell u$ and use the equation to transform temporal
  derivatives into spatial ones.  The proof for the sine propagator is
  identical.
\end{proof}

\begin{remark}
  By approximation, finite speed of propagation holds for rough data
  as well.
\end{remark}

In view of Proposition \ref{prop:finite} it is natural to extend the
definition of the wave propagators to functions defined on balls only.
This is most conveniently realized by means of Sobolev extensions.

\begin{lemma}
  \label{lem:extension}
Let $d\in \N$ and $x_0\in \R^d$.
For any $r>0$ there exists a linear map $\mc E_{r,x_0,d}: L^2(\B_r^d(x_0))\to
L^2(\R^d)$ such that $\mc E_{r,x_0,d} f|_{\B_r^d(x_0)}=f$ a.e.~and $f\in
H^k(\B_r^d(x_0))$ for $k\in \N$ implies $\mc E_{r,x_0,d} f\in
H^k(\R^d)$. Furthermore, there exists a constant $C_{r,k,d}>0$ such that
\[ \|\mc E_{r,x_0,d} f\|_{H^k(\R^d)}\leq C_{r,k,d} \|f\|_{H^k(\B_r^d(x_0))} \]
for all $k\in \N_0$, $x_0\in \R^d$, and $f\in H^k(\B_r^d(x_0))$.
\end{lemma}

\begin{proof}
  From e.g.~\cite{AdaFou03} we infer the existence of an
  extension $\mc E_d: L^2(\B^d)\to L^2(\R^d)$ such that $\mc E_d
  f|_{\B^d}=f$ a.e.~and $\|\mc E_d f\|_{H^k(\R^d)}\leq C_{k,d}
  \|f\|_{H^k(\B^d)}$ for all $k\in \N_0$ and all $f\in H^k(\B^d)$.
Note further that $f\in H^d(\B^d)$ implies $\mc E_d f\in
H^d(\R^d)\hookrightarrow C(\R^d)$ by Sobolev embedding and thus, $f$
and $\mc E_d f$ may be identified with continuous functions such that
$\mc E_d f|_{\B^d}=f$.
For $f\in H^d(\B_r^d(x_0))$ we now set 
$\mc E_{r,x_0,d}f(x):=\mc E_d (f_{1/r}(\cdot+x_0/r))(\frac{x-x_0}{r})$, where
$f_\lambda(x):=f(\frac{x}{\lambda})$ for any $\lambda>0$.
By density, $\mc E_{r,x_0,d}$ extends to all of $L^2(\B_r^d(x_0))$ and
it is straightforward to verify that $\mc E_{r,x_0,d}$ has the desired properties.
\end{proof}

\begin{definition}
\label{def:wploc}
Let $T>0$, $t\in [0,T)$, and $d\in \N$, $d\geq 3$. Then we define
\[ \cos(t|\nabla|), \frac{\sin(t|\nabla|)}{|\nabla|}:
L^2(\B_T^d(x_0))\to L^2(\B_{T-t}^d(x_0)) \] by
\begin{align*}
  \cos(t|\nabla|)f&:=\left . \left (\cos(t|\nabla|)\mc E_{T,x_0,d}f
                    \right ) \right |_{\B_{T-t}^d(x_0)} \\
  \frac{\sin(t|\nabla|)}{|\nabla|}f&:=\left (\left .
\frac{\sin(t|\nabla|)}{|\nabla|}\mc E_{T,x_0,d}f\right ) \right |_{\B_{T-t}^d(x_0)},
\end{align*}
where $\mc E_{T,x_0,d}$ is a Sobolev extension as in Lemma \ref{lem:extension}.
\end{definition}

\begin{remark}
  Proposition \ref{prop:finite} implies that Definition
  \ref{def:wploc} is independent of the extension chosen and
  that the wave propagators are bounded linear maps from $H^k(\B_T^d(x_0))$ to
  $H^k(\B_{T-t}^d(x_0))$ for all $k\in \N_0$, $T>0$, $t\in [0,T)$, and $x_0\in \R^d$.
\end{remark}

\subsection{Local well-posedness of semilinear wave
  equations}\label{Sec:LWP}

Next, we turn to the local Cauchy problem for nonlinear wave equations
of the form
\begin{equation}
  \label{Eq:NLW5d}
  \left \{
    \begin{array}{l}
      \partial_t^2 u(t,\cdot)-\Delta u(t,\cdot)  = \mc N(u(t,\cdot)) \\
      u(0,\cdot) = f, \quad \partial_0 u(0,\cdot) = g
    \end{array} \right . ,
\end{equation}
where $\mc N$ is some nonlinear operator.  In fact, we are going to
restrict ourselves to the following class of \emph{admissible}
nonlinearities.

\begin{definition}
  \label{def:adm}
  Let  $k\in \N$ and $x_0\in \R^d$, $d\in \N$. 
A map $\mc N: H^k(\R^d)\to H^{k-1}_\mathrm{loc}(\R^d)$ is called \emph{$(k,x_0)$-admissible}
  iff
$\mc N(0)=0$ and for any $R\geq 1$ there exists a constant
  $C_{R,k,x_0,d}>0$ such that
  \[ \|\mc N(f)-\mc N(g)\|_{H^{k-1}(\B_r^d(x_0))}\leq
  C_{R,k,x_0,d}
     \|f-g\|_{H^k(\B_r^d(x_0))} \]
    for all $r\in [\frac12 R,R]$ and all $f,g\in H^k(\R^d)$ satisfying
\[ \|f\|_{H^k(\B_R^d(x_0))}+\|g\|_{H^k(\B_R^d(x_0))}\leq R. \]
\end{definition}

\begin{remark}
  For any $r>0$, a $(k,x_0)$-admissible nonlinearity $\mc N$
  naturally restricts to a map $\mc N_r: H^k(\B^d_r(x_0))\to
  H^{k-1}(\B^d_r(x_0))$ by 
  \[ \mc N_{r}(f):=\mc N(\mc E_{r,x_0,d} f)|_{\B^d_r(x_0)}, \]
  where $\mc E_{r,x_0,d}$ is a Sobolev extension as in Lemma
  \ref{lem:extension}.  The Lipschitz bound in Definition
  \ref{def:adm} ensures that $\mc N_r$ is independent of the extension
  chosen.  For notational convenience we will identify $\mc N$ with
  $\mc N_r$.
\end{remark}

\begin{definition}
  Let $k\in \N_0$, $T>0$, $x_0\in \R^d$, $d\in \N$, and $T'\in (0,T)$. The Banach
  space $X^k_{T,x_0}(T')$ consists of functions
  \[ u: \bigcup_{t\in [0,T']}\{t\}\times \B^d_{T-t}(x_0) \to \R \]
  such that $u(t,\cdot)\in H^k(\B^d_{T-t}(x_0))$ for each
  $t\in [0,T']$ and the map
  $t\mapsto \|u(t,\cdot)\|_{H^k(\B^d_{T-t}(x_0))}$ is continuous on
  $[0,T']$. Furthermore, we set
  \[ \|u\|_{X^k_{T,x_0}(T')}:=\max_{t\in
    [0,T']}\|u(t,\cdot)\|_{H^k(\B^d_{T-t}(x_0))}. \]
For brevity we write $X_T^k(T'):=X_{T,0}^k(T')$.
\end{definition}

Appealing to Duhamel's principle, we consider the following notion of
solutions.
\begin{definition}
  Let $k\in \N$, $T>0$, $T'\in (0, T)$, and $x_0\in \R^d$, $d\geq 3$.
  Furthermore, assume that $\mc N$ is $(k,x_0)$-admissible.  We say
  that a function
  \[ u: \bigcup_{t\in [0,T']}\{t\}\times \B^d_{T-t}(x_0)\to \R \]
  is a \emph{strong $H^k$ solution} of Eq.~\eqref{Eq:NLW5d} in the
  truncated lightcone
  $\bigcup_{t\in [0,T']}\{t\}\times \B^d_{T-t}(x_0)\subset \R^{1,d}$
  iff $u\in X^k_{T,x_0}(T')$ and
  \[
  u(t,\cdot)=\cos(t|\nabla|)u(0,\cdot)+\frac{\sin(t|\nabla|)}{|\nabla|}\partial_0 u(0,\cdot)
  +\int_0^t \frac{\sin((t-s)|\nabla|)}{|\nabla|}\mc N(u(s,\cdot))ds \]
  for all $t\in [0,T']$.  
\end{definition}

\begin{theorem}[Local existence in lightcones]
  \label{thm:LWP}
  Let $k \in \N$, $M_0, T > 0$, and $x_0\in \R^d$, $d\geq 3$.
  Furthermore, assume that $\mc N$ is $(k,x_0)$-admissible. Then
  there exists a $T'\in (0, T)$ such that for all
  $ (f,g) \in H^k(\B^d_T(x_0)) \times H^{k-1}(\B_T^d(x_0))$ satisfying
  \[ \|f\|_{H^k(\B^d_T(x_0))}+\|g\|_{H^{k-1}(\B_T^d(x_0))} \leq M_0, \]
  the initial value problem Eq.~\eqref{Eq:NLW5d} has a strong $H^k$
  solution $u_{f,g}$ in the truncated lightcone
  $\bigcup_{t\in [0,T']}\{t\}\times \B^d_{T-t}(x_0)$.  Furthermore,
  $\partial_0 u_{f,g}\in X^{k-1}_{T,x_0}(T')$ and the solution map
  \[ (f,g)\mapsto (u_{f,g},\partial_0 u_{f,g}) \]
  is Lipschitz as a function from (a ball in)
  $H^k(\B^d_T(x_0))\times H^{k-1}(\B^d_T(x_0))$ to
  $X^k_{T,x_0}(T') \times X^{k-1}_{T,x_0}(T')$.
\end{theorem}

\begin{proof}
  Without loss of generality we may assume $x_0=0$. 
We set
  $M:=2M_0 \gamma$, where
  $\gamma:=\max_{s\in [0,\frac{T}{2}]}\gamma_d(T-s)$ and $\gamma_d$
  is the continuous function from Proposition \ref{prop:finite}.
  Furthermore, for $T'\in [0,\frac{T}{2}]$ we set
  \[ Y(T'):=\{u\in X_T^k(T'): \|u\|_{X^k_T(T')}\leq
  M\} \]
  and define a map $\mc K_{f,g}$ on
  $Y(T')$ by
  \[ \mc
  K_{f,g}(u)(t):=\cos(t|\nabla|)f+\frac{\sin(t|\nabla|)}{|\nabla|}g
  +\int_0^t \frac{\sin((t-s)|\nabla|)}{|\nabla|}\mc N(u(s,\cdot))
  ds,\quad t\in [0,T']. \]
  Let $u\in Y(T')$.  From Proposition \ref{prop:finite} and
  Definition \ref{def:adm} we
  infer the existence of a constant $\alpha>0$
   such that
  \begin{align*}
    \|\mc K_{f,g}(u)(t)\|_{H^k(\B^d_{T-t})}
    &\leq \gamma\|f\|_{H^k(\B^d_T)}+\gamma\|g\|_{H^{k-1}(\B^d_T)} 
      + \gamma\int_0^t \|\mc N(u(s,\cdot))\|_{H^{k-1}(\B^d_{T-s})}ds  \\
    &\leq \frac{M}{2}+\alpha\gamma\int_0^t \|u(s,\cdot)\|_{H^k(\B^d_{T-s})} ds \\
    &\leq \frac{M}{2}+\alpha\gamma T'\|u\|_{X^k_T(T')} \\
&\leq \frac{M}{2}+\alpha\gamma T'M
  \end{align*}
  for all $t\in [0,T']$.  Consequently, by choosing $T'>0$ small
  enough, we obtain
  \[ \|\mc K_{f,g}(u)\|_{X^k_T(T')}\leq M, \]
  which means that $\mc K_{f,g}(u)\in Y(T')$ whenever
  $u\in Y(T')$.  Similarly, for $u,v\in Y(T')$, we
  infer
  \begin{align*}
    \|\mc K_{f,g}(u)(t)-\mc K_{f,g}(v)(t)\|_{H^k(\B^d_{T-t})}
    &\leq  \gamma\int_0^t \|\mc N(u(s,\cdot))-\mc N(v(s,\cdot))\|_{H^{k-1}(\B^d_{T-s})}ds \\
    &\leq \alpha\gamma\, T'  \|u-v\|_{X^k_T(T')}
  \end{align*}
  for all $t\in [0,T']$, which yields
  \[ \|\mc K_{f,g}(u)-\mc K_{f,g}(v)\|_{X^k_T(T')}\leq \tfrac12
  \|u-v\|_{X^k_T(T')} \]
  upon choosing $T'>0$ sufficiently small.  Thus, since $Y(T')$ is a
  closed subset of the Banach space $X_T^k(T')$, the contraction
  mapping principle implies the existence of a fixed point
  $u_{f,g}\in Y(T')$ of $\mc K_{f,g}$.  Furthermore, we have
  \begin{equation}
    \label{eq:lwpdtu}
    \partial_t u_{f,g}(t,\cdot)=\partial_t \cos(t|\nabla|)f+\partial_t
    \frac{\sin(t|\nabla|)}{|\nabla|}g+\int_0^t \partial_t
    \frac{\sin((t-s)|\nabla|)}{|\nabla|}\mc N(u_{f,g}(s,\cdot))ds 
  \end{equation}
and Proposition \ref{prop:finite} yields
  \begin{align*}
    \|\partial_t u_{f,g}(t,\cdot)\|_{H^{k-1}(\B^d_{T-t})}
    &\lesssim 
      \|f\|_{H^k(\B^d_T)}+\|g\|_{H^{k-1}(\B^d_T)}+\int_0^t \|\mc N(u_{f,g}(s,\cdot))\|_{H^{k-1}(\B^d_{T-s})}ds \\
    &\lesssim  M_0+\|u_{f,g}\|_{X^k_T(T')}
  \end{align*}
  for all $t\in [0,T']$, which shows
  $\partial_0 u_{f,g}\in X^{k-1}_T(T')$.

  It remains to prove the Lipschitz continuity of the solution map
  $(f,g)\mapsto u_{f,g}$.  We have
  \begin{align*}
    \|u_{f,g}(t)-u_{\widetilde f, \widetilde g}(t)\|_{H^k(\B^d_{T-t})}
    &=\|\mc
      K_{f,g}(u_{f,g})(t)-\mc  K_{\widetilde f,\widetilde
      g}(u_{\widetilde f, \widetilde g})(t)\|_{H^k(\B^d_{T-t})} \\
    &\leq  \|\mc K_{f,g}(u_{f,g})(t)-\mc
      K_{f,g}(u_{\widetilde f, \widetilde g})(t)\|_{H^k(\B^d_{T-t})} \\
    &\quad  +\|\mc K_{f,g}(u_{\widetilde f, \widetilde g})(t)-\mc K_{\widetilde
      f,\widetilde g}(u_{\widetilde f, \widetilde g})(t)\|_{H^k(\B^d_{T-t})} \\
    &\leq \tfrac12 \|u_{f,g}- u_{\widetilde f,\widetilde
      g}\|_{X^k_T(T')} \\
    &\quad +\gamma_d(T)\|f-\widetilde
      f\|_{H^k(\B^d_T)}+\gamma_d(T)\|g-\widetilde
      g\|_{H^{k-1}(\B^d_T)}
  \end{align*}
  for all $t\in [0,T']$ and thus,
  \[ \|u_{f,g}-u_{\widetilde f,\widetilde g}\|_{X^k_T(T')}\lesssim
  \|(f,g)-(\widetilde f,\widetilde g)\|_{H^k(\B_T^d)\times
    H^{k-1}(\B^d_T)}. \] 
Finally, from Eq.~\eqref{eq:lwpdtu} we infer
  \begin{align*}
    \|\partial_t u_{f,g}(t,\cdot)-\partial_t u_{\widetilde
    f,\widetilde g}(t,\cdot)\|_{H^{k-1}(\B_{T-t}^d)}
&\lesssim \|(f,g)-(\widetilde f,\widetilde g)\|_{H^k(\B_T^d)\times
  H^{k-1}(\B_T^d)} \\
&\quad +\int_0^t \|\mc N(u_{f,g}(s,\cdot))-\mc
  N(u_{\widetilde f,\widetilde g}(s,\cdot))\|_{H^{k-1}(\B^d_{T-s})}ds \\
&\lesssim \|(f,g)-(\widetilde f,\widetilde g)\|_{H^k(\B_T^d)\times
  H^{k-1}(\B_T^d)} \\
&\quad +\|u_{f,g}-u_{\widetilde f,\widetilde g}\|_{X^k_T(T')} \\
&\lesssim \|(f,g)-(\widetilde f,\widetilde g)\|_{H^k(\B_T^d)\times
  H^{k-1}(\B_T^d)}
  \end{align*}
for all $t\in [0,T']$, which finishes the proof.
\end{proof}

Finite speed of propagation is valid for nonlinear equations as
well. This is expressed by the following uniqueness result.

\begin{theorem}[Uniqueness in lightcones]
  \label{thm:uniq}
  Let $k \in \N$, $T > 0$, $T'\in [0,T)$, and $x_0\in \R^d$,
  $d\geq 3$. Furthermore, assume that $\mc N$ is
  $(k,x_0)$-admissible.  Suppose $u$ and $v$ are both strong $H^k$
  solutions of Eq.~\eqref{Eq:NLW5d} in the truncated lightcone
  $\bigcup_{t\in [0,T']}\{t\}\times \B^d_{T-t}(x_0)$ with the same
  initial data, i.e., $u(0,\cdot)=v(0,\cdot)$ and
  $\partial_0 u(0,\cdot)=\partial_0 v(0,\cdot)$.  Then $u=v$.
\end{theorem}

\begin{proof}
  We have
\[ u(t,\cdot)-v(t,\cdot)=\int_0^t
\frac{\sin((t-s)|\nabla|)}{|\nabla|}\left [\mc N(u(s,\cdot))-\mc
  N(v(s,\cdot))\right ]ds \]
and thus,
\begin{align*}
\| u(t,\cdot)-v(t,\cdot)\|_{H^k(\B^d_{T-t})}
&\lesssim \int_0^t \|\mc N(u(s,\cdot))-\mc
  N(v(s,\cdot))\|_{H^{k-1}(\B_{T-s}^d)}ds \\
&\lesssim \int_0^t \|u(s,\cdot)-v(s,\cdot)\|_{H^k(\B_{T-s}^d)}ds
\end{align*}
for all $t\in [0,T']$. Consequently, Gronwall's inequality yields
$\|u(t,\cdot)-v(t,\cdot)\|_{H^k(\B_{T-t}^d)}=0$ for all $t\in [0,T']$.
\end{proof}

\subsection{Upgrade of regularity}
Now we take a different viewpoint and \emph{assume} that we already
have a strong $H^k$ solution. We would then like to conclude that the
solution is in fact smooth, provided the data are smooth.  To this
end, we need to strengthen the assumptions on the nonlinearity. We
start with an auxiliary result which will also be useful later in a
different context.

\begin{lemma}
  \label{lem:Moser}
Let $d\in \N$, $x_0\in \R^d$, $k\in \N$, and $k>\frac{d}{2}$. Furthermore,
let $F\in C^\infty(\R\times \R^d)$, $F(0,x)=\partial_1 F(0,x)=0$ for all $x\in \R^d$,
and for $f: \R^d\to\R$ set
\[ \mc N(f)(x):=F(f(x),x). \]
Then $\mc N$ maps $H^k(\R^d)$ to $H^{k}_\mathrm{loc}(\R^d)$ and for
any $R\geq 1$
there exists a constant $C_{R,k,x_0,d}>0$ such that
\begin{align*}
  \|\mc N(f)-\mc N(g)\|_{H^k(\B_r^d(x_0))}\leq 
  C_{R,k,x_0,d}\left (\|f\|_{H^k(\B_r^d(x_0))}+\|g\|_{H^k(\B_r^d(x_0))}\right )\|f-g\|_{H^k(\B_r^d(x_0))} 
\end{align*}
for all $r\in [0,R]$ and all $f,g\in H^k(\B_r^d)$ satisfying
$\|f\|_{H^k(\B_r^d)}+\|g\|_{H^k(\B_r^d)}\leq R$.
In particular, $\mc N$ is $(k,x_0)$-admissible. 
\end{lemma}

\begin{proof}
  We assume without loss of generality that $x_0=0$ and
  note that the assumption $k>\frac{d}{2}$ implies
  $H^k(\B_r^d)\hookrightarrow C(\overline{\B_r^d})$ for any
  $r>0$. Thus, elements of $H^k(\B_r^d)$ can be identified with
  continuous functions.  Furthermore, $H^k(\R^d)$ is a Banach algebra
  and thus,
  \begin{equation}
    \begin{split}
      \label{eq:algBR}
      \|fg\|_{H^k(\B_r^d)} &=\|\mc E_r f\mc E_r g\|_{H^k(\B_r^d)} \leq
      \|\mc E_r f\mc E_r g\|_{H^k(\R^d)}\lesssim \|\mc E_r
      f\|_{H^k(\R^d)}\|\mc E_r g\|_{H^k(\R^d)} \\
      &\lesssim \|f\|_{H^k(\B_r^d)}\|g\|_{H^k(\B_r^d)}
    \end{split}
  \end{equation}
for all $r>0$ and $f,g\in H^k(\B_r^d)$, where $\mc E_r:=\mc E_{r,0,d}$
is an extension as in Lemma \ref{lem:extension}.

Now we use the
fundamental theorem of calculus to obtain the identity
  \begin{align*}
    \mc N(f)(x)-\mc N(g)(x)
    &=F(f(x),x)-F(g(x),x)=\int_0^1 \tfrac{d}{ds} F\big
      (sf(x)+(1-s)g(x),x\big )ds \\
    &=[f(x)-g(x)]\int_0^1 \partial_1 F\big
      (sf(x)+(1-s)g(x),x\big )ds \\
&=[f(x)-g(x)]\int_0^1 \mc N'\big
      (sf+(1-s)g\big )(x)ds 
  \end{align*}
for all $x\in \R^d$, where $\mc N'(f)(x):=\partial_1 F(f(x),x)$.
We claim that $\mc N'$ maps $H^k(\B_r^d)$ to itself for any $r>0$ and
that for any $R\geq 1$, there exists a continuous function $\gamma_R:
[0,\infty)\to [0,\infty)$ such that
\begin{equation}
  \label{eq:MoserN'}
  \|\mc N'(f)\|_{H^k(\B_r^d)}\leq \gamma_R(\|\mc E_r f\|_{H^k(\R^d)})\|f\|_{H^k(\B_r^d)}
\end{equation}
for all $r\in (0,R]$ and $f\in H^k(\B_r^d)$.
Assume for the moment that this is true. 
Then Eq.~\eqref{eq:algBR} and the triangle inequality yield
\begin{align*}
  \|\mc N(f)-\mc N(g)\|_{H^k(\B^d_r)}
&\lesssim \|f-g\|_{H^k(\B_r^d)}\int_0^1 \left \|
  \mc N'\big(sf+(1-s)g\big)\right \|_{H^k(\B_r^d)}ds.
\end{align*}
Furthermore,
\begin{align*}
\int_0^1 &\left \|
  \mc N'\big(sf+(1-s)g\big)\right \|_{H^k(\B_r^d)}ds \\
&\leq \int_0^1 \gamma_R\left (\|s\mc E_r f+(1-s)\mc E_r g\|_{H^k(\R^d)}\right)\|sf+(1-s)g\|_{H^k(\B_r^d)}ds \\
&\leq \left (\|f\|_{H^k(\B_r^d)}+\|g\|_{H^k(\B_r^d)}\right
  )\int_0^1 \gamma_R\left (\|s\mc E_r f+(1-s)\mc E_r g\|_{H^k(\R^d)}\right)ds
\end{align*}
for all $r\in (0,R]$. This yields the stated bound and finishes the proof.
Consequently, it remains to prove Eq.~\eqref{eq:MoserN'}.

To this end, we employ a smooth cut-off $\chi_R:\R^d\to [0,1]$
satisfying $\chi_R(x)=1$ for $|x|\leq R$ and $\chi_R(x)=0$ for
$|x|\geq 2R$. We set $F_R(u,x):=\chi_R(x)\partial_1 F(u,x)$. Then $F_R\in
C^\infty(\R\times \R^d)$ and for any compact $K\subset \R$ and any
multi-index $\alpha\in \N_0^{1+d}$, we have $\partial^\alpha F_R\in
L^\infty(K\times \R^d)$.
Furthermore, by assumption, $F_R(0,x)=0$ for all $x\in \R^d$.
Thus, by Moser's inequality, see e.g.~\cite{Rau12}, Theorem
6.4.1, $x\mapsto F_R(f(x),x)$ belongs to $H^k(\R^d)$ for any $f\in
H^k(\R^d)$ and there exists a continuous function $\widetilde\gamma_R:
[0,\infty)\to [0,\infty)$ such that
\begin{align*}
  \|\mc N'(f)\|_{H^k(\B_r^d)}
  &=\|
    F_R(\mc E_r f(\cdot),\cdot)\|_{H^k(\B_r^d)}
    \leq \|F_R(\mc E_r f(\cdot),\cdot)\|_{H^k(\R^d)} \\
&\leq \widetilde\gamma_R\left (\|\mc E_r f\|_{H^k(\R^d)}\right )\|\mc
  E_r f\|_{H^k(\R^d)} \\
&\lesssim \widetilde\gamma_R\left (\|\mc E_r f\|_{H^k(\R^d)}\right )\|f\|_{H^k(\B_r^d)},
\end{align*}
which proves Eq.~\eqref{eq:MoserN'}.
\end{proof}

\begin{theorem}[Upgrade of regularity]
  \label{thm:reg}
  Let $k\in \N$, $k>\frac{d}{2}$, $T>0$, $T'\in [0,T)$, and
  $x_0\in \R^d$, $d\geq 3$.  Furthermore, assume that the nonlinear
  operator is given by $\mc N(f)(x)=F(f(x),x)$ for a function
  $F\in C^\infty(\R\times\R^d)$ satisfying $F(0,x)=\partial_1 F(0,x)=0$ for all
  $x\in\R^d$. Suppose that $u$ is a strong $H^k$ solution of
  Eq.~\eqref{Eq:NLW5d} in the truncated lightcone
  $\bigcup_{t\in [0,T']}\{t\}\times \B^d_{T-t}(x_0)$. If
  $u(0,\cdot), \partial_0 u(0,\cdot) \in C^\infty(\overline{\B^d_T})$
  then $u\in C^\infty (\overline{\bigcup_{t\in [0,T']}\{t\}\times
    \B^d_{T-t}(x_0)})$ and $u$ is a classical solution, i.e.,
\[ (\partial_t^2 -\Delta_x)u(t,x)=F(u(t,x),x) \]
for all $t\in [0,T']$ and $x\in \overline{\B_{T-t}^d(x_0)}$.
\end{theorem}

\begin{proof}
Without loss of generality we assume $x_0=0$. 
By assumption, we have
\begin{equation}
\label{eq:usmooth}
u(t,\cdot)=\cos(t|\nabla|)u(0,\cdot)+\frac{\sin(t|\nabla|)}{|\nabla|}\partial_0
u(0,\cdot)+\int_0^t
\frac{\sin((t-s)|\nabla|)}{|\nabla|}\mc N(u(s,\cdot))ds 
\end{equation}
for all $t\in [0,T']$. Furthermore, Lemma \ref{lem:Moser} yields $\mc
N(u(t,\cdot))\in H^k(\B_{T-t}^d)$ for all $t\in [0,T']$ and
from Eq.~\eqref{eq:usmooth} we infer
\begin{align*}
  \|u(t,\cdot)\|_{H^{k+1}(\B_{T-t}^d)}
&\lesssim
  \|u(0,\cdot)\|_{H^{k+1}(\B_T^d)}+\|\partial_0 u(0,\cdot)\|_{H^k(\B_T^d)}+\int_0^t \|\mc
  N(u(s,\cdot))\|_{H^k(\B_{T-s}^d)}ds \\
&\lesssim \int_0^t \|u(s,\cdot)\|_{H^k(\B_{T-s}^d)}ds \\
&\lesssim \|u\|_{X^k_T(T')},
\end{align*}
which implies $u(t,\cdot)\in H^{k+1}(\B_{T-t}^d)$ for all
$t\in [0,T']$. Inductively, we find $u(t,\cdot)\in H^\ell(\B^d_{T-t})$
for all $t\in [0,T']$ and any $\ell\in \N_0$.  By Sobolev embedding we
therefore obtain $u(t,\cdot)\in C^\infty(\overline{\B_{T-t}^d})$.  The
same type of argument yields
$\partial_t u(t,\cdot)\in C^\infty(\overline{\B_{T-t}^d})$.
Furthermore, with $\mc E_T:=\mc E_{T,0,d}$ the extension from Lemma
\ref{lem:extension}, we infer
\begin{align*}
\partial_t^2 u(t,\cdot)
&=\partial_t^2 \cos(t|\nabla|)\mc E_T u(0,\cdot)+\partial_t^2
\frac{\sin(t|\nabla|)}{|\nabla|}\mc E_T \partial_0
u(0,\cdot) \\
&\quad +\int_0^t
\partial_t^2 \frac{\sin((t-s)|\nabla|)}{|\nabla|}\mc E_{T-s}\mc
  N(u(s,\cdot))ds 
+\mc E_{T-t} \mc N(u(t,\cdot)) \\
&=\Delta u(t,\cdot)+\mc E_{T-t}\mc N(u(t,\cdot))
\end{align*}
and thus, $\partial_t^2 u(t,x)-\Delta_x u(t,x)=F(u(t,x),x)$ for all
$t\in [0,T']$ and $x\in \overline{\B_{T-t}^d(x_0)}$. Inductively, it
follows that $u\in C^\infty (\overline{\bigcup_{t\in [0,T']}\{t\}\times
    \B^d_{T-t}(x_0)})$.
\end{proof}

\subsection{Application to the wave maps equation}
To conclude this section, we show that the above theory applies to the
wave maps equation. To this end it suffices to prove that the
nonlinearity in Eq.~\eqref{eq:mainu} satisfies the hypotheses of Lemma
\ref{lem:Moser}.

\begin{lemma}
  \label{lem:wmsmooth}
Let $F: \R\times \R^5\to \R$ be given by
\[ F(u,x):=\frac{2|x|u-\sin(2|x|u)}{|x|^3}. \]
Then $F(0,x)=\partial_1 F(0,x)=0$ for all $x\in \R^5$ and $F\in
C^\infty(\R\times \R^5)$.
\end{lemma}

\begin{proof}
  From Taylor's theorem with integral remainder,
  \[ f(t_0+t)=\sum_{n=0}^N
  \frac{f^{(n)}(t_0)}{n!}t^n+\frac{t^{N+1}}{N!}\int_0^1
  f^{(N+1)}(t_0+st)(1-s)^N ds, \] we infer
\[ \sin(2|x|u)=2|x|u-4|x|^3 u^3\int_0^1 \cos(2s|x|u)(1-s)^2
ds \]
and thus,
\[ F(u,x)=4u^3\int_0^1 \cos(2s|x|u)(1-s)^2 ds. \]
Since cosine is an even function, it follows that $F\in C^\infty(\R\times \R^5)$
and $F(0,x)=\partial_1 F(0,x)=0$ for all $x\in \R^5$ is obvious.
\end{proof}

\section{The wave equation in hyperboloidal similarity coordinates}

\noindent In this section we study the \emph{free} wave equation on
$\R^{1,d}$ in hyperboloidal similarity coordinates.  In fact, we will
focus on the dimensions $d=1$ and $d=5$ and restrict ourselves to the
radial case.

\subsection{Coordinate systems}
Throughout this paper we use three different coordinate systems on
(portions of) $\R^{1,d}$, which we consistently denote by
\begin{align*}
  (t,x)&=(t,x^1,\dots, x^d)=(x^0, x^1, \dots, x^d) \in \R^{1+d} \\
  (\tau,\xi)&=(\tau,\xi^1,\dots,\xi^d)= (\xi^0,\xi^1,\dots,\xi^d) \in \R^{1+d}  \\
  (s,y)&=(s,y^1,\dots,y^d)=(y^0,y^1,\dots ,y^d) \in \R^{1+d}.
\end{align*}
Naturally, $(t,x)$ are the standard Cartesian coordinates where the
Minkowski metric takes the form $\eta=\mathrm{diag}(-1,1,\dots,1)$.
The \emph{standard similarity coordinates} $(\tau,\xi)$ are defined by
\[ (t,x)=(T-e^{-\tau}, e^{-\tau}\xi), \]
where $T\in \R$ is a free parameter. Strictly speaking, the
coordinates $(\tau,\xi)$ depend on $T$ but we suppress this in the
notation.  We have
\begin{align*}
  \partial_\tau u(T-e^{-\tau},e^{-\tau}\xi)&=e^{-\tau}\partial_0 u(T-e^{-\tau},e^{-\tau}\xi)-e^{-\tau}\xi^j\partial_j u(T-e^{-\tau},e^{-\tau}\xi)  \\
  \partial_{\xi^j}u(T-e^{-\tau},e^{-\tau}\xi)&=e^{-\tau}\partial_j u(T-e^{-\tau},e^{-\tau}\xi)
\end{align*}
and as a consequence, the wave operator is given by
\begin{align*}
  -\partial^\mu\partial_\mu &u(T-e^{-\tau}, e^{-\tau}\xi) \\
                            &=e^{2\tau}
                              \left [ \partial_\tau^2 + 2\xi^j \partial_{\xi^j}\partial_\tau -(\delta^{jk}-\xi^j\xi^k)\partial_{\xi^j}\partial_{\xi^k}
                              +\partial_\tau+2\xi^j\partial_{\xi^j} \right ]u(T-e^{-\tau}, e^{-\tau}\xi).
\end{align*}
The coordinates $(s,y)$ are defined by
\[ (t,x)=(T+e^{-s}h(y), e^{-s}y), \]
where again $T\in \R$ is a free parameter and
\[ h(y):=\sqrt{2+|y|^2}-2 \]
is called the \emph{height function}.
There is a high degree of arbitrariness in our choice of $h$.
The only really essential property is the fact that
$\{(h(y),y)\subset \R^{1,1}: y\in \R\}$ is a Cauchy surface
that is asymptotic to a forward lightcone. The particular
$h$ we use is convenient because it leads to comparatively simple
algebraic expressions in the following.
Note also that the choice $h(y)=-1$
yields the standard similarity coordinates $(\tau,\xi)$ from
above. By the chain rule, we infer
\begin{align*}
      \partial_s u(T+e^{-s}h(y),e^{-s}y)&=-e^{-s}h(y)\partial_0
                                          u(T+e^{-s}h(y),e^{-s}y)-e^{-s}y^j\partial_ju(T+e^{-s}h(y),e^{-s}y) \\
  \partial_{y^j} u(T+e^{-s}h(y),e^{-s}y)&=e^{-s}\partial_j
                                          h(y)\partial_0
                                          u(T+e^{-s}h(y),e^{-s}y)+e^{-s}\partial_j u(T+e^{-s}h(y),e^{-s}y).
\end{align*}
For brevity we introduce the following notation for the partial
derivatives expressed in the new coordinates.
\begin{definition}
  \label{def:mcD}
  We define
  \begin{align*}
    (\Dh_0 v)(s,y)&:=\frac{e^s}{y^k\partial_k h(y)-h(y)}\left [
                    \partial_s+y^k\partial_{y^k} \right ]v(s,y) \\
    (\Dh_j v)(s,y)&:=e^s\partial_{y^j} v(s,y)-\partial_j h(y)(\Dh_0 v)(s,y)
  \end{align*}
\end{definition}
Then we have $\mc D_\mu v(s,y)=\partial_\mu u(T+e^{-s}h(y),e^{-s}y)$
and thus,
\[ \partial^\mu \partial_\mu u(T+e^{-s}h(y),e^{-s}y)=\mc D^\mu\mc
D_\mu v(s,y), \]
where $v(s,y)=u(T+e^{-s}h(y),e^{-s}y)$.  Note that by construction,
the differential operators $\mc D_\mu$ and $\mc D_\nu$ commute.  In
the case $d=1$ we have
\begin{align*}
  \Dh_0 v(s,y)&=\frac{e^s}{y h'(y)-h(y)}(\partial_s+y\partial_y)v(s,y)  \\
  \Dh_1 v(s,y)&=-\frac{e^s}{yh'(y)-h(y)}[h'(y)\partial_s+h(y)\partial_y]v(s,y).
\end{align*}

Finally, we note that there is a convenient direct relation between
the coordinates $(\tau,\xi)$ and $(s,y)$ given by
\begin{equation}
  \label{eq:sytauxi}
  (\tau,\xi)=\left (s-\log(-h(y)), -\frac{y}{h(y)} \right ).
\end{equation}
In particular, this implies the identity
\begin{equation}
  \label{eq:sytauxiwave}
  -\mc D^\mu\mc D_\mu v(s,y)=e^{2\tau}
  \left [ \partial_\tau^2 + 2\xi^j \partial_{\xi^j}\partial_\tau -(\delta^{jk}-\xi^j\xi^k)\partial_{\xi^j}\partial_{\xi^k}
    +\partial_\tau+2\xi^j\partial_{\xi^j} \right ]w(\tau,\xi),
\end{equation}
where $v(s,y)=w(s-\log(-h(y)), -y/h(y))$.

 \subsection{Control of the wave evolution}
 Let $u\in C^2(\R^{1,1})$ satisfy the wave equation
 \[ \partial_t^2 u(t,x)-\partial_x^2 u(t,x)=0. \]
 Furthermore, assume that $u(t,\cdot)$ is odd for all $t\in \R$.  In
 HSC we obtain
 \begin{equation}
   \begin{split}
     \label{eq:wavev} 0&=\Dh_0^2 v-\Dh_1^2
     v=(\Dh_0-\Dh_1)(\Dh_0+\Dh_1)v =(\Dh_0+\Dh_1)(\Dh_0-\Dh_1)v
   \end{split}
 \end{equation}
 where $v(s,y)=u(T+e^{-s}h(y),e^{-s}y)$.  If we set
 $v_\pm:=\mc D_0 v\pm \mc D_1 v$, Eq.~\eqref{eq:wavev} implies
 \begin{equation}
   \label{eq:wavevpm}
   \begin{split}
     [1-h'(y)]\partial_s v_-(s,y)&=-[y-h(y)]\partial_y v_-(s,y) \\
     [1+h'(y)]\partial_s v_+(s,y)&=-[y+h(y)]\partial_y v_+(s,y).
   \end{split}
 \end{equation}
 Note that $y-h(y)$ has a unique zero at $y=-\frac12$ and $y-h(y)<0$
 for $y<-\frac12$.  Geometrically, $y=\pm\frac12$ is the boundary of
 the backward lightcone with tip $(T,0)$.  By testing the first
 equation with $v_-$ and integrating over $[-R, R]$, we find
 \begin{align}
   \label{eq:energyv-}
   \frac{d}{ds}\int_{-R}^{R}v_-(s,y)^2 [1-h'(y)]dy &=-\int_{-R}^{R} \partial_y [v_-(s,y)^2] [y-h(y)]dy \nonumber \\
                                                   &= -v_-(s,y)^2 [y-h(y)]\Big |_{-R}^{R}
                                                     +\int_{-R}^{R} v_-(s,y)^2 [1-h'(y)]dy \nonumber \\
                                                   &\leq \int_{-R}^{R} v_-(s,y)^2 [1-h'(y)]dy,
 \end{align}
 provided $R\geq \frac12$.  Integration with respect to $s$ and
 $1-h'(y)\simeq 1$ for $y\in [-R,R]$ yield the bound
 \[ \|v_-(s,\cdot)\|_{L^2(\B_R)}\lesssim
 e^{s/2}\|v_-(s_0,\cdot)\|_{L^2(\B_R)} \]
 for all $s\geq s_0$ and any fixed $s_0$.  Analogously, we infer
 \begin{equation}
\label{eq:aprioriv+}
 \|v_+(s,\cdot)\|_{L^2(\B_R)}\lesssim
 e^{s/2}\|v_+(s_0,\cdot)\|_{L^2(\B_R)}. 
\end{equation}
 Consequently, from
 \begin{equation}
   \label{eq:vvpm}
   \begin{split}
     \partial_s v(s,y)&=\tfrac12 e^{-s}[y-h(y)]v_-(s,y)-\tfrac12 e^{-s}[y+h(y)]v_+(s,y) \\
     \partial_y v(s,y)&=-\tfrac12 e^{-s}[1-h'(y)] v_-(s,y)+\tfrac12
     e^{-s}[1+h'(y)]v_+(s,y)
   \end{split}
 \end{equation}
 we obtain the bound
 \begin{align*}
   \|v(s,\cdot)\|_{\dot H^1(\B_R)}+\|\partial_s v(s,\cdot)\|_{L^2(\B_R)}&\lesssim e^{-s}\left (\|v_-(s,\cdot)\|_{L^2(\B_R)}+\|v_+(s,\cdot)\|_{L^2(\B_R)} \right ) \\
                                                                        &\lesssim e^{-s/2}\left (\|v_-(s_0,\cdot)\|_{L^2(\B_R)}+\|v_+(s_0,\cdot)\|_{L^2(\B_R)} \right ) \\
                                                                        &\lesssim e^{-s/2}\left (\|v(s_0,\cdot)\|_{\dot H^1(\B_R)}+\|\partial_0 v(s_0,\cdot)\|_{L^2(\B_R)} \right ),
 \end{align*}
 where we have used the fact that $yh'(y)-h(y)\geq \frac12$ for all
 $y\in \R$.  Since $u(t,\cdot)$ is assumed to be odd, we have the
 boundary condition $v(s,0)=u(T+e^{-s}h(0),0)=0$ for all $s\in \R$ and
 thus,
 \[ v(s,y)=\int_0^y \partial_{y'} v(s,y')dy' \]
 which, by Cauchy-Schwarz, yields the final \emph{energy estimate}
 \[ \|v(s,\cdot)\|_{H^1(\B_R)}+\|\partial_s
 v(s,\cdot)\|_{L^2(\B_R)}\lesssim e^{-s/2}\left
   (\|v(s_0,\cdot)\|_{H^1(\B_R)}+\|\partial_0
   v(s_0,\cdot)\|_{L^2(\B_R)} \right ). \]
 We generalize to higher derivatives.

 \begin{lemma}
   \label{lem:vHk}
   Fix $s_0\in \R$, $R\geq \frac12$, $T\in \R$, and $k\in \N$, $k\geq 2$. Furthermore, assume
   that $u\in C^k(\R^{1,1})$ satisfies
   \[ \partial_t^2 u(t,x)-\partial_x^2 u(t,x)=0 \]
   and suppose $u(t,\cdot)$ is odd for all $t\in \R$. Let
   $v(s,y):=u(T+e^{-s}h(y),e^{-s}y)$.  Then we have the bounds
   \[ \| v(s,\cdot)\|_{H^\ell(\B_R)}+ \|\partial_s
   v(s,\cdot)\|_{H^{\ell-1}(\B_R)}\lesssim \,e^{-s/2}\left (\|
     v(s_0,\cdot)\|_{H^\ell(\B_R)}+\|\partial_0
     v(s_0,\cdot)\|_{H^{\ell-1}(\B_R)}\right ) \]
   for all $s\geq s_0$ and all $\ell\in \N$ satisfying $\ell\leq k$.
 \end{lemma}

\begin{proof}
  Define the differential operators
  \[ (L_\pm f)(y):=-\frac{y\pm h(y)}{1\pm h'(y)}f'(y),\qquad (D_\pm
  f)(y):=\frac{1}{1\pm h'(y)}f'(y). \]
  Then Eq.~\eqref{eq:wavevpm} can be written as
  \begin{equation}
    \label{eq:Lv} 
    \partial_s v_\pm(s,\cdot)=L_\pm v_\pm(s,\cdot). 
  \end{equation}
  We have the commutator relation $[D_\pm,L_\pm]=-D_\pm$ and thus,
  applying $D_\pm^j$ to Eq.~\eqref{eq:Lv}, for $0\leq j\leq k-1$, yields
  \[ \partial_s D_\pm^jv_\pm(s,\cdot)=D_\pm^j L_\pm
  v_\pm(s,\cdot)=L_\pm D_\pm^j
  v_\pm(s,\cdot)-jD_{\pm}^jv_\pm(s,\cdot). \]
  Consequently, by testing with $[1\pm h'(y)]D_\pm^j v_\pm(s,\cdot)$,
  we infer
  \begin{align}
    \label{eq:kbound}
    \|D_\pm^jv_\pm(s,\cdot)\|_{L^2(\B_R)}^2\lesssim e^{(1-2j) s} \|D_\pm^j v_\pm(s_0,\cdot)\|_{L^2(\B_R)}^2
  \end{align}
  for any $0\leq j\leq k-1$, cf.~\eqref{eq:aprioriv+}.
  Now we claim that, for any  $\ell\in \N_0$,
  \begin{equation}
    \label{eq:Hkequiv}
    \|f\|_{H^\ell(\B_R)}\simeq \sum_{j=0}^\ell \|D_\pm^j f\|_{L^2(\B_R)} .
  \end{equation}
  Suppose for the moment that Eq.~\eqref{eq:Hkequiv} is true. Then
  Eq.~\eqref{eq:kbound} implies
  \begin{align*}
    \|v_\pm(s,\cdot)\|_{H^{\ell-1}(\B_R)}
&\simeq \sum_{j=0}^{\ell-1}\|D_\pm^j v_\pm(s,\cdot)\|_{L^2(\B_R)}
\lesssim e^{s/2} \sum_{j=0}^{\ell-1}\|D_\pm^j v_\pm(s_0,\cdot)\|_{L^2(\B_R)} \\
&\lesssim e^{s/2}\|v_\pm(s_0,\cdot)\|_{H^{\ell-1}(\B_R)}
  \end{align*}
  for any $0\leq\ell\leq k$
  and the claim follows from Eq.~\eqref{eq:vvpm} and the boundary
  condition $v(s,0)=0$.

  It remains to prove Eq.~\eqref{eq:Hkequiv}. Note that
  $1\pm h'(y)\gtrsim 1$ for all $y\in \B_R$. Consequently, the bound
  $\|D_\pm^j f\|_{L^2(\B_R)}\lesssim \|f\|_{H^j(\B_R)}$ is trivial.
  Conversely,
  \[ D_\pm^{\ell}f(y)=\sum_{j=0}^{\ell-1}
  a_{\pm,j}(y)f^{(j)}(y)+\frac{1}{1\pm h'(y)}f^{(\ell)}(y) \]
  for functions $a_{\pm,j}\in C^\infty(\R)$ and thus,
  \[ \|f^{(\ell)}\|_{L^2(\B_R)}\lesssim
  \|D_\pm^{\ell}f\|_{L^2(\B_R)}+\sum_{j=0}^{\ell-1}\|f^{(j)}\|_{L^2(\B_R)}
  \lesssim \|D_\pm^{\ell}f\|_{L^2(\B_R)}+\|f\|_{H^{\ell-1}(\B_R)}. \]
  Consequently, the claim follows inductively.
\end{proof}

\begin{remark}
  Lemma \ref{lem:vHk} shows that the full range of energy bounds is
  available in the HSC. Even better, the evolution decays
  exponentially in these coordinates. This is a scaling effect.
\end{remark}

\subsection{Radial wave evolution in 5 space dimensions}
Let $u\in C^\infty(\R^{1,d})$ satisfy
$(\partial_t^2-\Delta)u(t,\cdot)=0$ and suppose $u(t,\cdot)$ is
radial. Then there exists a function
$\widehat u \in C^\infty(\R^{1,1})$ such that
$u(t,x)=\widehat u(t,|x|)$. In addition, $\widehat u(t,\cdot)$ is even
and satisfies
\[ (\partial_t^2-\partial_r^2-\tfrac{d-1}{r}\partial_r)\widehat
u(t,r)=0. \]
It is well-known that radial wave evolution in five space dimensions
can be reduced to the case $d=1$. 
This is a consequence of the intertwining identity\footnote{Similar
  formulas exist for all odd dimensions.}
\begin{equation}
\label{eq:ident51}
\partial_r^2 (r^2\partial_r+3r)=(r^2\partial_r+3r)(\partial_r^2+\tfrac{4}{r}\partial_r). 
\end{equation}
More precisely, let $\Omega\subset \R^{1,1}$ be a domain and suppose
$\widehat u\in C^3(\Omega)$.
Then it follows directly from Eq.~\eqref{eq:ident51} that
$(\partial_t^2-\partial_r^2-\frac{4}{r}\partial_r)\widehat u(t,r)=0$ implies
$(\partial_t^2-\partial_r^2)\widetilde u(t,r)=0$, where $\widetilde
u(t,r):=(r^2\partial_r+3r)\widehat u(t,r)$.
The converse is slightly more subtle. To begin with,
$(\partial_t^2-\partial_r^2)\widetilde u(t,r)=0$ implies that
$(\partial_t^2-\partial_r^2-\frac{4}{r}\partial_r)\widehat u(t,r)$
belongs to the kernel of $r^2\partial_r+3r$. The equation
$(r^2\partial_r+3r)U(t,r)=0$ has the general solution $U(t,r)=\frac{f(t)}{r^3}$ for a
free function $f$. Consequently, we obtain
$(\partial_t^2-\partial_r^2-\frac{4}{r}\partial_r)\widehat u(t,r)=0$ \emph{but
  only for those $t$ where $(t,0)\in \Omega$}.
This appears to cause problems for the evolution in HSC, cf.~Fig.~\ref{fig:HSC}.

In order to deal with this issue, we first
recall that
\[ \mc D_0 \widehat v(s,\eta)=e^s
h_1(\eta)(\partial_s+\eta\partial_\eta)\widehat v(s,\eta),\qquad
h_1(\eta):=\frac{1}{\eta h'(\eta)-h(\eta)} \]
and thus,
\begin{align*}
\mc D_0^2 \widehat v(s,\eta)
&=
e^{2s}h_1(\eta)^2\left [\partial_s^2+2\eta\partial_s\partial_\eta
  +\eta^2\partial_\eta^2
+\left (\eta\tfrac{h_1'(\eta)}{h_1(\eta)}+1\right )\partial_s+\left 
(\eta^2\tfrac{h_1'(\eta)}{h_1(\eta)}+2\eta\right )\partial_\eta
\right ]\widehat v(s,\eta).
\end{align*}
Similarly,
\[ \mc D_1 \widehat v(s,\eta)=-e^s
h_1(\eta)[h'(\eta)\partial_s+h(\eta)\partial_\eta]\widehat
v(s,\eta) \]
and therefore,
\begin{align*}
  \mc D_1^2 \widehat v(s,\eta)
=
e^{2s}h_1(\eta)^2 \Big [
&h'(\eta)^2\partial_s^2+2h'(\eta)h(\eta)\partial_s\partial_y+h(\eta)^2\partial_\eta^2 \\
&+\Big ( h''(\eta)h(\eta)+h'(\eta)^2+h'(\eta)h(\eta)\tfrac{h_1'(\eta)}{h_1(\eta)}\Big)\partial_s \\
&+\Big ( \tfrac{h_1'(\eta)}{h_1(\eta)}h(\eta)^2+2h'(\eta)h(\eta)\Big
  )\partial_\eta
\Big ]\widehat v(s,\eta) .
\end{align*}
Consequently, the radial, $d$-dimensional wave equation in HSC,
\[ \mc D_0^2 \widehat v(s,\eta)-\mc D_1^2 \widehat
v(s,\eta)-\frac{(d-1)e^s}{\eta}\mc D_1 \widehat v(s,\eta)=0, \]
can be written as the system
\begin{equation}
\label{eq:wavedHSCsys}
 \partial_s \begin{pmatrix}\widehat v(s,\cdot) \\ \partial_s
  \widehat v(s,\cdot) \end{pmatrix}=\widehat{\mb L}_d \begin{pmatrix}
  \widehat v(s,\cdot) \\ \partial_s \widehat v(s,\cdot) \end{pmatrix}, 
\end{equation}
with the spatial differential operator
\[ \widehat{\mb L}_d \begin{pmatrix} \widehat f_1 \\ \widehat
  f_2 \end{pmatrix}:=\begin{pmatrix} \widehat f_2 \\
c_{12}\widehat f_1''
+c^d_{11}\widehat f_1'+c_{21}\widehat f_2'
+c^d_{20}\widehat f_2 \end{pmatrix} \]
and the coefficients
\begin{align*}
c_{12}(\eta)&=\frac{h(\eta)^2-\eta^2}{1-h'(\eta)^2} \\
c^d_{11}(\eta)&=-\frac{1}{1-h'(\eta)^2}\left [
\frac{(d-1)h(\eta)}{\eta
              h_1(\eta)}+[\eta^2-h(\eta)^2]\frac{h_1'(\eta)}{h_1(\eta)}
+2[\eta-h'(\eta)h(\eta)] \right ] \\
c_{21}(\eta)&=2\frac{h'(\eta)h(\eta)-\eta}{1-h'(\eta)^2} \\
c^d_{20}(\eta)&=-\frac{1}{1-h'(\eta)^2}\left [ \frac{(d-1)h'(\eta)}{\eta
              h_1(\eta)}+[\eta-h'(\eta)h(\eta)]\frac{h_1'(\eta)}{h_1(\eta)}
-h''(\eta)h(\eta)-h'(\eta)^2+1 \right ].
\end{align*}
The equation $\widetilde u(t,r)=(r^2\partial_r+3r)\widehat u(t,r)$ in HSC reads
\begin{equation}
\begin{split}
\label{eq:intertwHSC}
 \widetilde v(s,\eta)&=e^{-2s}\eta^2 \mc D_1 \widehat
  v(s,\eta)+3e^{-s}\eta \widehat v(s,\eta) \\
&=-e^{-s}\eta^2 h_1(\eta)h'(\eta)\partial_s\widehat
  v(s,\eta)-e^{-s}\eta^2 h_1(\eta)h(\eta)\partial_\eta\widehat
  v(s,\eta)
+3e^{-s}\eta\widehat v(s,\eta) \\
&=:e^{-s}\left [
a_{11}(\eta)\partial_\eta+a_{10}(\eta)+a_{20}(\eta)\partial_s\right ]\widehat v(s,\eta).
\end{split}
\end{equation}
and differentiation with respect to $s$ yields
\begin{align*}
 \partial_s \widetilde v(s,\eta)
=&e^{-s}\left
   [-a_{11}(\eta)\partial_\eta-a_{10}(\eta)+[a_{10}(\eta)-a_{20}(\eta)]\partial_s
+a_{11}(\eta)\partial_\eta\partial_s + a_{20}(\eta)\partial_s^2\right
   ]\widehat v(s,\eta).
\end{align*}
If we assume for the moment that $\widehat v$ solves
Eq.~\eqref{eq:wavedHSCsys} with $d=5$, we may replace $\partial_s^2 \widehat
v(s,\eta)$ by lower-order derivatives in $s$. Explicitly, this yields
\begin{equation}
\begin{split}
\label{eq:intertwHSC_s}
   \partial_s \widetilde v(s,\eta)
=e^{-s}\Big [ &a_{20}(\eta)c_{12}(\eta)\partial_\eta^2
+[a_{20}(\eta)c^5_{11}(\eta)-a_{11}(\eta)]\partial_\eta
-a_{10}(\eta) \\
&+[a_{20}(\eta)c_{21}(\eta)+a_{11}(\eta)]\partial_\eta\partial_s \\
&+[a_{10}(\eta)+a_{20}(\eta)(c^5_{20}(\eta)-1)]\partial_s
\Big ]\widehat v(s,\eta).
\end{split}
\end{equation}
We combine Eqs.~\eqref{eq:intertwHSC} and \eqref{eq:intertwHSC_s} into
the single vector-valued equation
\[ \begin{pmatrix}\widetilde v(s,\cdot) \\ \partial_s \widetilde
  v(s,\cdot)
\end{pmatrix}=e^{-s}\widehat{\mb D}_5
\begin{pmatrix}\widehat v(s,\cdot) \\ \partial_s \widehat
  v(s,\cdot)
\end{pmatrix} \]
with the spatial differential operator
\[ \widehat{\mb D}_5 \begin{pmatrix}
\widehat f_1 \\ \widehat f_2 \end{pmatrix}
=\begin{pmatrix}
a_{11} \widehat f_1' + a_{10}\widehat f_1+a_{20}\widehat f_2 \\
b_{12}\widehat f_1''
+b_{11}\widehat f_1'
+b_{10}\widehat f_1
+b_{21}\widehat f_2'
+b_{20}\widehat f_2
\end{pmatrix} \]
and the coefficients
\begin{align*}
  b_{12}&=a_{20}c_{12} & b_{11}&=a_{20}c^5_{11}-a_{11} & 
b_{10}&=-a_{10} \\
 b_{21}&=a_{20}c_{21}+a_{11} & b_{20}&=a_{10}+a_{20}(c^5_{20}-1).
\end{align*}
The intertwining relation Eq.~\eqref{eq:ident51} now manifests itself
as 
\begin{equation}
  \label{eq:D5L5}
  \widehat{\mb D}_5\widehat{\mb L}_5 \widehat{\mb f}=
\widehat{\mb L}_1\widehat{\mb D}_5 \widehat{\mb f}+\widehat{\mb
  D}_5\widehat{\mb f},
\end{equation}
which may be verified by a straightforward (but, admittedly,
lengthy) computation.

     \begin{definition}
       \label{def:D}
       For $R>0$, $\eta\in [-R,R]$, and
       $\mb f \in C^\infty(\overline{\B_R^5})^2$ radial, we set
       $\mb E_1 \mb f(\eta):=\mb f(\eta e_1)$. Furthermore,
       \[ \mb D_5 \mb f:=\widehat {\mb D}_5\mb E_1 \mb f. \]
     \end{definition}

\begin{definition}
  Let $I\subset \R$ be a symmetric interval around the origin and
  $k\in \N$. Then we set
  \begin{align*}
    C^k_\pm(I)&:=\{f\in C^k(I): f(x)=\pm f(-x) \mbox{ for all }x\in I\} \\
    C^\infty_\pm(I)&:=\{f\in C^\infty(I): f(x)=\pm f(-x) \mbox{ for all }x\in I\} \\
    H^k_\pm(I)&:=\{f\in H^k(I): f(x)=\pm f(-x) \mbox{ for all }x\in I\} 
  \end{align*}
\end{definition}

The following result establishes the key mapping properties of
$\mb D_5$ which, in conjunction with Lemma \ref{lem:vHk}, yield the
desired energy bounds for $v$. The proof is rather lengthy and
therefore postponed to the appendix.

\begin{proposition}
  \label{prop:D}
  Fix $R>0$, $k\in \N$, and $k\geq 2$. Then the operator $\mb D_5$
  extends to a bijective map
  \[ \mb D_5: H^{k+1}_\mathrm{rad}(\B_R^5) \times
  H_\mathrm{rad}^{k}(\B_R^5)\to H^{k}_-(\B_R)\times H^{k-1}_-(\B_R) \]
  and we have
  \[ \|\mb D_5\mb f\|_{H^k(\B_R)\times H^{k-1}(\B_R)}\simeq \|\mb
  f\|_{H^{k+1}(\B^5_R)\times H^{k}(\B^5_R)} \]
  for all
  $\mb f\in H^{k+1}_\mathrm{rad}(\B_R^5)\times
  H^{k}_\mathrm{rad}(\B_R^5)$.
\end{proposition}

\begin{proof}
  See Section \ref{sec:proofD}.
\end{proof}

\subsection{Semigroup formulation}
 
So far we have proved a priori bounds, i.e., we have assumed that the
solution already exists. Now we turn to the proof of existence. To
this end, we employ the machinery of strongly continuous semigroups.
As before, we start with the case $d=1$.  From above we know that if
$u$ satisfies $\partial_t^2 u(t,x)-\partial_x^2 u(t,x)=0$ and
$v(s,y)=u(T+e^{-s}h(y),e^{-s}y)$, then $v_\pm:=\mc D_0v\pm \mc D_1 v$
satisfy
\[ [1\pm h'(y)]\partial_s v_\pm(s,y)=-[y\pm h(y)]\partial_y
v_\pm(s,y). \] Equivalently,
\[ \partial_s v_\pm(s,\cdot)=L_\pm v_\pm(s,\cdot) \]
with the spatial differential operator
\[ L_\pm f(y):=-\frac{y\pm h(y)}{1\pm h'(y)}f'(y). \]

\begin{proposition}
  \label{prop:gen}
  Let $R\geq \tfrac12$ and $k\in \N$. Then the operator
  $L_\pm: C^\infty(\overline{\B_R})\subset H^{k-1}(\B_R)\to
  H^{k-1}(\B_R)$
  is closable and its closure $\overline{L_\pm}$ generates a strongly
  continuous one-parameter semigroup $S_\pm$ on $H^{k-1}(\B_R)$ with
  the bound
  \[ \|S_\pm(s)f\|_{H^{k-1}(\B_R)}\lesssim
  e^{s/2}\|f\|_{H^{k-1}(\B_R)} \]
  for all $s\geq 0$ and $f\in H^{k-1}(\B_R)$.
\end{proposition}

\begin{proof}
  We define two inner products on $L^2(\B_R)$ by
  \[ (f|g)_\pm:=\int_{-R}^R f(y)\overline{g(y)}[1\pm h'(y)]dy \]
  and denote the induced norms by $\|\cdot\|_\pm$.  A straightforward
  integration by parts using $R\geq\frac12$ yields the bound
  \[ \Re (L_\pm f|f)_\pm\leq \tfrac12 \|f\|_\pm^2 \]
  for all $f\in C^\infty(\overline{\B_R})$,
  cf.~Eq.~\eqref{eq:energyv-}.  Furthermore, we set
  \[ D_\pm f(y)=\frac{1}{1\pm h'(y)}f'(y) \]
  and define an inner product
  \[ (f|g)_{\pm,k-1}:=\sum_{j=0}^{k-1} (D_\pm^j f|D_\pm^j g)_\pm \]
  with induced norm $\|\cdot\|_{\pm,k-1}$.  Recall from the proof of
  Lemma \ref{lem:vHk} that
  $\|\cdot\|_{\pm,k-1}\simeq \|\cdot\|_{H^{k-1}(\B_R)}$.  We have the
  commutator relation $[D_\pm,L_\pm]=-D_\pm$ and thus,
  \begin{align*}
    \Re(L_\pm f|f)_{\pm,k-1}&=\Re \sum_{j=0}^{k-1} (D_\pm^j L_\pm f|D_\pm^j f)_\pm=\sum_{j=0}^{k-1} \left [\Re (L_\pm D_\pm^j f|D_\pm^j f)_\pm-j(D_\pm^j f|D_\pm^j f)_\pm \right ] \\
                            &\leq \tfrac12 \sum_{j=0}^{k-1} \|D_\pm^j f\|_\pm^2 \\
                            &=\tfrac12 \|f\|_{\pm,k-1}^2
  \end{align*}
  for all $f\in C^\infty(\overline{\B_R})$.  Thus, by the
  Lumer-Phillips Theorem \cite{EngNag00} it suffices to prove that the
  range of $1-L_\pm$ is dense in $H^{k-1}(\B_R)$. In other words, we
  have to show that for each given $F\in C^\infty(\overline{\B_R})$,
  there exists an $f\in C^\infty(\overline{\B_R})$ such that
  $(1-L_\pm)f=F$.  The equation $(1-L_+)f=F$ reads
  \[ \frac{y+h(y)}{1+h'(y)}f'(y)+f(y)=F(y). \]
  An explicit solution is given by
  \begin{align*}
    f(y)&=\frac{1}{y+h(y)}\int_{1/2}^y [1+h'(t)]F(t)dt \\
        &=
          \frac{y-\frac12}{y+h(y)}\int_0^1 [1+h'(\tfrac12+t(y-\tfrac12))]F(\tfrac12+t(y-\tfrac12))dt. 
  \end{align*}
  Since $y=\frac12$ is the only zero of $y+h(y)$ and
  $1+h'(\frac12)\not=0$, it is evident that
  $f\in C^\infty(\overline{\B_R})$.  Analogously, one proves the
  density of the range of $1-L_-$ and we are done.
\end{proof}

The next lemma shows that the closure $\overline{L_\pm}$ acts as a
classical differential operator, provided the underlying Sobolev space
contains $C^1(\overline{\B_R})$.

\begin{lemma}
  \label{lem:closure}
  Let $R\geq\frac12$, $k\in \N$, $k\geq 3$, and consider the closure
  $\overline{L_\pm}$ of the operator
  $L_\pm: C^\infty(\overline{\B_R} )\subset H^{k-1}(\B_R)\to
  H^{k-1}(\B_R)$.
  Then $\mc D(\overline{L_\pm})\subset C^{k-2}(\overline{\B_R})$ and
  we have
  \[ \overline{L_\pm} f(y)=-\frac{y\pm h(y)}{1\pm h'(y)}f'(y) \]
  for all $f\in \mc D(\overline{L_\pm})$.
\end{lemma}

\begin{proof}
  Let $f\in \mc D(\overline{L_\pm})$. By definition, there exists a
  sequence $(f_n)_{n\in\N}\subset C^\infty(\overline{\B_R})$ such that
  $f_n\to f$ and $L_\pm f_n\to \overline{L_\pm}f$ in
  $H^{k-1}(\B_R)$. By Sobolev embedding we see that
  $f \in C^{k-2}(\overline{\B_R})$ and
  \begin{align*}
    \left | L_\pm f_n(y)+\frac{y\pm h(y)}{1\pm h'(y)}f'(y)\right |&\lesssim  
                                                                    \|f_n'-f'\|_{L^\infty(\B_R)}\lesssim \|f_n-f\|_{H^{k-1}(\B_R)}\to 0
  \end{align*}
  as $n\to\infty$.
\end{proof}

As a corollary, we obtain classical solutions for the half-wave
equations.

\begin{corollary}
  \label{cor:class}
  Let $R\geq \frac12$, $k\in \N$, and $k\geq 3$. Furthermore, let
  $f_\pm \in C^\infty(\overline{\B_R})$ and set
  \[ v_\pm(s,y):=S_\pm(s)f_\pm(y), \]
  where $S_\pm$ is the semigroup on $H^{k-1}(\B_R)$ from Proposition
  \ref{prop:gen}.  Then
  $v_\pm\in C^1([0,\infty)\times \overline{\B_R})$ and
  \[ (\mc D_0\mp \mc D_1)v_\pm(s,y)=0. \]
\end{corollary}

\begin{proof}
  Since
  $f_\pm \in C^\infty(\overline{\B_R})\subset \mc
  D(\overline{L_\pm})$,
  semigroup theory implies
  $\partial_s v_\pm(s,\cdot)=\overline{L_\pm}v_\pm(s,\cdot)$.
  Consequently, Lemma \ref{lem:closure} finishes the proof.
\end{proof}

Now we can easily construct a semigroup that produces a solution to
the one-dimensional wave equation in HSC.

\begin{definition}
  Let $R\geq\frac12$, $k\in \N$, and $k\geq 3$. For
  $(f_1,f_2)\in H^{k}(\B_R)\times H^{k-1}(\B_R)$ and
  $(f_-,f_+)\in H^{k-1}(\B_R)\times H^{k-1}(\B_R)$ we set
  \begin{align*} \mb A\left (\begin{array}{c} f_1 \\ f_2 \end{array}
    \right )(y):= \frac{1}{yh'(y)-h(y)}\left (
    \begin{array}{c}
      (y+h(y))f_1'(y)+(1+h'(y))f_2(y) \\
      (y-h(y))f_1'(y)+(1-h'(y))f_2(y)
    \end{array} 
    \right )
  \end{align*}
  and
  \[
  \mb B\left (\begin{array}{c}f_- \\ f_+\end{array} \right )(y):=
  \frac12 \left (\begin{array}{c}
                   -\int_0^y (1-h'(t))f_-(t)dt+\int_0^y (1+h'(t))f_+(t)dt \\
                   (y-h(y))f_-(y)-(y+h(y))f_+(y)
                 \end{array} \right ).
               \]
               Furthermore, for $s\geq 0$, we define
               $\mb S_1(s): H^k(\B_R)\times H^{k-1}(\B_R) \to
               H^k(\B_R)\times H^{k-1}(\B_R)$ by
               \[ \mb S_1(s):=e^{-s}\mb B\left (\begin{array}{cc}S_-(s) & 0 \\
                                                  0 &
                                                      S_+(s) \end{array}
                                                      \right )\mb
                                                      A, \]
                                                      where $S_\pm$
                                                      are the
                                                      semigroups on
                                                      $H^{k-1}(\B_R)$
                                                      constructed in
                                                      Proposition
                                                      \ref{prop:gen}.
                                                    \end{definition}

                                                    As the following
                                                    result shows,
                                                    $\mb S_1$ is the
                                                    solution operator
                                                    for the
                                                    one-dimensional
                                                    wave equation in
                                                    HSC with a
                                                    Dirichlet
                                                    condition at the center.

      \begin{proposition}
        \label{prop:S1}
        Let $\mb f\in C_-^\infty(\overline{\B_R})^2$ and set
        $v(s,y):=[\mb S_1(s)\mb f]_1(y)$.  Then
        $v\in C^2([0,\infty)\times \overline{\B_R})$, $v(s,\cdot)$ is
        odd for all $s\geq 0$, and we have
        $\partial_s v(s,y)=[\mb S_1(s)\mb f]_2(y)$ as well as
        \[ \mc D_0^2 v(s,y)-\mc D_1^2 v(s,y)=0 \]
for all $(s,y)\in [0,\infty)\times \overline{\B_R}$.
        Furthermore, the family $\{\mb S_1(s): s\geq 0\}$ forms a
        strongly continuous semigroup of bounded operators on
        $H_-^k(\B_R)\times H_-^{k-1}(\B_R)$ with generator
        \[ \mb L_1 =\mb B\left (\begin{array}{cc} \overline{L_-} & 0
            \\ 0 & \overline{L_+} \end{array} \right )\mb A-\mb I. \]
      \end{proposition}

\begin{proof}
We define $v_\pm$ by
\[ \begin{pmatrix}v_-(s,\cdot) \\ v_+(s,\cdot)\end{pmatrix}:=
\begin{pmatrix}S_-(s) & 0 \\ 0 & S_+(s)\end{pmatrix}\mb A\mb f. \]
From Corollary \ref{cor:class} we have $v_\pm \in C^1([0,\infty)\times
\overline{\B_R})$ and $(\mc D_0\mp \mc D_1)
v_\pm=0$. Furthermore,
\begin{align*}
  v_-(0,-y)&=\frac{1}{-yh'(-y)-h(-y)}\left
             [(-y+h(-y))f_1'(-y)+(1+h'(-y))f_2(-y)\right] \\
&=\frac{1}{yh'(y)-h(y)}\left [(-y+h(y))f_1'(y)-(1-h'(y))f_2(y)\right]
  \\
&=-v_+(0,y).
\end{align*}
Since $v_-(s,-y)$ satisfies the same equation as $v_+(s,y)$, it
follows from the a priori bound Eq.~\eqref{eq:aprioriv+} that
\[ \|v_+(s,\cdot)+v_-(s,-(\cdot))\|_{L^2(\B_R)}
\lesssim e^{s/2}\|v_+(0,\cdot)+v_-(0,-(\cdot))\|_{L^2(\B_R)}=0 \]
and thus, 
$v_-(s,-y)=-v_+(s,y)$ for all $(s,y)\in [0,\infty)\times \overline{\B_R}$.
Consequently,
the function
 $y\mapsto -(1-h'(y))v_-(s,y)+(1+h'(y))v_+(s,y)$ is even, whereas
 $y\mapsto (y-h(y))v_-(s,y)-(y+h(y))v_+(s,y)$ is odd.
This shows that $\mb S_1(s)$ maps odd functions to odd functions.
  The first statement now follows from Eq.~\eqref{eq:vvpm} (or a straightforward computation).  To prove
  the semigroup property, we first note that $\mb A\mb B=\mb I$ and
  thus,
  \begin{align*}
    \mb S_1(s+t)&=e^{-s-t}\mb B\left (\begin{array}{cc}S_-(s+t) & 0 \\
                                        0 & S_+(s+t) \end{array} \right )\mb A \\
                &=e^{-s-t}\mb B\left (\begin{array}{cc}S_-(s) & 0 \\
                                        0 & S_+(s) \end{array} \right )\mb A\mb B
                                            \left (\begin{array}{cc}S_-(t) & 0 \\
                                                     0 & S_+(t) \end{array} \right )\mb A \\
                &=\mb S_1(s)\mb S_1(t)
  \end{align*}
  for all $s,t\geq 0$.  Furthermore, it is obvious that
  $s\mapsto \mb S_1(s)$ is strongly continuous. 
Finally, $\mb S_1(0)\mb f=\mb B\mb A \mb f=\mb f$ since $\mb f$ is
odd.
The statement about the generator is obvious.
\end{proof}

By conjugating with $\mb D_5$, we obtain the solution operator for the
5-dimensional wave equation. This leads to the main result of this
section.

\begin{definition}
  \label{def:S5}
  For $s\geq 0$ we define
  $\mb S_5(s): H^{k+1}_\mathrm{rad}(\B_R^5)\times
  H^k_\mathrm{rad}(\B_R^5) \to H^{k+1}_\mathrm{rad}(\B_R^5)\times
  H^k_\mathrm{rad}(\B_R^5)$ by
  \[ \mb S_5(s):=e^s\mb D_5^{-1}\mb S_1(s)\mb D_5. \]
\end{definition}

\begin{theorem}
  \label{thm:S5}
  The family $\{\mb S_5(s): s\geq 0\}$ forms a strongly continuous
  semigroup of bounded operators on
  $H^{k+1}_\mathrm{rad}(\B_R^5)\times H^k_{\mathrm{rad}}(\B_R^5)$ and
  we have
  \[ \|\mb S_5(s)\mb f\|_{H^{k+1}(\B_R^5)\times H^k(\B_R^5)}\lesssim
  e^{s/2} \|\mb f\|_{H^{k+1}(\B_R^5)\times H^k(\B_R^5)} \]
  for all $s\geq 0$ and
  $\mb f\in H^{k+1}_\mathrm{rad}(\B_R^5)\times
  H^k_\mathrm{rad}(\B_R^5)$.
  The generator $\mb L_5$ of $\mb S_5$ is given by
  \[ \mb L_5=\mb D_5^{-1}\mb B\left (\begin{array}{cc}\overline{L_-} & 0 \\
                                       0 & \overline{L_+}\end{array}
                                           \right )\mb A\mb D_5. \]
                                           Furthermore, the function
                                           $v(s,\cdot)=[\mb S_5(s) \mb
                                           f]_1$
                                           belongs to
                                           $C^2([0,\infty)\times
                                           \overline{\B_R^5})$
                                           and satisfies
                                           \[ \mc D_0^2 v-\mc D^j\mc
                                           D_j v=0. \]
                                           Finally,
                                           $\partial_s v(s,\cdot)=[\mb
                                           S_5(s)\mb f]_2$.
                                         \end{theorem}

                                         Finally, we obtain the
                                         explicit form of $\mb L_5$.
                                         To keep equations within
                                         margins, we define the
                                         following auxiliary
                                         quantities.

                                         \begin{definition}
                                           \label{def:H}
                                           We set
                                           \begin{align*}
                                             H_0{}^0(s,y)&:=\frac{e^s}{y^\ell\partial_\ell h(y)-h(y)} & H_0{}^j(s,y)&:=\frac{e^s y^j}{y^\ell\partial_\ell h(y)-h(y)} \\
                                             H_j{}^0(s,y)&:=-\frac{e^s\partial_j h(y)}{y^\ell\partial_\ell h(y)-h(y)} &
                                                                                                                        H_j{}^k(s,y)&:=e^s\delta_j{}^k-\frac{e^s\partial_j h(y)}{y^\ell\partial_\ell h(y)-h(y)}y^k
                                           \end{align*}
                                         \end{definition}
                                         Then we have
                                         $\mc
                                         D_\mu=H_\mu{}^\nu \partial_\nu$,
                                         see Definition \ref{def:mcD},
                                         and thus,
                                         \begin{equation}
                                           \begin{split}
                                             \label{eq:DmuDmu}
                                             \mc D^\mu\mc
                                             D_\mu&=H^{\mu\nu}\partial_\nu(H_\mu{}^\lambda\partial_\lambda)
                                             =H^{\mu\nu}H_\mu{}^\lambda\partial_\nu\partial_\lambda+H^{\mu\nu}\partial_\nu H_\mu{}^\lambda\partial_\lambda \\
                                             &=H^{\mu
                                               0}H_\mu{}^0\partial_0^2+2
                                             H^{\mu
                                               j}H_\mu{}^0\partial_j\partial_0
                                             +H^{\mu\nu}\partial_\nu
                                             H_\mu{}^0\partial_0
                                             +H^{\mu
                                               j}H_\mu{}^k \partial_j\partial_k+H^{\mu\nu}\partial_\nu
                                             H_\mu{}^j \partial_j.
                                           \end{split}
                                         \end{equation}
                                         It follows that
                                         $\mc D^\mu\mc D_\mu v=0$ is
                                         equivalent to
                                         \[ \partial_0 \left
                                           (\begin{array}{c}v
                                             \\ \partial_0
                                             v \end{array} \right )
                                         =\left (\begin{array}{c}\partial_0 v \\
                                                   -\frac{H^{\mu
                                                   j}H_\mu{}^k}{H^{\mu
                                                   0}H_\mu{}^0}\partial_j\partial_k
                                                   v-\frac{H^{\mu\nu}\partial_\nu
                                                   H_\mu{}^j}{H^{\mu
                                                   0}H_\mu{}^0}\partial_j
                                                   v -2\frac{H^{\mu
                                                   j}H_\mu{}^0}{H^{\mu
                                                   0}H_\mu{}^0}\partial_j \partial_0
                                                   v-\frac{H^{\mu\nu}\partial_\nu
                                                   H_\mu{}^0}{H^{\mu
                                                   0}H_\mu{}^0} \partial_0
                                                   v \end{array}
  \right )\] and thus,
  \begin{equation}
    \label{eq:L5}
    \mb L_5 \left (\begin{array}{c} f_1 \\ f_2 \end{array} \right )=
    \left (\begin{array}{c}
             f_2 \\ 
             -\frac{H^{\mu j}H_\mu{}^k}{H^{\mu 0}H_\mu{}^0}\partial_j\partial_k f_1-\frac{H^{\mu\nu}\partial_\nu H_\mu{}^j}{H^{\mu 0}H_\mu{}^0}\partial_j f_1
             -2\frac{H^{\mu j}H_\mu{}^0}{H^{\mu 0}H_\mu{}^0}\partial_j f_2-\frac{H^{\mu\nu}\partial_\nu H_\mu{}^0}{H^{\mu 0}H_\mu{}^0} f_2 \end{array} \right ).
       \end{equation}

       \section{Wave maps in hyperboloidal similarity coordinates}

       \noindent Now we return to the wave maps equation
       \begin{equation}
         \label{eq:mainuu}
         \left (\partial_t^2-\Delta_x\right )u(t,x)=\frac{2|x|u(t,x)-\sin(2|x|u(t,x))}{|x|^3}
       \end{equation}
       with the one-parameter family $\{u_T^*: T\in \R\}$ of blowup
       solutions given by
       \[ u_T^*(t,x)=\frac{4}{|x|}\arctan\left
         (\frac{|x|}{T-t+\sqrt{(T-t)^2+|x|^2}}\right ). \]

       \subsection{Perturbations of the blowup solution}
       \label{sec:pert}
       We would like to study the stability of $u^*_{T}$ and thus, we
       insert the ansatz $u(t,x)=u^*_{T}(t,x)+\widetilde u(t,x)$ into
       Eq.~\eqref{eq:mainuu} which yields
       \begin{equation}
         \label{eq:mainutilde}
         \left [\partial_t^2-\Delta_x+V_T(t,x)\right ]\widetilde
         u(t,x)=F_T(\widetilde u(t,x),t,x),
       \end{equation}
       where
       \begin{align*}
         V_T(t,x)&=\frac{2\cos(2|x|u_T^*(t,x))-2}{|x|^2}  
       \end{align*}
       and
       \begin{align*}
         F_T(\widetilde u(t,x),t,x)=-|x|^{-3}\Big [ 
         &\sin\big (2|x|u_T^*(t,x)+2|x|\widetilde u(t,x)\big ) \\
         &-\sin(2|x|u_T^*(t,x))-2|x|\cos(2|x|u_T^*(t,x))\widetilde u(t,x)
           \Big ].
       \end{align*}

       In hyperboloidal similarity coordinates,
       Eq.~\eqref{eq:mainutilde} reads
       \begin{equation}
         \label{eq:mainv}
         -\mc D^\mu\mc D_\mu v(s,y)+V_T(\eta_T(s,y))v(s,y)=F_T(v(s,y),\eta_T(s,y)),
       \end{equation}
       where, as always, $\eta_T(s,y)=(T+e^{-s}h(y), e^{-s}y)$ and
       $v(s,y)=\widetilde u(\eta_T(s,y))$.  By definition of
       $\mb L_5$, Eq.~\eqref{eq:mainv} is equivalent to
       \begin{equation}
         \begin{split}
           \label{eq:vsys}
           \partial_s \begin{pmatrix}
             v(s,y) \\
             \partial_s v(s,y)
           \end{pmatrix}
           =& \mb L_5\begin{pmatrix}
             \ v(s,\cdot) \\
             \partial_s v(s,\cdot)
           \end{pmatrix}(y) +\begin{pmatrix} 0 \\ (H^{\mu
               0}(s,y)H_\mu{}^0(s,y))^{-1}V_T(\eta_T(s,y)) v(s,y)
           \end{pmatrix} \\
           &-\begin{pmatrix}
             0 \\
             (H^{\mu 0}(s,y)H_\mu{}^0(s,y))^{-1} F_T(
             v(s,y),\eta_T(s,y))
           \end{pmatrix}.
         \end{split}
       \end{equation}
       Note that
       \[ H^{\mu 0}(s,y)H_\mu{}^0
       (s,y)=-e^{2s}\frac{1-\partial^jh(y)\partial_j
         h(y)}{[y^\ell \partial_\ell h(y)-h(y)]^2}=:e^{2s}H(y)^{-1} \]
       and
       \[ u_T^*(\eta_T(s,y))=\frac{4e^s}{|y|}\arctan\left
         (\frac{|y|}{\sqrt{|y|^2+h(y)^2}-h(y)}\right
       )=:e^s\alpha_0(y). \]
       Observe that $\alpha_0\in C^\infty(\R^5)$.  Consequently,
       \[ V(y):=(H^{\mu 0}(s,y)H_\mu{}^0(s,y))^{-1}V_T(\eta_T(s,y))=
       H(y)\frac{2\cos(2|y|\alpha_0(y))-2}{|y|^2} \]
       is independent of $s$.  Furthermore,
       \begin{align*}
         F_T(v(s,y),\eta_T(s,y))
         =-e^{3s}|y|^{-3}\Big [ 
         &\sin\big (2e^{-s}|y|u_T^*(\eta_T(s,y))+2e^{-s}|y|v(s,y)\big ) \\
         &-\sin\big(2e^{-s}|y|u_T^*(\eta_T(s,y))\big ) \\
         &-2e^{-s}|y|\cos \big (2e^{-s}|y|u_T^*(\eta_T(s,y))\big )v(s,y)
           \Big ]
       \end{align*}
       and we can write
       \[(H^{\mu 0}(s,y)H_\mu{}^0(s,y))^{-1} F_T(
       v(s,y),\eta_T(s,y))=e^s H(y)N(e^{-s}|y|v(s,y),y), \] where
       \begin{align*}
         N(p,y):=-\frac{ 
         \sin(2|y|\alpha_0(y)+2p)
         -\sin(2|y|\alpha_0(y))-2\cos(2|y|\alpha_0(y))p}{|y|^3}.
       \end{align*}
       In order to obtain an autonomous equation, we rescale and write
       Eq.~\eqref{eq:vsys} in terms of
       \[ \Phi(s)(y):=\begin{pmatrix} \phi_1(s)(y) \\
         \phi_2(s)(y)\end{pmatrix}:=e^{-s}\begin{pmatrix} v(s,y)
         \\ \partial_s v(s,y) \end{pmatrix}. \] This yields
       \begin{equation}
         \label{eq:Psi}
         \partial_s \Phi(s)=(\mb L_5-\mb I+\mb L')\Phi(s)+\mb N(\Phi(s)), 
       \end{equation}
       where
       \begin{align*}
         \mb L' \begin{pmatrix}
           f_1 \\ f_2 \end{pmatrix}(y)&:=\begin{pmatrix} 0 \\ V(y) f_1(y) \end{pmatrix}  \\
         \mb N \left (\begin{array}{c}
                        f_1 \\ f_2 \end{array} \right )(y)&:=\left (\begin{array}{c} 0 \\ -H(y) N(|y|f_1(y),y) \end{array} \right ).
       \end{align*}
       In the following, we write $\mb L:=\mb L_5-\mb I+\mb L'$.

\subsection{Analysis of the linearized evolution}
The rest of this section is devoted to the analysis of
Eq.~\eqref{eq:Psi}.  The first step is to develop a sufficiently good
understanding of the linearized equation that is obtained from
Eq.~\eqref{eq:Psi} by dropping the nonlinearity.  We start with a
simple lemma that constructs a semigroup $\mb S$ which governs the
linearized flow. In particular, this yields the well-posedness of the
linearized Cauchy problem in the sense of semigroup theory.

\begin{definition}
  For $R>0$ and $k\in \N_0$ we set
  \[ \mc H^{k}_R:=H^{k+1}_\mathrm{rad}(\B_R^5)\times
  H^{k}_\mathrm{rad}(\B_R^5) \] and
  \[ \|(f_1,f_2)\|_{\mc
    H^{k}_R}^2:=\|f_1\|_{H^{k+1}(\B_R^5)}^2+\|f_2\|_{H^{k}(\B_R^5)}^2 \]
\end{definition}

\begin{lemma}
  \label{lem:S}
  Let $R\geq\frac12$, $k\in \N$, and $k\geq 3$.  Then the operator
  $\mb L=\mb L_5-\mb I+\mb L'$ is the generator of a strongly
  continuous semigroup $\{\mb S(s): s\geq 0\}$ on $\mc H^k_R$.
  Furthermore, every $\lambda\in \sigma(\mb L)$ with
  $\Re\lambda>-\frac12$ is an eigenvalue with finite algebraic
  multiplicity.
\end{lemma}

\begin{proof}
 Since $H^{k+1}_\mathrm{rad}(\B_R^5)\hookrightarrow
 H^k_{\mathrm{rad}}(\B_R^5)$ is compact and $V\in C^\infty(\R^5)$, it
 follows that $\mb L': \mc H_R^k\to\mc H_R^k$ is compact and the
 bounded perturbation theorem implies that $\mb L_5-\mb I+\mb L'$
 generates a semigroup $\mb S(s)$ on $\mc H_R^k$.
Now suppose $\lambda\in \sigma(\mb L)$ and
$\Re\lambda>-\frac12$. Since $\sigma(\mb L_5-\mb I)\subset \{z\in \C:
\Re z\leq -\frac12\}$ by the growth bound in Theorem \ref{thm:S5},
we have the identity $\lambda \mb I-\mb L=[\mb I-\mb L'\mb
R_{\mb L_5-\mb I}(\lambda)](\lambda\mb I-\mb L_5+\mb
I)$. Consequently, $1\in \sigma(\mb L'\mb R_{\mb L_5-\mb I}(\lambda))$ and by
the compactness of $\mb L'$ we see that in fact $1\in \sigma_p(\mb
L'\mb R_{\mb L_5-\mb I}(\lambda))$.
This means that there exists a nonzero $\mb g\in \mc H_R^k$ in the
kernel of $\mb I-\mb L'\mb R_{\mb L_5-\mb I}(\lambda)$. Thus, $\mb
f:=\mb R_{\mb L_5-\mb I}(\lambda)\mb g$ is nonzero, belongs to $\mc
D(\mb L)$, and satisfies
\[ (\lambda\mb I-\mb L)\mb f= 
[\mb I-\mb L'\mb
R_{\mb L_5-\mb I}(\lambda)](\lambda\mb I-\mb L_5+\mb
I)\mb R_{\mb L_5-\mb I}(\lambda)\mb g=[\mb I-\mb L'\mb
R_{\mb L_5-\mb I}(\lambda)]\mb g=\mb 0.
 \]
 In other words, $\mb f$ is an eigenfunction of $\mb L$ to the
 eigenvalue $\lambda\in \sigma_p(\mb L)$.  Finally, suppose that
 $\lambda$ has infinite algebraic multiplicity. Then, by \cite{Kat95},
 p.~239, Theorem 5.28, $\lambda$ would belong to the essential
 spectrum of $\mb L$. This, however, is impossible since
 $\lambda\notin \sigma(\mb L_5-\mb I)$ and the essential spectrum is
 stable under compact perturbations, see \cite{Kat95}, p.~244, Theorem
 5.35.
\end{proof}

\subsection{Spectral analysis of the generator}

Next, we turn to the analysis of the point spectrum of $\mb L$.  As a
matter of fact, the spectral analysis of $\mb L$ is essentially
independent of the particular choice of the height function $h$ and
can be reduced to the case $h(y)=-1$.  This will allow us to utilize
the spectral information from \cite{CosDonXia16, CosDonGlo17} to show
that the only unstable eigenvalue of $\mb L$ is $\lambda=1$.

\begin{definition}
  \label{def:f1}
  We set
  \[ \mb f_1^*(y):=\begin{pmatrix}f_{1,1}^*(y) \\
    f_{1,2}^*(y)\end{pmatrix} :=\frac{1}{|y|^2+h(y)^2}\begin{pmatrix}1
    \\ 2\end{pmatrix}. \]
\end{definition}

\begin{lemma}
  \label{lem:sigmapL}
  Let $R\geq \frac12$, $k\in \N$, and $k\geq 4$. Furthermore, let
  $\mb L: \mc D(\mb L)\subset \mc H^k_R\to\mc H^k_R$ be the operator
  defined in Lemma \ref{lem:S}.  Then
  $\ker(\mb I-\mb L)=\langle \mb f_1^*\rangle$. Moreover, if
  $\lambda\in \sigma(\mb L)$ and $\Re\lambda\geq 0$, then $\lambda=1$.
\end{lemma}

\begin{proof}
  Obviously, $\mb f_1^*\in C^\infty(\overline{\B_R^5})^2$ and thus,
  $\mb f_1^*\in \mc D(\mb L)$.  The blowup solution $u_T^*$ satisfies
  \[
  (\partial_t^2-\Delta_x)u_T^*(t,x)=\frac{2|x|u_T^*(t,x)-\sin(2|x|u_T^*(t,x))}{|x|^3} \]
  and differentiating this equation with respect to $T$ yields
  \[ (\partial_t^2-\Delta_x)\partial_T
  u_T^*(t,x)=-\frac{2\cos(2|x|u_T^*(t,x))-2}{|x|^2}\partial_T
  u_T^*(t,x).
  \]
  A straightforward computation yields
  \[ \partial_T u_T^*(t,x)=\frac{4}{|x|}\partial_T \arctan\left
    (\frac{|x|}{T-t+\sqrt{(T-t)^2+|x|^2}}\right )
  =-\frac{2}{(T-t)^2+|x|^2} \] and thus,
  \[ (\partial_T
  u_T^*)(T+e^{-s}h(y),e^{-s}y)=-\frac{2e^{2s}}{|y|^2+h(y)^2}.
  \]
  Consequently, $\partial_s (e^s \mb f_1^*)=\mb L (e^s \mb f_1^*)$,
  which is equivalent to $(\mb I-\mb L)\mb f_1^*=\mb 0$ and thus,
  $\langle \mb f_1^*\rangle \subset \ker(\mb I-\mb L)$. The reverse
  inclusion is a simple consequence of basic ODE theory since we
  restrict ourselves to radial functions.

  Suppose now that $\lambda\in \sigma(\mb L)$ and $\Re\lambda\geq 0$.
  By Lemma \ref{lem:S} it follows that $\lambda\in \sigma_p(\mb L)$
  and thus, there exists a nontrivial
  $\mb f=(f_1,f_2)\in \mc D(\mb L)$ such that
  $(\lambda\mb I-\mb L)\mb f=\mb 0$.  Equivalently,
  $\partial_s (e^{\lambda s} \mb f)=\mb L(e^{\lambda s} \mb f)$ or
  \[ \partial_s (e^{(\lambda+1)s}\mb f)=(\mb L_5+\mb
  L')(e^{(\lambda+1)s}\mb f). \]
  By Sobolev embedding, the function $v(s,y):=e^{(\lambda+1)s}f_1(y)$
  belongs to $C^2(\R\times \B^5_{1/2})$ and by definition of $\mb L_5$
  and $\mb L'$, $v$ satisfies
  \begin{equation}
    \label{eq:sigmaPLv}
    -\mc D^\mu \mc D_\mu v(s,y)+V_T(\eta_T(s,y))v(s,y)=0 
  \end{equation}
  for all $(s,y)\in \R\times \B^5_{1/2}$.  Note that $v$ is nontrivial
  since the first component of $(\lambda\mb I-\mb L)\mb f=\mb 0$ reads
  $\lambda f_1-f_2+f_1=0$.  Now recall that
  \[ V_T(\eta_T(s,y))=e^{2s}\frac{2\cos(2|y|\alpha_0(y))-2}{|y|^2} \]
  and, since $h(y)<0$ for all $y\in \overline{\B_{1/2}^5}$, we can
  write
  \[ \alpha_0(y)=\frac{4}{|y|}\arctan\left
    (\frac{|y|}{\sqrt{|y|^2+h(y)^2}-h(y)}\right
  )=\frac{4}{|y|}\arctan\left
    (\frac{-|y|/h(y)}{1+\sqrt{1+|y|^2/h(y)^2}}\right ). \]
  Consequently,
  \[
  V_T(\eta_T(s,y))=e^{2s}h(y)^{-2}\frac{2\cos(2|y|\alpha_0(y))-2}{|y|^2/h(y)^2}=e^{2s}h(y)^{-2}V_0(y/h(y)) \]
  with
  \begin{equation}
    \label{def:V0}
    V_0(\xi)=\frac{2}{|\xi|^2}\left [\cos\left (8\arctan\left
          (\frac{|\xi|}{1+\sqrt{1+|\xi|^2}}\right)\right)-1\right
    ]=-\frac{16}{(1+|\xi|^2)^2}. 
  \end{equation}
  Therefore, by setting $v(s,y):=w(s-\log(-h(y)),-y/h(y))$,
  Eq.~\eqref{eq:sigmaPLv} transforms into
  \[ \left [
    \partial_\tau^2+2\xi^j\partial_{\xi^j}\partial_\tau-(\delta^{jk}-\xi^j\xi^k)\partial_{\xi^j}\partial_{\xi^k}
    +\partial_\tau+2\xi^j\partial_{\xi^j}+V_0(\xi)\right
  ]w(\tau,\xi)=0 \]
  for all $(\tau,\xi)\in \R\times \B^5$, see
  Eq.~\eqref{eq:sytauxiwave}.  Explicitly, we have
  \begin{align*}
    w(\tau,\xi)&=v\left (\tau+\log\left (\frac{2}{2+\sqrt{2(1+|\xi|^2)}}\right ), \frac{2\xi}{2+\sqrt{2(1+|\xi|^2)}}\right ) \\
               &=e^{(\lambda+1)\tau} \left (\frac{2}{2+\sqrt{2(1+|\xi|^2)}}\right )^{\lambda+1}f_1\left (\frac{2\xi}{2+\sqrt{2(1+|\xi|^2)}}\right ) \\
               &=:e^{(\lambda+1)\tau}f(\xi)
  \end{align*}
  and thus, $f$ satisfies
  \begin{equation}
    \label{eq:specf}
    \left [
      -(\delta^{jk}-\xi^j\xi^k)\partial_{\xi^j}\partial_{\xi^k}+2(\lambda+2)\xi^j\partial_{\xi^j}
      +(\lambda+1)(\lambda+2)+V_0(\xi)\right ]f(\xi)=0 
  \end{equation}
  for all $\xi\in \B^5$.  Note that $f\in H^5(\B^5)$ and thus, by
  Sobolev embedding, $f\in C^2(\overline{\B^5})$.  Furthermore, since
  $f$ is radial, we may write $f(\xi)=\widehat f(|\xi|)/|\xi|$ for a
  nontrivial odd function $\widehat f\in C^2([0,1])$.  In terms of
  $\widehat f$, Eq.~\eqref{eq:specf} reads
  \[ -(1-\rho^2)\widehat f''(\rho)-\frac{2}{\rho}\widehat
  f'(\rho)+2(\lambda+1)\rho \widehat f'(\rho) +\lambda(\lambda+1)\widehat
  f(\rho)+\frac{2(1-6\rho^2+\rho^4)}{\rho^2(1+\rho^2)^2}\widehat
  f(\rho)=0 \]
  for $\rho\in (0,1)$.  Frobenius' method yields
  $\widehat f\in C^\infty([0,1])$ and thus, by \cite{CosDonXia16,
    CosDonGlo17}, we conclude that $\lambda=1$.
\end{proof}

\begin{remark}
  The proof of Lemma \ref{lem:sigmapL} shows that the existence of the
  eigenvalue $\lambda=1$ is a mere consequence of the time translation
  symmetry of the wave maps equation \eqref{eq:main}. Consequently,
  this instability of the linearized flow does \emph{not} indicate an
  instability of the blowup profile.
\end{remark}

By Lemma \ref{lem:sigmapL}, the eigenvalue $1\in \sigma_p(\mb L)$ is
isolated. This allows us to define the corresponding spectral
projection.

\begin{definition}
  Fix $R\geq\frac12$, $k\in \N$, $k\geq 4$, and let
  $\mb L: \mc D(\mb L)\subset \mc H^k_R\to\mc H^k_R$ be the operator
  from Lemma \ref{lem:S}.  Furthermore, let $\gamma: [0,2\pi]\to \C$
  be given by $\gamma(t)=1+\frac12 e^{i t}$. Then we set
  \[ \mb P:=\frac{1}{2\pi i}\int_\gamma \mb R_{\mb
    L}(\lambda)d\lambda. \]
\end{definition}

\begin{proposition}
  \label{prop:P}
  The projection $\mb P$ commutes with the semigroup $\mb S(s)$ and we
  have
  \[ \rg \mb P=\langle \mb f_1^*\rangle. \]
\end{proposition}

\begin{proof}
  The fact that $\mb P$ commutes with $\mb S(s)$ follows from the
  abstract theory, see e.g.~\cite{Kat95, EngNag00}.  To prove the
  statement about $\rg \mb P$, we first recall from Lemma \ref{lem:S}
  that $\rg \mb P\subset \mc D(\mb L)$ is finite-dimensional.
  Consequently, the part $\mb L_{\rg \mb P}$ of $\mb L$ in $\rg \mb P$
  is an operator acting on a finite-dimensional Hilbert space with
  spectrum $\sigma(\mb L_{\rg\mb P})=\{1\}$.  This implies that
  $\mb I-\mb L_{\rg\mb P}$ is nilpotent. Thus, there exists an
  $\ell\in \N$ such that $(\mb I-\mb L_{\rg\mb P})^\ell=\mb 0$.  We
  claim that $\mb I-\mb L_{\rg\mb P}=\mb 0$.  Suppose this were not
  true, i.e., $\mb I-\mb L_{\rg \mb P}\not= \mb 0$. Then, by Lemma
  \ref{lem:sigmapL},
  \[ \rg (\mb I-\mb L_{\rg\mb P})\subset \ker(\mb I-\mb L_{\rg\mb
    P})\subset \ker(\mb I-\mb L)=\langle \mb f_1^* \rangle \]
  and thus, there exists an
  $\mb f=(f_1,f_2)\in \rg\mb P \subset H^5_\mathrm{rad}(\B^5_R)\times
  H_\mathrm{rad}^4(\B^5_R)\subset C^2(\overline{\B^5_{1/2}})\times
  C^1(\overline{\B^5_{1/2}})$ such that
  \[ \mb f_1^*=(\mb I-\mb L_{\rg\mb P})\mb f=(\mb I-\mb L)\mb f=(2\mb
  I-\mb L_5-\mb L')\mb f. \]
  From the explicit form of $\mb L_5$ in Eq.~\eqref{eq:L5} we infer
  $f_{1,1}^*=2f_1-f_2$ and
  \begin{align*}
    H^{\mu 0}H_\mu{}^0(s,y)f_{1,2}^*(y)=&H^{\mu j}H_\mu{}^k(s,y)\partial_j\partial_k f_1(y)
                                          +H^{\mu\nu}\partial_\nu H_\mu{}^j\partial_j f_1(y) \\
                                        &+2H^{\mu j}H_\mu{}^0(s,y)\partial_j f_2(y)
                                          +[2H^{\mu
                                          0}H_{\mu}{}^0(s,y)+H^{\mu\nu}\partial_\nu
                                          H_\mu{}^0(s,y)]
                                          f_2(y) \\
                                        &-e^{2s}h(y)^{-2}V_0(y/h(y))f_1(y)
  \end{align*}
  for all $(s,y)\in \R\times \B^5_{1/2}$. The potential $V_0$ is given
  in Eq.~\eqref{def:V0}.  Consequently,
  \begin{equation}
    \begin{split}
      \label{eq:Pf1}
      e^{-2s}&H^{\mu j}H_\mu{}^k(s,y)\partial_j\partial_k f_1(y)
      +e^{-2s}\left [ H^{\mu\nu}\partial_\nu H_\mu{}^j(s,y) + 4H^{\mu j}H_\mu{}^0(s,y)\right ]\partial_j f_1(y) \\
      &+2e^{-2s}[2H^{\mu 0}H_\mu{}^0(s,y)+H^{\mu\nu}\partial_\nu H_\mu{}^0(s,y)]
      f_1(y)-h(y)^{-2}V_0(y/h(y))f_1(y)=G(y),
    \end{split}
  \end{equation}
  where
  \begin{align*}
    G(y)&=2e^{-2s}H^{\mu j}H_\mu{}^0(s,y)\partial_j f_{1,1}^*(y)
          +e^{-2s}[2H^{\mu 0}H_{\mu}{}^0(s,y)+H^{\mu\nu}\partial_\nu
          H_\mu{}^0(s,y)]f_{1,1}^*(y) \\
          &\quad +e^{-2s}H^{\mu 0}H_\mu{}^0(s,y)f_{1,2}^*(y) \\
        &=2e^{-2s}H^{\mu j}H_\mu{}^0(s,y)\partial_j f_{1,1}^*(y)
          +e^{-2s}[4H^{\mu 0}H_\mu{}^0(s,y)+H^{\mu\nu}\partial_\nu H_\mu{}^0(s,y)]f_{1,1}^*(y).
  \end{align*}
  Obviously, $G$ is radial and belongs to
  $C^\infty(\overline{\B^5_{1/2}})$.
  Explicitly, we have
  \begin{align*}
    e^{-2s}H^{\mu j}H_\mu{}^0(s,y)
    &=-e^{-2s}H_0{}^jH_0{}^0(s,y)+e^{-2s}H^{k j}H_k{}^0(s,y) \\
    &=-\frac{y^j}{[y^\ell\partial_\ell h(y)-h(y)]^2}
      -\frac{\partial^j h(y)}{y^\ell\partial_\ell
      h(y)-h(y)}+\frac{\partial^k h(y)\partial_k
      h(y)}{[y^\ell\partial_\ell h(y)-h(y)]^2}y^j \\
    &=-\frac{1-\partial^k h(y)\partial_k h(y)}{[y^\ell\partial_\ell
      h(y)-h(y)]^2}y^j
      -\frac{\partial^j h(y)}{y^\ell\partial_\ell h(y)-h(y)} \\
    &=-h_1(|y|)\left [h_1(|y|)\left [1-\widehat h'(|y|)^2\right ]+\frac{\widehat
      h'(|y|)}{|y|}\right ]y^j,
  \end{align*}
  where $\widehat h(|y|):=h(y)=\sqrt{2+|y|^2}-2$ and
  \[ h_1(|y|):=\frac{1}{|y|\widehat h'(|y|)-\widehat h(|y|)}. \] 
Next,  
  \begin{align*}
    e^{-2s}H^{\mu 0}H_\mu{}^0(s,y)
    &=-e^{-2s}H_0{}^0H_0{}^0(s,y)+e^{-2s}H^{j 0}H_j{}^0(s,y) 
=-\frac{1-\partial^j h(y)\partial_j h(y)}{[y^\ell\partial_\ell
      h(y)-h(y)]^2} \\
      &=-h_1(|y|)^2\left [1-\widehat h'(|y|)^2\right].
  \end{align*}
  Furthermore, we have
  \begin{align*} H^{\mu\nu}\partial_\nu H_\mu{}^0
    &=H^{\mu 0}\partial_0 H_\mu{}^0+H^{\mu k}\partial_k H_\mu{}^0
      =H^{\mu 0}H_\mu{}^0+H^{0k}\partial_k H_0{}^0+H^{jk}\partial_k
      H_j{}^0 \\
    &=H^{\mu 0}H_\mu{}^0-H_0{}^k\partial_k H_0{}^0+H^{jk}\partial_k H_j{}^0
  \end{align*}                                                 
 and 
  \begin{align*}
    e^{-2s}H_0{}^k\partial_k H_0{}^0(s,y)
    &=\frac{y^k}{y^\ell \partial_\ell
      h(y)-h(y)}\partial_{y^k}\frac{1}{y^\ell\partial_\ell h(y)-h(y)}
=|y|h_1'(|y|)h_1(|y|).
  \end{align*}
  Finally,
  \begin{align*}
    e^{-2s}H^{jk}\partial_k H_j{}^0(s,y)
    &=-\left (\delta^{jk}-\frac{\partial^j h(y)y^k}{y^\ell \partial_\ell
      h(y)-h(y)}\right )
      \partial_{y^k}\frac{\partial_j h(y)}{y^\ell\partial_\ell
      h(y)-h(y)} \\
    &=-\frac{\partial^j\partial_j h(y)}{y^\ell \partial_\ell
      h(y)-h(y)}
      -\partial_j h(y)\partial_{y_j}\frac{1}{y^\ell\partial_\ell
      h(y)-h(y)} \\
      &\quad +\frac{\partial^j h(y)y^k}{[y^\ell\partial_\ell
        h(y)-h(y)]^2}\partial_j\partial_k h(y)  \\
    &\quad +\frac{\partial^j h(y)\partial_j h(y)}{y^\ell\partial_\ell
        h(y)-h(y)}
        y^k\partial_{y^k}\frac{1}{y^\ell\partial_\ell h(y)-h(y)}
  \end{align*}
  and thus, in terms of $\widehat h$ and $h_1$,
  \begin{align*}
    e^{-2s}H^{jk}\partial_k H_j{}^0(s,y)
    &=-h_1(|y|)\left [\widehat h''(|y|)+\frac{4\widehat
      h'(|y|)}{|y|}\right ]-h_1'(|y|)\widehat h'(|y|) \\
    &\quad +|y|h_1(|y|)^2\widehat h'(|y|)\widehat
      h''(|y|)+|y|h_1'(|y|)h_1(|y|)\widehat h'(|y|)^2.
  \end{align*}
  In summary,
  \begin{align*}
    e^{-2s}H^{\mu\nu}\partial_\nu H_\mu{}^0(s,y)
    &=-h_1(|y|)\left [\widehat h''(|y|)+\frac{4\widehat
      h'(|y|)}{|y|}\right ]
      +|y|h_1(|y|)^2\widehat h''(|y|)\widehat h'(|y|) \\
    &\quad -\left [h_1(|y|)^2+|y|h_1'(|y|)h_1(|y|)\right]\left [1-\widehat
      h'(|y|)^2\right]
      -h_1'(|y|)\widehat h'(|y|).
  \end{align*}
  With these explicit expressions at hand it is straightforward to
  check that $G(y)<0$ for all $y\in \B_{1/2}^5$.
  In particular, $G$ has no zeros in $\B^5_{1/2}$ and this will be the
  key property.

  Observe that $(\mb I-\mb L)\mb f_1^*=\mb 0$ implies that $f_{1,1}^*$
  solves Eq.~\eqref{eq:Pf1} with $G=0$. We claim that another solution
  is given by
  \[ \widetilde
  f_{1,1}^*(y)=\frac{1}{|y|^2+h(y)^2}\int_{1/2}^{-|y|/h(y)}\frac{(1+\rho^2)^2}{\rho^4(1-\rho^2)}d\rho. \]
  To see this, we start from the radial version of
  Eq.~\eqref{eq:specf} with $\lambda=1$,
  \begin{equation}
    \label{eq:Pfrad}
    -(1-\rho^2)f''(\rho)-\frac{4}{\rho}f'(\rho)+6\rho f'(\rho)+6f(\rho)-\frac{16}{(1+\rho^2)^2}f(\rho)=0. 
  \end{equation}
  Eq.~\eqref{eq:Pfrad} is of the form $f''+pf'+qf=0$ with
  $p(\rho)=\frac{4}{\rho}-\frac{2\rho}{1-\rho^2}$.  Consequently, the
  Wronskian $W(f,g)$ of two solutions $f,g$ of Eq.~\eqref{eq:Pfrad} is
  given by
  \[ W(f,g)(\rho)=W(f,g)(\tfrac12)e^{-\int_{1/2}^\rho
    p(t)dt}=\tfrac{3}{64}W(f,g)(\tfrac12)\frac{1}{\rho^4(1-\rho^2)}. \]
  Note that $\rho\mapsto \frac{1}{1+\rho^2}$ is a solution of
  Eq.~\eqref{eq:Pfrad} (cf.~the proof of Lemma \ref{lem:sigmapL}) and
  thus, by the reduction formula, another solution is given by
  \[ \rho\mapsto \frac{1}{1+\rho^2}\int_{1/2}^\rho
  \frac{(1+r^2)^2}{r^4(1-r^2)}dr. \]
  As a consequence, we see that the function
  \[ w(\tau,\xi):=e^{2\tau}\frac{1}{1+|\xi|^2}\int_{1/2}^{|\xi|}
  \frac{(1+\rho^2)^2}{\rho^4(1-\rho^2)}d\rho \] satisfies
  \[ e^{2\tau}\left [
    \partial_\tau^2+2\xi^j\partial_{\xi^j}\partial_\tau-(\delta^{jk}-\xi^j\xi^k)\partial_{\xi^j}\partial_{\xi^k}
    +\partial_\tau+2\xi^j\partial_{\xi^j}+V_0(\xi)\right
  ]w(\tau,\xi)=0 \]
  for all $(\tau,\xi)\in \R\times \B^5\setminus \{0\}$.  This means
  that $v(s,y)=w(s-\log(-h(y)),-y/h(y))$ satisfies
  \[ -\mc D^\mu\mc D_\mu v(s,y)+e^{2s}h(y)^{-2}V_0(y/h(y))v(s,y)=0 \]
  for all $(s,y)\in \R\times \B_{1/2}^5\setminus \{0\}$,
  cf.~Eq.~\eqref{eq:sytauxiwave}.
  We have
  \begin{align*}
    v(s,y)&=w(s-\log(-h(y)),-y/h(y))=e^{2s}h(y)^{-2}\frac{1}{1+|y|^2/h(y)^2}
            \int_{1/2}^{-|y|/h(y)}\frac{(1+\rho^2)^2}{\rho^4(1-\rho^2)}d\rho \\
          &=e^{2s}\widetilde f_{1,1}^*(y)
  \end{align*}
  and thus, $\widetilde f_{1,1}^*$ satisfies Eq.~\eqref{eq:Pf1} with
  $G=0$ and for all $y\in \B_{1/2}^5\setminus \{0\}$, as claimed.

  By definition, we have
  \begin{align*}
    e^{-2s}H^{\mu j}H_\mu{}^k(s,y)
    &=e^{-2s}[-H_0{}^jH_0{}^k(s,y)+H^{mj}H_m{}^k(s,y)] \\
    &=\delta^{jk}-\frac{1-\partial^m h(y)\partial_m h(y)}{[y^\ell \partial_\ell h(y)-h(y)]^2}y^jy^k
      -\frac{1}{y^\ell \partial_\ell h(y)-h(y)}\left [y^j \partial^k h(y)+y^k\partial^j h(y)\right ]
  \end{align*}
  and thus, if $f(y)=\widehat f(|y|)$, we obtain
  \begin{align*}
    e^{-2s}H^{\mu j}H_\mu{}^k(s,y)\partial_j\partial_k f(y)
    &=\widehat f''(|y|)+\frac{4}{|y|}\widehat f'(|y|)
-\frac{1-\partial^j h(y)\partial_j h(y)}{[y^\ell\partial_\ell h(y)-h(y)]^2}|y|^2\widehat f''(|y|) \\
    &\quad -\frac{2y^j\partial_j h(y)}{y^\ell\partial_\ell h(y)-h(y)}\widehat f''(|y|) \\
    &=[1-a(|y|)]\widehat f''(|y|)+\frac{4}{|y|}\widehat f'(|y|),
  \end{align*}
  where
  \begin{align*}
    a(|y|)&=\frac{2\sqrt{2+|y|^2}-1}{2(\sqrt{2+|y|^2}-1)^2}|y|^2.
  \end{align*}
  Consequently, if we write $f_1(y)=\widehat f_1(|y|)$, we see that
  Eq.~\eqref{eq:Pf1} is of the form
  \begin{equation}
    \label{eq:Pf1hat}
    \widehat f_1''(\eta)+\widetilde p(\eta)\widehat
    f_1'(\eta)+\widetilde q(\eta)f_1(\eta)
=\frac{\widehat G(\eta)}{1-a(\eta)},\qquad \eta\in (0,\tfrac12), 
  \end{equation}
  where $G(y)=\widehat G(|y|)$.  By the above, the homogeneous version of
  Eq.~\eqref{eq:Pf1hat} has the solutions \begin{align*}
                                            \phi(\eta)&=f_{1,1}^*(\eta e_1)=\frac{1}{\eta^2+h(\eta e_1)^2} \\
                                            \psi(\eta)&=\widetilde
                                                        f_{1,1}^*(\eta
                                                        e_1)=\phi(\eta)\int_{1/2}^{-\eta/h(\eta
                                                        e_1)}
                                                        \frac{(1+\rho^2)^2}{\rho^4(1-\rho^2)}d\rho.
                                          \end{align*}
                                          As for the asymptotic
                                          behavior, we note that
                                          $\phi\in
                                          C^\infty([0,\frac12])$
                                          whereas
                                          \begin{align*}
                                            |\psi(\eta)|&\simeq
                                                          \eta^{-3}|\log(\tfrac12-\eta)|,
                                            & |\psi'(\eta)|&\simeq \eta^{-4}(\tfrac12-\eta)^{-1} 
                                          \end{align*}
                                          for all
                                          $\eta\in (0,\frac12)$.
                                          Furthermore, we have
                                          $W(\phi,\psi)(\eta)\simeq
                                          \eta^{-4}(\frac12-\eta)^{-1}$.
                                          Consequently, by the
                                          variation of constants
                                          formula, there exist
                                          constants $c_1, c_2\in \C$
                                          such that
                                          \begin{equation}
                                            \begin{split}
                                              \label{eq:Pvoc}
                                              \widehat f_1(\eta)=&c_1 \phi(\eta)+c_2\psi(\eta) \\
                                              &-\phi(\eta)\int_0^\eta
                                              \frac{\psi(\rho)}{W(\phi,\psi)(\rho)}\frac{\widehat
                                                G(\rho)}{1-a(\rho)}d\rho
                                              +\psi(\eta)\int_0^\eta
                                              \frac{\phi(\rho)}{W(\phi,\psi)(\rho)}\frac{\widehat
                                                G(\rho)}{1-a(\rho)}d\rho
                                            \end{split}
                                          \end{equation} 
                                          for $\eta\in
                                          (0,\frac12)$.
                                          Taking the limit
                                          $\eta\to 0+$ yields $c_2=0$
                                          since
                                          $\widehat f_1, \widehat G\in
                                          C([0,\frac12])$.
                                          Note further that
                                          \[ \lim_{\eta\to \frac12-}
                                          \int_0^\eta
                                          \frac{\psi(\rho)}{W(\phi,\psi)(\rho)}\frac{\widehat
                                            G(\rho)}{1-a(\rho)}d\rho \]
                                          exists and thus, by sending
                                          $\eta\to \frac12-$,
                                          Eq.~\eqref{eq:Pvoc} implies
                                          \[ \int_0^{1/2}
                                          \frac{\phi(\rho)}{W(\phi,\psi)(\rho)}\frac{\widehat
                                            G(\rho)}{1-a(\rho)}d\rho=0. \]
                                          But this is impossible
                                          because the integrand has no
                                          zeros in $(0,\frac12)$.
                                          This contradiction shows
                                          that
                                          $(\mb I-\mb L)\mb f=\mb 0$
                                          for all $\mb f\in \rg\mb P$
                                          and from Lemma
                                          \ref{lem:sigmapL} we
                                          conclude that
                                          $\rg\mb P=\langle \mb
                                          f_1^*\rangle$.
                                        \end{proof}

\subsection{Control of the linearized flow}
We arrive at the main result on the linearized flow.

\begin{theorem}
  \label{thm:S}
  Fix $R\geq\frac12$, $k\in \N$, $k\geq 4$ and let $\mb S$ be the
  semigroup on $\mc H^k_R$ from Lemma \ref{lem:S}. Then there exist
  $\omega_0,M>0$ such that
  \begin{align*}
    \mb S(s)\mb P\mb f&=e^s\mb P\mb f \\
    \|\mb S(s)(\mb I-\mb P)\mb f\|_{\mc H_R^{k}}&\leq 
                                                  M e^{-\omega_0 s}\|(\mb I-\mb P)\mb f\|_{\mc H^k_R}
  \end{align*}
  for all $s\geq 0$ and $\mb f\in \mc H^k_R$.
\end{theorem}

\begin{proof}
  The first statement follows directly from Lemma \ref{lem:sigmapL}
  and Proposition \ref{prop:P}. As for the second statement, we first
  claim that there exists an $N\in \N$ such that
  \begin{equation}
    \label{eq:SR}
    \|\mb R_{\mb L}(\lambda)\mb f\|_{\mc H_R^k}\lesssim \|\mb
    f\|_{\mc H_R^k} 
  \end{equation}
  for all $\mb f\in \mc H_R^k$ and all
  $\lambda\in \Omega_N:=\{z\in \C: \Re z\geq -\frac14, |z|\geq N\}$.
  Indeed, from Theorem \ref{thm:S5} we infer
  $\Omega_N\subset \rho(\mb L_5-\mb I)$ and thus, for any
  $\lambda\in \Omega_N$ we have the identity
  $\lambda\mb I-\mb L=[\mb I-\mb L'\mb R_{\mb L_5-\mb
    I}(\lambda)](\lambda\mb I-\mb L_5+\mb I)$
  which shows that $\lambda\in \rho(\mb L)$ if and only if the
  operator $\mb I-\mb L'\mb R_{\mb L_5-\mb I}(\lambda)$ is bounded
  invertible. By a Neumann series argument we see that this is the
  case if $\|\mb L'\mb R_{\mb L_5-\mb I}(\lambda)\|_{\mc
    H^k_R}<1$. Recall that
  \[ \mb L'\mb R_{\mb L_5-\mb I}(\lambda)\mb f(y)=\begin{pmatrix} 0 \\
    V(y)[\mb R_{\mb L_5-\mb I}(\lambda)\mb f]_1(y)
  \end{pmatrix}
  \]
  and from the first component of the identity
  $(\lambda\mb I-\mb L_5+\mb I)\mb R_{\mb L_5-\mb I}(\lambda)\mb f=\mb f$
  we infer
  \[ (\lambda+1)[\mb R_{\mb L_5-\mb I}(\lambda)\mb f]_1-[\mb R_{\mb
    L_5-\mb I}(\lambda)\mb f]_2=f_1. \]
  Consequently, by noting that $V\in C^\infty(\overline{\B_R^5})$, we
  obtain
  \begin{align*}
    \|\mb L'\mb R_{\mb L_5-\mb I}(\lambda)\mb f\|_{\mc H_R^k}
    &\lesssim \|[\mb R_{\mb L_5-\mb I}(\lambda)\mb f]_1\|_{H^k(\B_R^5)}
      \lesssim |\lambda|^{-1}\|\mb f\|_{\mc H^k_R}+|\lambda|^{-1}\|\mb
      R_{\mb L_5-\mb I}(\lambda)\mb f\|_{\mc H^k_R} \\
    &\lesssim |\lambda|^{-1}\|\mb f\|_{\mc H^k_R}.
  \end{align*}
  If $N\in \N$ is sufficiently large, we therefore have
  $\|\mb L'\mb R_{\mb L_5-\mb I}(\lambda)\|_{\mc H_R^k}\leq \frac12$
  for all $\lambda\in \Omega_N$ and Eq.~\eqref{eq:SR} follows.

  Furthermore, from Lemma \ref{lem:sigmapL} we infer the existence of
  an $\omega_0>0$ such that
  \[ \|\mb R_{\mb L}(\lambda)(\mb I-\mb P)\|_{\mc H^k_R}\lesssim 1 \]
  for all $\lambda\in \C$ satisfying $\Re\lambda\geq -\omega_0$ and
  $|\lambda|\leq N$.  Combining this with Eq.~\eqref{eq:SR} we obtain
  \[ \|\mb R_{\mb L}(\lambda)(\mb I-\mb P)\|_{\mc H^k_R}\lesssim 1 \]
  for all $\lambda\in \C$ with $\Re\lambda\geq -\omega_0$.
  Consequently, an application of the Gearhart-Pr\"uss Theorem (see
  e.g.~\cite{EngNag00}, p.~302, Theorem 1.11) finishes the proof.
\end{proof}

\begin{definition}
  From now on, $\omega_0>0$ denotes the corresponding constant from
  Theorem \ref{thm:S}.
\end{definition}

\subsection{Bounds on the nonlinearity}

Next, we show that the nonlinearity is locally Lipschitz.

\begin{lemma}
  \label{lem:N}
  Fix $R, M>0$ and $k\in \N$, $k\geq 2$. Then we have the bound
  \[ \|\mb N(\mb f)-\mb N(\mb g)\|_{\mc H_R^k}\lesssim \left (\|\mb
    f\|_{\mc H_R^k}+\|\mb g\|_{\mc H_R^k}\right )\|\mb f-\mb g\|_{\mc
    H_R^k}
  \]
  for all $\mb f,\mb g\in \mc H_R^k$ satisfying
  $\|\mb f\|_{\mc H_R^k}, \|\mb g\|_{\mc H_R^k}\leq M$.
\end{lemma}

\begin{proof}
  Recall that
  \[ \mb N \begin{pmatrix} f_1 \\ f_2 \end{pmatrix}(y)
  =\begin{pmatrix} 0 \\ -H(y)N(|y|f_1(y), y) \end{pmatrix},
  \]
  where $H\in C^\infty(\overline{\B_R^5})$ and
  \begin{align*}
    N(|y|f_1(y),y)&=-\frac{\sin(2|y|\alpha_0(y)+2|y|f_1(y))-\sin(2|y|\alpha_0(y))-2|y|\cos(2|y|\alpha_0(y))f_1(y)}{|y|^3}.
  \end{align*}
  From Taylor's theorem with integral remainder we infer
  \begin{align*}
    \sin(x_0+x)-\sin(x_0)-\cos(x_0)x&=-\frac{\sin
                                      x_0}{2}x^2-\frac{x^3}{2}\int_0^1
                                      \cos(x_0+tx)(1-t)^2 dt
  \end{align*}
  and thus,
  \begin{align*}
    N(|y|f_1(y),y)&=\frac{2\sin(2|y|\alpha_0(y))}{|y|}f_1(y)^2 \\
                  &\quad +4f_1(y)^3\int_0^1 
                    \cos(2|y|\alpha_0(y)+2t|y|f_1(y))(1-t)^2 dt \\
                  &=\frac{2\sin(2|y|\alpha_0(y))}{|y|}f_1(y)^2+f_1(y)^3\Phi_0(f_1(y),|y|),
  \end{align*}
  where
  \[ \Phi_0(u,|y|)=4\int_0^1 \cos \left (2|y|\left
      (\alpha_0(y)+tu\right) \right )(1-t)^2 dt.
  \]
  Note that $y\mapsto \frac{2\sin(2|y|\alpha_0(y))}{|y|}$ belongs to
  $C^\infty(\R^5)$. Furthermore, $\Phi_0\in C^\infty(\R\times\R)$ and
  $\partial_u^j \Phi_0(u,\cdot)$ is even for any $j\in \N_0$.
  Consequently, the map $(u,y)\mapsto \Phi_0(u, |y|)$ belongs to
  $C^\infty(\R\times \R^5)$. We set $\mc N(f)(y):=-H(y)N(|y|f(y),y)$. 
 Then, by Lemma \ref{lem:Moser},
 \begin{align*}
   \|\mb N(\mb f)-\mb N(\mb g)\|_{\mc H^k_R}
&\lesssim \|\mc N(f_1)-\mc N(g_1)\|_{H^k(\B_R^5)}
\leq \|\mc N(f_1)-\mc N(g_1)\|_{H^{k+1}(\B_R^5)} \\
&\lesssim \left (\|f_1\|_{H^{k+1}(\B_R^5)}+\|g_1\|_{H^{k+1}(\B_R^5)}\right
  )
\|f_1-g_1\|_{H^{k+1}(\B_R^5)} \\
&\lesssim \left ( \|\mb f\|_{\mc H^k_R}+\|\mb g\|_{\mc H^k_R}\right )
\|\mb f-\mb g\|_{\mc H^k_R}
 \end{align*}
since $k+1\geq 3>\frac52$.
\end{proof}

\subsection{Analysis of the nonlinear evolution}
Now we turn to the full equation \eqref{eq:Psi}. By Duhamel's
principle,
\begin{equation}
  \label{eq:Psiw}
  \Phi(s)=\mb S(s-s_0)\Phi(s_0)+\int_{s_0}^s \mb S(s-s')\mb N(\Phi(s'))ds'
\end{equation}
is a weak formulation of Eq.~\eqref{eq:Psi}.  In general, this
equation will not have a solution for all $s\geq s_0$ due to the
one-dimensional instability of the linearized flow. Consequently, as
an intermediate step, we modify Eq.~\eqref{eq:Psiw} according to the
Lyapunov-Perron method from dynamical systems theory.
\begin{definition}
  For $R\geq\frac12$, $k\in \N$, $k\geq 4$, $s_0\in \R$, and $\omega_0>0$
  from Theorem \ref{thm:S}, we set
  \[ \mc X_R^k(s_0):=\{\Phi\in C([s_0,\infty),\mc H_R^k):
  \|\Phi\|_{\mc X_R^k(s_0)}<\infty\}, \] where
  \[ \|\Phi\|_{\mc X_R^k(s_0)}:=\sup_{s>s_0}\left [e^{\omega_0
      s}\|\Phi(s)\|_{\mc H_R^k}\right ]. \]
  Furthermore, we define
  $\mb C_{s_0}: \mc X_R^k(s_0)\times \mc H_R^k \to \rg \mb P$ by
  \[ \mb C_{s_0}(\Phi, \mb f):=\mb P\left (\mb f+\int_{s_0}^\infty
    e^{s_0-s'}\mb N(\Phi(s'))ds' \right ). \]
\end{definition}
Instead of Eq.~\eqref{eq:Psiw} we now consider the modified equation
\begin{equation}
  \label{eq:Psim}
  \Phi(s)=\mb S(s-s_0)\left [\mb f-\mb C_{s_0}(\Phi,\mb f)\right ]+\int_{s_0}^s \mb
  S(s-s')\mb N(\Phi(s'))ds'.
\end{equation}
This modification is standard in center manifold theory
and it allows one to later mod out the instability by adjusting the blowup time
$T$. 

\begin{proposition}
  \label{prop:mod}
  Fix $R\geq\frac12$, $s_0\in \R$, and $k\in \N$, $k\geq 4$. Then there
  exists a $c>0$ and a $\delta>0$ such that for any
  $\mb f\in \mc H_R^k$ satisfying
  $\|\mb f\|_{\mc H_R^k}\leq \frac{\delta}{c}$, there exists a unique
  solution $\Phi_{\mb f}\in C([s_0,\infty),\mc H_R^k)$ to
  Eq.~\eqref{eq:Psim} that satisfies $\|\Phi_{\mb f}(s)\|_{\mc H_R^k}\leq \delta e^{-\omega_0 s}$ for all
  $s\geq s_0$. Furthermore, the solution map $\mb f\mapsto \Phi_{\mb f}$ is
  Lipschitz as a function from (a small ball in) $\mc H_R^k$ to
  $\mc X_R^k(s_0)$.
\end{proposition}

\begin{proof}
We set $\mc Y_\delta:=\{\Phi\in \mc X_R^k(s_0): \|\Phi\|_{\mc
  X_R^k(s_0)}\leq \delta\}$ and
\[ \mb K_{\mb f}(\Phi)(s):=  \mb S(s-s_0)\left [\mb f-\mb C_{s_0}(\Phi,\mb f)\right ]+\int_{s_0}^s \mb
  S(s-s')\mb N(\Phi(s'))ds'. \]
Let $\Phi\in \mc Y_\delta$. By Theorem \ref{thm:S} and Proposition \ref{prop:P},
\begin{align*}
\mb P\mb K_{\mb f}(\Phi)(s)
&=e^{s-s_0}[\mb P\mb f-\mb P\mb C_{s_0}(\Phi,\mb f)]+\int_{s_0}^s
  e^{s-s'}\mb P\mb N(\Phi(s'))ds' \\
&=-\int_s^\infty e^{s-s'}\mb P\mb N(\Phi(s'))ds'
\end{align*}
and, since $\mb N(\mb 0)=\mb 0$, Lemma \ref{lem:N} yields
\begin{align*}
 \|\mb P\mb K_{\mb f}(\Phi)(s)\|_{\mc H_R^k}
&\lesssim
e^s\int_s^\infty e^{-s'}
\|\Phi(s')\|_{\mc H_R^k}^2 ds'
\lesssim \|\Phi\|_{\mc X_R^k(s_0)}^2 e^s\int_s^\infty e^{-s'-2\omega_0
  s'}ds'  \\
&\lesssim \delta^2 e^{-2\omega_0 s}
\end{align*}
for all $s\geq s_0$.
Furthermore,
\begin{align*}
  (\mb I-\mb P)\mb K_{\mb f}(\Phi)(s)
&=\mb S(s-s_0)(\mb I-\mb P)\mb f+\int_{s_0}^s \mb S(s-s')(\mb I-\mb P)\mb N(\Phi(s'))ds'
\end{align*}
and thus, by Theorem \ref{thm:S},
\begin{align*}
  \|(\mb I-\mb P)\mb K_{\mb f}(\Phi)(s)\|_{\mc H_R^k}
&\lesssim e^{-\omega_0 s}\|\mb f\|_{\mc H_R^k}+\int_{s_0}^s
  e^{-\omega_0(s-s')}\|\mb N(\Phi(s'))\|_{\mc H_R^k}ds' \\
&\lesssim \tfrac{\delta}{c}e^{-\omega_0 s}
+e^{-\omega_0 s}\int_{s_0}^s e^{\omega_0 s'}\|\Phi(s')\|_{\mc H_R^k}^2
  ds'  \\
&\lesssim \tfrac{\delta}{c}e^{-\omega_0 s} + \|\Phi\|_{\mc
  X_R^k(s_0)}^2 e^{-\omega_0
  s}\int_{s_0}^s e^{-\omega_0 s'}ds' \\
&\lesssim \tfrac{\delta}{c}e^{-\omega_0 s} +\delta^2 e^{-\omega_0 s}
\end{align*}
for all $s\geq s_0$. 
Consequently, by choosing $c>0$ large enough and $\delta>0$ small enough, we infer $\|\mb K_{\mb
  f}(\Phi)(s)\|_{\mc H_R^k}\leq \delta e^{-\omega_0 s}$ for all $s\geq
s_0$. In other words,
$\mb K_{\mb f}(\Phi)\in \mc Y_\delta$ for all $\Phi\in\mc Y_\delta$.

Next, we show that $\mb K_{\mb f}$ is a contraction on $\mc Y_\delta$.
For $\Phi,\Psi\in \mc Y_\delta$ we have
\begin{align*}
  \mb P\mb K_{\mb f}(\Phi)(s)-\mb P\mb K_{\mb
  f}(\Psi)(s)=-\int_s^\infty e^{s-s'}\mb P\left [
\mb N(\Phi(s'))-\mb N(\Psi(s'))\right ]ds'
\end{align*}
and thus, by Lemma \ref{lem:N},
\begin{align*}
  \|\mb P\mb K_{\mb f}(\Phi)(s)-\mb P\mb K_{\mb
  f}(\Psi)(s)\|_{\mc H_R^k}
&\lesssim e^s \int_s^\infty e^{-s'}\left [\|\Phi(s')\|_{\mc
  H_R^k}+\|\Psi(s')\|_{\mc H_R^k}\right ]\|\Phi(s')-\Psi(s')\|_{\mc
  H_R^k}ds' \\
&\lesssim \delta \|\Phi-\Psi\|_{\mc X_R^k(s_0)}e^s\int_s^\infty
  e^{-s'-2\omega_0 s'}ds' \\
&\lesssim \delta e^{-2\omega_0 s}\|\Phi-\Psi\|_{\mc X_R^k(s_0)}
\end{align*}
for all $s\geq s_0$.
Similarly,
\begin{align*}
(\mb I-\mb P)\mb K_{\mb f}(\Phi)(s)-(\mb I-\mb P)\mb K_{\mb
  f}(\Psi)(s)=
\int_{s_0}^s \mb S(s-s')(\mb I-\mb P)\left [
\mb N(\Phi(s'))-\mb N(\Psi(s'))\right ]ds'
\end{align*}
and thus, by Theorem \ref{thm:S} and Lemma \ref{lem:N},
\begin{align*}
  \|&(\mb I-\mb P)\mb K_{\mb f}(\Phi)(s)-(\mb I-\mb P)\mb K_{\mb
  f}(\Psi)(s)\|_{\mc H_R^k} \\
&\lesssim \int_{s_0}^s e^{-\omega_0(s-s')}\|\mb N(\Phi(s'))-\mb
  N(\Psi(s'))\|_{\mc H_R^k}ds' \\
&\lesssim \int_{s_0}^s e^{-\omega_0(s-s')}
\left [\|\Phi(s')\|_{\mc
  H_R^k}+\|\Psi(s')\|_{\mc H_R^k}\right ]\|\Phi(s')-\Psi(s')\|_{\mc
  H_R^k} ds' \\
&\lesssim \delta \|\Phi-\Psi\|_{\mc X_R^k(s_0)}e^{-\omega_0
  s}\int_{s_0}^s e^{-\omega_0 s'}ds' \\
&\lesssim \delta e^{-\omega_0 s}\|\Phi-\Psi\|_{\mc X_R^k(s_0)}
\end{align*}
for all $s\geq s_0$. In summary, $\|\mb K_\mb f(\Phi)-\mb K_{\mb
  f}(\Psi)\|_{\mc X_R^k(s_0)}\lesssim \delta \|\Phi-\Psi\|_{\mc
  X_R^k(s_0)}$ for all $\Phi,\Psi\in \mc Y_\delta$ and upon choosing
$\delta>0$ sufficiently small, the contraction mapping principle
yields the existence of a unique fixed point $\Phi_{\mb f}\in \mc
Y_\delta$ of $\mb K_{\mb f}$.

Finally, we prove the Lipschitz continuity of the solution map.
We have
\begin{align*}
  \|\Phi_{\mb f}-\Phi_{\mb g}\|_{\mc X_R^k(s_0)}
&=\|\mb K_{\mb f}(\Phi_{\mb f})-\mb K_{\mb g}(\Phi_{\mb g})\|_{\mc
  X_R^k(s_0)} \\
&\leq \|\mb K_{\mb f}(\Phi_{\mb f})-\mb K_{\mb f}(\Phi_{\mb g})\|_{\mc
  X_R^k(s_0)}+
\|\mb K_{\mb f}(\Phi_{\mb g})-\mb K_{\mb g}(\Phi_{\mb g})\|_{\mc
  X_R^k(s_0)} \\
&\lesssim \delta \|\Phi_{\mb f}-\Phi_{\mb g}\|_{\mc X_R^k(s_0)}
+\|\mb K_{\mb f}(\Phi_{\mb g})-\mb K_{\mb g}(\Phi_{\mb g})\|_{\mc
  X_R^k(s_0)}
\end{align*}
and, since
\[ \mb K_{\mb f}(\Phi_{\mb g})(s)-\mb K_{\mb g}(\Phi_{\mb g})(s)=\mb
S(s-s_0)(\mb I-\mb P)(\mb f-\mb g), \]
Theorem \ref{thm:S} yields
\[ \|\mb K_{\mb f}(\Phi_{\mb g})(s)-\mb K_{\mb g}(\Phi_{\mb
  g})(s)\|_{\mc H_R^k}\lesssim e^{-\omega_0 s}\|\mb f-\mb g\|_{\mc
  H_R^k} \]
for all $s\geq s_0$.
In summary, $\|\Phi_{\mb f}-\Phi_{\mb g}\|_{\mc X_R^k(s_0)}\lesssim 
\delta \|\Phi_{\mb f}-\Phi_{\mb g}\|_{\mc X_R^k(s_0)}+\|\mb f-\mb
g\|_{\mc H_R^k}$ and if $\delta>0$ is sufficiently small, the claimed
Lipschitz bound follows.
\end{proof}

\section{Proof of the main result}

\noindent We are now in a position to prove Theorem \ref{thm:main}

\subsection{Construction of the data on the hyperboloid}

As a first step, we evolve the data prescribed at $t=0$ using the
standard Cauchy theory. 
For this we use the following local existence result.

\begin{definition}
  For $\epsilon>0$ we define the spacetime region
  $\Lambda_\epsilon\subset \R^{1,5}$ by
  \[ \Lambda_\epsilon:=[-4\epsilon,4\epsilon]\times \R^5 \cup
  \{(t,x)\in \R^{1,5}: -|x|+\epsilon\leq t\leq |x|-\epsilon\}, \]
  see Fig.~\ref{fig:Lambda}.
\end{definition}

\begin{definition}
  For $\delta,\epsilon>0$ and $m\in \N$ we set
  \[ \mc B_{\delta,\epsilon}^m:=\{(f,g)\in C^\infty(\R^5)\times
  C^\infty(\R^5)\mbox{ radial}: \mathrm{supp}(f,g)\subset \B^5_\epsilon,
  \|(f,g)\|_{H^m(\R^5)\times H^{m-1}(\R^5)}\leq \delta\}. \]
\end{definition}

\begin{lemma}
  \label{lem:locwm}
  Let $m\in \N$ and $m\geq 8$. Then there exists an $\epsilon>0$ such
  that for any pair of functions $(f,g)\in \mc B_{1,\epsilon}^m$ there
  exists a unique solution $u=u_{f,g}\in C^\infty(\Lambda_\epsilon)$
  in $\Lambda_\epsilon$ to the Cauchy problem
  \begin{equation}
    \label{eq:Cauchyinit}
    \left \{\begin{array}{l}
              (\partial_t^2-\Delta_x)u(t,x)=\frac{2|x|u(t,x)-\sin(2|x|u(t,x))}{|x|^3} \\
              u(0,x)=u_1^*(0,x)+f(x),\quad \partial_0 u(0,x)=\partial_0 u_1^*(0,x)+g(x).
            \end{array} \right.
        \end{equation}
        Furthermore, for any multi-index $\alpha\in \N_0^6$ of length
        $|\alpha|\leq m-3$, we have the estimate
        \[ \sup_{(t,x)\in \Lambda_\epsilon}|\partial^\alpha
        u_{f,g}(t,x)-\partial^\alpha u_1^*(t,x)|\lesssim
        \|(f,g)\|_{H^m(\R^5)\times H^{m-1}(\R^5)}. \]
      \end{lemma}

\begin{proof}
  Thanks to Lemma \ref{lem:wmsmooth}, Theorems \ref{thm:LWP},
  \ref{thm:uniq}, and
  \ref{thm:reg} apply to the Cauchy problem
  \eqref{eq:Cauchyinit}.  From Theorem \ref{thm:LWP} we obtain an
  $\epsilon \in (0,\frac19)$ and the existence of a solution $u$ to
  the Cauchy problem \eqref{eq:Cauchyinit} in the truncated lightcone
  $\bigcup_{t\in [-4\epsilon,4\epsilon]}\{t\}\times \B^5_{1-|t|}$ for
  any $(f,g)\in \mc B_{1,1}^m$. In particular, this existence result
  holds for all $(f,g)\in \mc B_{1,\epsilon}^m\subset \mc B_{1,1}^m$.
  Let $(f,g)\in \mc B_{1,\epsilon}^m$.  Since the support of $(f,g)$
  is contained in the ball $\B^5_\epsilon$, it follows from finite
  speed of propagation (Theorem \ref{thm:uniq}) that the unique
  solution $u$ to Eq.~\eqref{eq:Cauchyinit} in the domain
  $\{(t,x)\in \R^{1,5}: -|x|+\epsilon\leq t\leq |x|-\epsilon\}$ is
  $u=u_1^*$.  In summary, we obtain a solution $u$ in
  $\Lambda_\epsilon$ and by Theorem \ref{thm:reg},
  $u\in C^\infty(\Lambda_\epsilon)$.

  Furthermore, from Theorem \ref{thm:LWP} we have the Lipschitz bound
  \begin{align*}
    \sup_{t\in
    [-4\epsilon,4\epsilon]}
    &\|(u(t,\cdot),\partial_t u(t,\cdot))-(u_1^*(t,\cdot), \partial_t
      u_1^*(t,\cdot))\|_{H^m(\R^5)\times H^{m-1}(\R^5)} \\
    &= \sup_{t\in
      [-4\epsilon,4\epsilon]}\|(u(t,\cdot),\partial_t u(t,\cdot))-(u_1^*(t,\cdot), \partial_t
      u_1^*(t,\cdot))\|_{H^m(\B_{1-|t|}^5)\times H^{m-1}(\B_{1-|t|}^5)} \\
    &\lesssim \|(u(0,\cdot),\partial_0
      u(0,\cdot))-(u_1^*(0,\cdot),\partial_0
      u_1^*(0,\cdot))\|_{H^m(\R^5)\times H^{m-1}(\R^5)} \\
    &=\|(f,g)\|_{H^m(\R^5)\times H^{m-1}(\R^5)}
  \end{align*}
  and thus, by Sobolev embedding,
  \[ \sup_{(t,x)\in \Lambda_\epsilon}|\partial^\alpha
  u(t,x)-\partial^\alpha u_1^*(t,x)|\lesssim \|(f,g)\|_{H^m(\R^5)\times
    H^{m-1}(\R^5)}
  \]
  for all multi-indices $\alpha\in \{0,1\}\times \N_0^5$ of length
  $|\alpha|\leq m-3$.  Time derivatives of higher order are estimated
  by using the equation to translate them into spatial derivatives.
 This way, the stated bound follows.
\end{proof}

Now we obtain the initial data for the hyperboloidal evolution by
evaluating the solution from Lemma \ref{lem:locwm} on a suitable
hyperboloid.  Recall from Section \ref{sec:pert} that in terms of the
variable
\begin{align*}
  \Phi(s)(y)&:=e^{-s}\begin{pmatrix} v(s,y) \\ \partial_s
    v(s,y) \end{pmatrix}:=e^{-s}\begin{pmatrix}\widetilde
    u(T+e^{-s}h(y),e^{-s}y) \\ \partial_s \widetilde
    u(T+e^{-s}h(y),e^{-s}y) \end{pmatrix} \\
            &:=e^{-s} \begin{pmatrix}
              u(T+e^{-s}h(y),e^{-s}y)-u_T^*(T+e^{-s}h(y),e^{-s}y) \\
              \partial_s u(T+e^{-s}h(y),e^{-s}y)-\partial_s
              u_T^*(T+e^{-s}h(y),e^{-s}y)
            \end{pmatrix},
\end{align*}
Eq.~\eqref{eq:mainu} reads
\begin{equation}
  \label{eq:Phifull}
  \partial_s\Phi(s)=\mb L \Phi(s)+\mb N(\Phi(s)). 
\end{equation}
This motivates the definition of the following \emph{initial data
  operator}.

\begin{definition}
  Let $R\geq\frac12$ and $k\in \N$, $k\geq 4$.  For $\epsilon>0$
  sufficiently small, we define a map
  $\mb U: \mc B^{k+4}_{1,\epsilon}\times [1-\epsilon,1+\epsilon]\to
  \mc H_R^k$
  as follows.  For $(f,g)\in \mc B_{1,\epsilon}^{k+4}$ let
  $u_{f,g}\in C^\infty(\Lambda_\epsilon)$ be the corresponding unique
  solution to the Cauchy problem \eqref{eq:Cauchyinit} from Lemma
  \ref{lem:locwm}.  Then we set
  \[ \mb U((f,g),T)(y):=e^{-s}\left .\begin{pmatrix}
      u_{f,g}(T+e^{-s}h(y),e^{-s}y)-u_T^*(T+e^{-s}h(y),e^{-s}y) \\
      \partial_s u_{f,g}(T+e^{-s}h(y),e^{-s}y)-\partial_s
      u_T^*(T+e^{-s}h(y),e^{-s}y)
    \end{pmatrix}\right |_{s=\log(-\frac{h(0)}{1+2\epsilon})}.
  \]
\end{definition}

\begin{figure}[ht]
  \centering
  \includegraphics{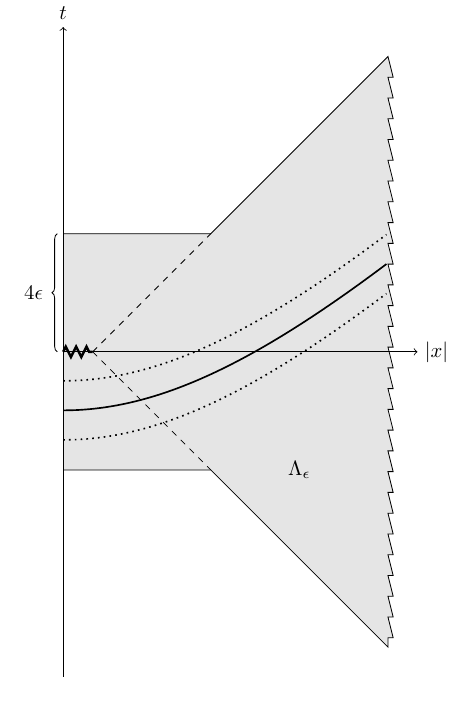}
  \caption{A spacetime diagram illustrating the construction of the
    initial data for the hyperboloidal evolution. The shaded region
    depicts the domain $\Lambda_\epsilon$. The solid curved line
    represents the hyperboloid parametrized by the map
    $y\mapsto (T+e^{-s}h(y), e^{-s}y)$ with $T=1$ and
    $s=\log(-\frac{h(0)}{1+2\epsilon})$. The dotted lines are the
    corresponding translated hyperboloids with the same $s$ but
    $T=1-\epsilon$ and $T=1+\epsilon$, respectively. The zigzag
    segment represents the support of the initial data $(f,g)$.}
  \label{fig:Lambda}
\end{figure}

Consequently, our goal is now to solve Eq.~\eqref{eq:Phifull} for
$s\geq s_0=\log(-\frac{h(0)}{1+2\epsilon})$ with initial data
$\Phi(s_0)=\mb U((f,g),T)$ and a suitable
$T \in [1-\epsilon,1+\epsilon]$.

\subsection{Properties of the initial data operator}
First, we prove mapping properties for the initial data operator
$\mb U$.

\begin{lemma}
  \label{lem:U}
  Let $R\geq\frac12$ and $k\in \N$, $k\geq 4$.  Then there exists an
  $\epsilon>0$ such that the initial data operator
  $\mb U: \mc B^{k+4}_{1,\epsilon}\times [1-\epsilon,1+\epsilon]\to
  \mc H_R^k$
  is well-defined and for any $(f,g)\in \mc B_{1,\epsilon}^{k+4}$, the
  map $\mb U((f,g),\cdot): [1-\epsilon,1+\epsilon]\to\mc H_R^k$ is
  continuous. Furthermore, there exists a $\gamma_\epsilon\in \R$ such
  that
  \[ \mb U((f,g),T)=\gamma_\epsilon (T-1)\mb f_1^*+\mb V((f,g),T), \]
  where $\mb V((f,g),T)$ satisfies the bound
  \[ \|\mb V((f,g),T)\|_{\mc H^k_R}\lesssim
  \|(f,g)\|_{H^{k+4}(\R^5)\times H^{k+3}(\R^5)}+|T-1|^2 . \]
\end{lemma}

\begin{proof}
  According to Lemma \ref{lem:locwm}, there exists an $\epsilon>0$
  such that the operator is well-defined since the hyperboloids on
  which the function $u_{f,g}$ is evaluated lie entirely inside of
  $\Lambda_\epsilon$, see Fig.~\ref{fig:Lambda}.  The statement about
  the continuity is a simple consequence of the fact that
  $u_{f,g}, u_T^* \in C^\infty(\Lambda_\epsilon)$.  To prove the last
  assertion, we rewrite $\mb U$ as
  \begin{align*}
    \mb U((f,g),T)(y)
    &=
      e^{-s}\left .\begin{pmatrix}
          u_{f,g}(T+e^{-s}h(y),e^{-s}y)-u_1^*(T+e^{-s}h(y),e^{-s}y) \\
          \partial_s u_{f,g}(T+e^{-s}h(y),e^{-s}y)-\partial_s
          u_1^*(T+e^{-s}h(y),e^{-s}y)
        \end{pmatrix}\right |_{s=\log(-\frac{h(0)}{1+2\epsilon})} \\
    &\quad +e^{-s}\left .\begin{pmatrix}
        u_1^*(T+e^{-s}h(y),e^{-s}y)-u_T^*(T+e^{-s}h(y),e^{-s}y) \\
        \partial_s u_1^*(T+e^{-s}h(y),e^{-s}y)-\partial_s
        u_T^*(T+e^{-s}h(y),e^{-s}y)
      \end{pmatrix}\right |_{s=\log(-\frac{h(0)}{1+2\epsilon})}.
  \end{align*}
  Now note that $u_{T+t_0}^*(t+t_0,x)=u_T^*(t,x)$ for any $t_0\in \R$
  and thus, by expanding around $T=1$, we infer
  \begin{align*}
    u_1^*(T+e^{-s}h(y),e^{-s}y)
    &=u_{2-T}^*(1+e^{-s}h(y),e^{-s}y) \\
    &=u_1^*(1+e^{-s}h(y),e^{-s}y) \\
    &\quad +\partial_T u_{2-T}^*(1+e^{-s}h(y),e^{-s}y)|_{T=1}(T-1)+\varphi_T(s,y),
  \end{align*}
  where
  $\|(\varphi_T(s,\cdot),\partial_s\varphi_T(s,\cdot))\|_{\mc
    H_R^k}\leq C_s |T-1|^2$.  Since
  \[ \partial_T u_T^*(t,x)=-\frac{2}{(T-t)^2+|x|^2}, \]
  see the proof of Lemma \ref{lem:sigmapL}, we obtain
  \[ \partial_T u_{2-T}^*(t,x)|_{T=1}=-\partial_T
  u_T^*(t,x)|_{T=1}=\frac{2}{(1-t)^2+|x|^2} \] and thus,
  \begin{align*}
    u_1^*(T+e^{-s}h(y),e^{-s}y)
    &=u_1^*(1+e^{-s}h(y),e^{-s}y)+2\frac{e^{2s}}{|y|^2+h(y)^2}(T-1)+\varphi_T(s,y)
    \\
    &=u_T^*(T+e^{-s}h(y),e^{-s}y)+2e^{2s}f_{1,1}^*(y)(T-1)+\varphi_T(s,y) \\
    \partial_s u_1^*(T+e^{-s}h(y),e^{-s}y)
    &=\partial_s u_T^*(T+e^{-s}h(y),e^{-s}y)+2e^{2s}f_{1,2}^*(y)(T-1)+\partial_s\varphi_T(s,y) .
  \end{align*}
  This yields the stated representation.  The bound on
$\mb V((f,g),T)$
  follows easily by Sobolev embedding and the estimate from Lemma
  \ref{lem:locwm}.
\end{proof}

\FloatBarrier

\subsection{Hyperboloidal evolution}

The last step in the proof of Theorem \ref{thm:main} consists of the
hyperboloidal evolution.

\begin{proposition}
  \label{prop:hyp}
  Let $R\geq\frac12$ and $k\in \N$, $k\geq 4$.  Then there exists an
  $M>0$ and $\delta,\epsilon>0$ such that for any pair
  $(f,g)\in \mc B_{\delta/M,\epsilon}^{k+4}$ there exists a
  $T_{f,g}\in [1-\delta,1+\delta]$ and a unique function
  $\Phi_{f,g}\in C([s_0,\infty),\mc H_R^k)$ that satisfies
  \[ \Phi_{f,g}(s)=\mb S(s-s_0)\mb U((f,g),T_{f,g})+\int_{s_0}^s \mb
  S(s-s')\mb N(\Phi_{f,g}(s'))ds',\qquad s_0:=\log\left
    (-\frac{h(0)}{1+2\epsilon}\right ) \]
and
  $\|\Phi_{f,g}(s)\|_{\mc H^k_R}\leq \delta e^{-\omega_0 s}$ for all
  $s\geq s_0$.
\end{proposition}

\begin{proof}
  Let $c,\delta>0$ be as in Proposition \ref{prop:mod} and choose
  $\epsilon>0$ so small that
  $\mb U: \mc B_{1,\epsilon}^{k+4}\times [1-\epsilon,1+\epsilon]\to\mc
  H_R^k$
  is well-defined, see Lemma \ref{lem:U}.  By choosing $M_0\geq 1$
  sufficiently large, we obtain $\frac{\delta}{M_0}\leq \epsilon$ and
  Lemma \ref{lem:U} yields the bound
  $\|\mb U((f,g),T)\|_{\mc H_R^k}\leq \frac{\delta}{c}$ for all
  $(f,g)\in \mc B_{\delta/M_0^2,\epsilon}^{k+4}$ and
  $T\in [1-\frac{\delta}{M_0},1+\frac{\delta}{M_0}]$.  Consequently,
  Proposition \ref{prop:mod} implies that for any
  $(f,g)\in \mc B_{\delta/M_0^2,\epsilon}^{k+4}$ and
  $T\in [1-\frac{\delta}{M_0},1+\frac{\delta}{M_0}]$, there exists a
  unique function $\Phi_{f,g,T}\in C([s_0,\infty),\mc H_R^k)$
  satisfying
  \[ \Phi_{f,g,T}(s)=\mb S(s-s_0)\left [\mb U((f,g),T)-\mb C_{s_0}\big
    (\Phi_{f,g,T}, \mb U((f,g),T)\big )\right ]+\int_{s_0}^s \mb
  S(s-s')\mb N\big (\Phi_{f,g,T}(s')\big )ds',
  \]
  where $s_0:=\log(-\frac{h(0)}{1+2\epsilon})$.  Furthermore, we have
  the bound
  $\|\Phi_{f,g,T}(s)\|_{\mc H_R^k}\leq \delta e^{-\omega_0 s}$ for all
  $s\geq s_0$.  Thus, our goal is to show that there exists a
  $T_{f,g}\in [1-\frac{\delta}{M_0}, 1+\frac{\delta}{M_0}]$ such that
  \[ \mb C_{s_0}\big (\Phi_{f,g,T_{f,g}}, \mb U((f,g),T_{f,g})\big
  )=\mb 0. \]
  To this end, we define a function
  $\varphi_{f,g}: [1-\frac{\delta}{M_0},1+\frac{\delta}{M_0}]\to \R$
  by
  \[ \varphi_{f,g}(T):=\Big (\mb C_{s_0}\big (\Phi_{f,g,T}, \mb
  U((f,g),T)\big ) \Big | \mb f_1^* \Big )_{\mc H_R^k}.  \]
  By Proposition \ref{prop:mod} and Lemma \ref{lem:U}, $\varphi_{f,g}$
  is continuous.  Recall that
  \[ \mb C_{s_0}\big (\Phi_{f,g,T}, \mb U((f,g),T)\big ) =\mb P\mb
  U((f,g),T)+\mb P\int_{s_0}^\infty e^{s_0-s'}\mb N(\Phi_{f,g,T}(s'))ds' \]
  and from Lemmas \ref{lem:N} and \ref{lem:U} we see that there exists
  a nonzero constant $\widetilde \gamma_\epsilon$ such that
  \[ \varphi_{f,g}(T)=\widetilde\gamma_\epsilon(T-1)+\phi_{f,g}(T) \]
  with a continuous function
  $\phi_{f,g}: [1-\frac{\delta}{M_0},1+\frac{\delta}{M_0}]\to \R$
  satisfying
  $|\phi_{f,g}(T)|\lesssim
  \frac{\delta}{M_0^2}+\delta^2$
  for all $T\in [1-\frac{\delta}{M_0},1+\frac{\delta}{M_0}]$.
  Consequently, the condition $\varphi_{f,g}(T)=0$ is equivalent to
  the fixed point problem
  $T=1-\widetilde\gamma_\epsilon^{-1}\phi_{f,g}(T)$ and if we choose
  $M_0$ large enough and $\delta>0$ sufficiently small,
  $T\mapsto 1-\widetilde\gamma_\epsilon^{-1}\phi_{f,g}(T)$ becomes a
  continuous self-map of the interval
  $[1-\frac{\delta}{M_0},1+\frac{\delta}{M_0}]$ which necessarily has
  a fixed point $T_{f,g}$.  By Proposition \ref{prop:P}, $\mb C_{s_0}$
  has values in $\langle \mb f_1^*\rangle$ and thus,
  \[ \mb C_{s_0}\big (\Phi_{f,g,T_{f,g}}, \mb U((f,g),T_{f,g})\big
  )=\mb 0, \] as desired. The proof is finished by setting $M=M_0^2$.
\end{proof}

\subsection{Proof of Theorem \ref{thm:main}}

Let $m\in \N$, $m\geq 8$.
According to Lemmas \ref{lem:locwm}, \ref{lem:U}, and Proposition
\ref{prop:hyp}, there exists an $\epsilon>0$ such that for any pair of
functions $(f,g)\in \mc B_{\delta/M,\epsilon}^m$ there exists a
$T\in [1-\delta,1+\delta]$ and a continuous function $\Phi=(\phi_1,\phi_2): [s_0,\infty)\to \mc
H_R^{m-4}$ that satisfies
\[ \Phi(s)=\mb S(s-s_0)\mb U((f,g),T)+\int_{s_0}^s \mb
S(s-s')\mb N(\Phi(s'))ds' \]
for all $s\geq s_0:=\log(-\frac{h(0)}{1+2\epsilon})$ and
$\|\Phi(s)\|_{\mc H^{m-4}_R}\leq\delta e^{-\omega_0 s}$.
Since the data $\mb U((f,g),T)$ are smooth, they belong to $\mc
D(\mb L)$ and the function $\Phi$ is a
classical solution to the equation
\[ \partial_s \Phi(s)=\mb L\Phi(s)+\mb N(\Phi(s)). \]
By a simple inductive argument it follows that $\Phi$ is
in fact smooth, cf.~the proof of Theorem \ref{thm:reg}.
By construction, the function $u$, given by
\[ (u\circ \eta_T)(s,y)=(u_T^*\circ\eta_T)(s,y)+e^{s}\phi_1(s)(y), \]
satisfies Eq.~\eqref{eq:mainu} in the domain
$\eta_T([s_0,\infty)\times \B_R^5)$ and we have
\[ \partial_s(u\circ \eta_T)(s,y)=\partial_s (u_T^*\circ
\eta_T)(s,y)+e^{s}\phi_2(s)(y). \]
By Theorem \ref{thm:uniq} we have $u(0,\cdot)=u_1^*(0,\cdot)+f$ and
$\partial_0 u(0,\cdot)=\partial_0 u_1^*(0,\cdot)+g$ and
the stated bounds in Theorem \ref{thm:main} follow
immediately from 
\[ \|\Phi(s)\|_{\mc H_R^{m-4}}^2=\|\phi_1(s)\|_{H^{m-3}(\B_R^5)}^2
+\|\phi_2(s)\|_{H^{m-4}(\B_R^5)}^2\leq \delta^2 e^{-2\omega_0 s}. \]
Finally, $u=u_1^*$ in $\Omega_{T,b}\setminus
\eta_T([s_0,\infty)\times \B_R^5)$ is a consequence of finite speed of
propagation, Theorem \ref{thm:uniq}.

\appendix

\section{Technical lemmas}

\noindent In this section we collect some technical lemmas and elementary
estimates.  We start with two variants of the classical Hardy inequality
in one dimension.

\begin{lemma}
  \label{lem:Hardy1}
  Let $s\in \R\setminus \{-\frac12\}$. Then we have the estimate
  \[ \||\cdot|^s f\|_{L^2(\R)}\lesssim
  \||\cdot|^{s+1}f'\|_{L^2(\R)} \] for all $f\in \mc S(\R)$ satisfying
  \[ \lim_{x\to 0}\left [|x|^{s+\frac12}|f(x)|\right ]=0. \]
\end{lemma}

\begin{proof}
  For $n\in \N$ we have
  \begin{align*}
    \int_{\R\setminus \B_{1/n}}|x|^{2s}|f(x)|^2 dx
    &=\int_{-\infty}^{-1/n} (-x)^{2s}f(x)^2dx+\int_{1/n}^\infty  
      x^{2s}f(x)^2 dx \\
    &=-\tfrac{1}{2s+1}(-x)^{2s+1}f(x)^2\Big|_{-\infty}^{-1/n}
      +\tfrac{1}{2s+1}x^{2s+1}f(x)^2 \Big |_{1/n}^\infty \\
    &\quad +\frac{2}{2s+1}\left [\int_{-\infty}^{-1/n}(-x)^{2s+1}f'(x)f(x)dx
      -\int_{1/n}^{\infty}x^{2s+1}f'(x)f(x)dx \right ]
  \end{align*}
  and Cauchy-Schwarz implies
  \[ \||\cdot|^s f\|_{L^2(\R\setminus \B_{1/n})}^2 \lesssim
B_n(f)
+\||\cdot|^s f\|_{L^2(\R\setminus \B_{1/n})}\|
|\cdot|^{s+1}f'\|_{L^2(\R\setminus
  \B_{1/n})}, \]
where
\[ B_n(f):=    (\tfrac{1}{n})^{2s+1}f(-\tfrac{1}{n})^2
  +(\tfrac{1}{n})^{2s+1}f(\tfrac{1}{n})^2. \]
By assumption, $B_n(f)\to 0$ as $n\to\infty$ and the 
  $L^2$ norms in the above inequality are monotonically increasing
  functions of $n$.
  Consequently, we infer
  \begin{align*} \||\cdot|^s f\|_{L^2(\R\setminus \B_{1/n})} &\lesssim
    \frac{B_n(f)}{\||\cdot|^s f\|_{L^2(\R\setminus \B_{1/n})}}
    +\||\cdot|^{s+1}f'\|_{L^2(\R\setminus\B_{1/n})} \\
    &\lesssim
    \frac{B_n(f)}{\||\cdot|^s f\|_{L^2(\R\setminus \B)}}
    +\||\cdot|^{s+1}f'\|_{L^2(\R\setminus\B_{1/n})}.
    \end{align*}
Both sides of this inequality have limits in
  the extended reals $\R\cup \{\infty\}$ and the claim follows by
  taking $n\to\infty$.
\end{proof}

\begin{lemma}
\label{lem:HardyR}
  Let $s<-\frac12$ and $R>0$. Then we have the estimate
\[ \||\cdot|^s f\|_{L^2(\B_R)}\lesssim
\||\cdot|^{s+1}f'\|_{L^2(\B_R)} \]
for all $f\in C^1(\overline{\B_R})$ satisfying
\[ \lim_{x\to 0}\left [|x|^{s+\frac12}|f(x)|\right ]=0. \]
\end{lemma}

\begin{proof}
  Integration by parts, cf.~ the proof of Lemma \ref{lem:Hardy1}.
\end{proof}

Next, we derive a convenient expression for the $H^k(\R^5)$ norm of a
radial function $f\in \mc S(\R^5)$ in terms of a weighted $H^k(\R)$
norm of its representative $\widehat f(|x|)=f(x)$.

\begin{lemma}
  \label{lem:Hk5}
  Let $k\in \N_0$. Then we have
  \[ \|f\|_{H^k(\R^5)}\simeq \||\cdot|^2 \widehat f\|_{H^k(\R)} \]
  for all radial $f\in \mc S(\R^5)$, where $f(x)=\widehat f(|x|)$.
\end{lemma}
 
\begin{proof}
  For $k=0$ the statement is trivial. Thus, assume $k\in \N$.  Let
  \[ \mc F_d f(\xi):=\int_{\R^d}e^{-i \xi x}f(x)dx \]
  denote the Fourier transform in $d$ dimensions.  Since $f$ is
  radial, we have
  \[ \mc F_5 f(\xi)=(2\pi)^\frac52 |\xi|^{-\frac32}\int_0^\infty \widehat
  f(r)J_{3/2}(r|\xi|)r^\frac52 dr, \]
  see e.g.~\cite{Gra14a}, p.~577.  The Bessel function $J_{3/2}$ can
  be given in terms of elementary functions and we have
  \[ J_{3/2}(z)=\sqrt 2 \pi^{-\frac12}\left (z^{-\frac32}\sin
    z-z^{-\frac12}\cos z\right ). \] Consequently,
  \begin{align*}
    \mc F_5 f(\xi)&=8\pi^2 |\xi|^{-3}\int_0^\infty \widehat
                    f(r)\sin(r|\xi|)r dr
                    -8\pi^2 |\xi|^{-2}\int_0^\infty \widehat f(r)\cos(r|\xi|)r^2 dr \\
                  &=4i\pi^2|\xi|^{-3}\int_\R e^{-i|\xi| r}r\widehat
                    f(r)dr
                    -4\pi^2|\xi|^{-2}\int_\R e^{-i|\xi| r}r^2 \widehat
                    f(r)dr \\
                  &=4i\pi^2 |\xi|^{-3}\mc F_1((\cdot)\widehat f)(|\xi|)
                    -4\pi^2 |\xi|^{-2}\mc F_1(|\cdot|^2 \widehat f)(|\xi|)
  \end{align*}
  since $\widehat f$ is even. Lemma \ref{lem:Hardy1} now yields
  \begin{equation}
    \begin{split}
      \label{eq:fest1}
      \|f\|_{\dot H^k(\R^5)}&\simeq \||\cdot|^k \mc F_5
      f\|_{L^2(\R^5)} \lesssim \||\cdot|^{k-1}\mc F_1((\cdot)\widehat
      f)\|_{L^2(\R)}
      +\||\cdot|^k \mc F_1(|\cdot|^2\widehat f)\|_{L^2(\R)} \\
      &\lesssim \||\cdot|^k \mc F_1((\cdot)\widehat f)'\|_{L^2(\R)}
      +\||\cdot|^k \mc F_1(|\cdot|^2\widehat f)\|_{L^2(\R)} 
      \simeq \||\cdot|^k \mc F_1(|\cdot|^2\widehat f)\|_{L^2(\R)} \\
      &\simeq \||\cdot|^2
      \widehat f\|_{\dot H^k(\R)}
    \end{split}
  \end{equation}
  and thus,
  \[ \|f\|_{H^k(\R^5)}\simeq \||\cdot|^2\widehat
  f\|_{L^2(\R)}+\||\cdot|^2
  \widehat f\|_{\dot H^k(\R)}\lesssim \||\cdot|^2 \widehat
  f\|_{H^k(\R)}. \]

  Conversely, we have $\widehat f(r)=f(re_1)$ and thus,
  $\widehat f^{(j)}(r)=\partial_1^j f(re_1)$ for all $j\in
  \N_0$. Hardy's inequality yields
  \begin{align*}
    \|\widehat f\|_{L^2(\R)}&\simeq \||\cdot|^{-2}f\|_{L^2(\R^5)}\lesssim \|f\|_{\dot H^2(\R^5)} \\
    \||\cdot|\widehat f'\|_{L^2(\R)} &\simeq \||\cdot|^{-1}\partial_1 f\|_{L^2(\R^5)}
                                       \lesssim \|f\|_{\dot H^2(\R^5)} \\
    \||\cdot|^2 \widehat f''\|_{L^2(\R)}&\simeq \|\partial_1^2 f\|_{L^2(\R^5)}
                                          \lesssim \|f\|_{\dot H^2(\R^5)}.
  \end{align*}
  These estimates imply
  $\||\cdot|^2 \widehat f\|_{H^k(\R)}\lesssim \|f\|_{H^k(\R^5)}$,
  which finishes the proof.
\end{proof}

By a standard extension argument, the same bounds hold on balls.

\begin{lemma}
\label{lem:ext_apx}
Fix $R>0$ and $k\in \N_0$. Then there exists an extension $\mc E:
C^k(\overline{\B_R})\to C^k(\R)$ such that $\mc E
f|_{\B_R}=f$, $\supp(\mc E f)\subset
\B_{2R}$, and
\[ \|\mc E f\|_{H^k(\R\setminus \B_R)}\lesssim \|f\|_{H^k(\B_R\setminus
  \B_{R/2})} \]
for all $f\in C^k(\overline{\B_R})$.
\end{lemma}

\begin{proof}
We start with the simplest case $k=0$. Let $\chi: \R\to [0,1]$ be a smooth cut-off
satisfying
\[ \chi(x)=\begin{cases}1 & |x|\leq 1 \\
    0 & |x|\geq \frac32
  \end{cases}. \]
Then we define
\[ \mc E f(x):=\chi(\tfrac{x}{R})\begin{cases}
    f(-2R-x) & x<-R \\
    f(x) & x\in [-R,R] \\
    f(2R-x) & x>R 
  \end{cases}.
\]
Evidently, $\mc Ef|_{\B_R}=f$, $\supp(\mc E f)\subset \B_{\frac32
  R}\subset \B_{2R}$, and $\mc E f\in C(\R\setminus\{-R,R\})$. Furthermore,
\begin{align*}
  \lim_{x\to -R-}\mc E f(x)
  &=\chi(-1)f(-R)=f(-R)=\lim_{x\to -R+}f(x)=\lim_{x\to -R+}\mc E f(x) \\
  \lim_{x\to R+}\mc E f(x)
  &=\chi(1)f(R)=f(R)=\lim_{x\to R-}f(x)=\lim_{x\to R-}\mc E f(x)
\end{align*}
and thus, $\mc E f\in C(\R)$.
Finally,
\begin{align*}
  \|\mc E f\|_{L^2(\R\setminus \B_R)}^2
  &=\int_{-\infty}^{-R}|\mc E f(x)|^2 dx+\int_{R}^\infty |\mc E
    f(x)|^2 dx \\
  &=\int_{-\frac32 R}^{-R}\chi(\tfrac{x}{R})^2|f(-2R-x)|^2 dx
    +\int_{R}^{\frac 32 R}\chi(\tfrac{x}{R})^2 |f(2R-x)|^2 dx \\
  &\leq \int_{-\frac32 R}^{-R}|f(-2R-x)|^2 dx
    +\int_{R}^{\frac 32 R}|f(2R-x)|^2 dx \\
  &=\int_{-R}^{-\frac12 R}|f(x)|^2 dx +\int_{\frac12 R}^R |f(x)|^2 dx
  \\
  &=\|f\|_{L^2(\B_R\setminus\B_{R/2})}^2.
\end{align*}
Note, however, that in general $\mc E f\notin C^1(\R)$ since
\[ \lim_{x\to R+}(\mc E f)'(x)=-f'(R). \]
Thus, for $k\geq 1$ the above construction needs to be modified
slightly. The idea is to add suitable polynomials to compensate for
the lack of differentiability at the points $x=-R$ and $x=R$.
For instance, in the case $k=1$ we set
\[ \mc E f(x):=\chi(\tfrac{x}{R})\begin{cases}
    -f(-2R-x)+2f(-R) & x<-R \\
    f(x) & x\in [-R,R] \\
    -f(2R-x)+2f(R) & x>R 
  \end{cases}.
\]
Then $\mc E f\in C^1(\R)$. Furthermore, the one-dimensional Sobolev embedding
yields the bound $|f(-R)|+|f(R)|\lesssim \|f\|_{H^1(\B_R\setminus
  \B_{R/2})}$ and the claimed estimate follows. Similar constructions
exist for general $k\in \N$, cf.~\cite{DonSch16}, Lemma B.2. We omit the details.
\end{proof}

\begin{corollary}
\label{cor:Hk5}
  Fix $R>0$ and $k\in \N$. Then we have
  \[ \|f\|_{H^k(\B^5_R)}\simeq \||\cdot|^2 \widehat f\|_{H^k(\B_R)} \]
  for all radial $f\in C^\infty(\overline{\B^5_R})$, where
  $f(x)=\widehat f(|x|)$.
\end{corollary}

\begin{proof}
Let $\widetilde{\mc E}f(x):=(\mc E\widehat f)(|x|)$, where $\mc E:
C^k(\overline{\B_R})\to C^k(\R)$ is
an extension as in Lemma \ref{lem:ext_apx}.
Then, by Lemma \ref{lem:Hk5}, 
\begin{align*}
\|f\|_{H^k(\B_R^5)}&\lesssim \|\widetilde{\mc E} f\|_{H^k(\R^5)}
\simeq \||\cdot|^2\mc E \widehat f\|_{H^k(\R)} 
\simeq \||\cdot|^2 \widehat f\|_{H^k(\B_R)}+\||\cdot|^2 \mc E\widehat
  f\|_{H^k(\R\setminus \B_R)} \\
&\lesssim \||\cdot|^2 \widehat f\|_{H^k(\B_R)}+\|\widehat
  f\|_{H^k(\B_R\setminus\B_{R/2})} \\
&\lesssim \||\cdot|^2 \widehat f\|_{H^k(\B_R)}.
\end{align*}
Conversely,
\begin{align*}
  \||\cdot|^2\widehat f\|_{H^k(\B_R)}
&\lesssim \||\cdot|^2\mc E\widehat f\|_{H^k(\R)}
\simeq \|\widetilde{\mc E}f\|_{H^k(\R^5)}
\simeq \|f\|_{H^k(\B_R^5)}+\|\widetilde{\mc
  E}f\|_{H^k(\R^5\setminus\B_R^5)} \\
&\lesssim \|f\|_{H^k(\B_R^5)}+\|\mc E \widehat f\|_{H^k(\R\setminus
  \B_R)}
\lesssim \|f\|_{H^k(\B_R^5)}+ \|\widehat
  f\|_{H^k(\B_R\setminus\B_{R/2})} \\
&\lesssim \|f\|_{H^k(\B_R^5)},
\end{align*}
again by Lemma \ref{lem:Hk5}.
\end{proof}

To conclude this section, we consider a class of integral operators
that will appear frequently in the proof of Proposition \ref{prop:D}.

\begin{lemma}
  \label{lem:intop}
Let $R>0$, $m,n\in \N_0$, and $n+1-m\geq 0$. Furthermore, let
$\varphi\in C^\infty(\overline{\B_R})$ and define $T:
C^\infty(\overline{\B_R})\to C^\infty(\overline{\B_R})$ by
\[ Tf(x):=\frac{1}{x^m}\int_0^x y^n \varphi(y)f(y)dy. \]
Then we have the bound
\[ \|Tf\|_{H^k(\B_R)}\leq C_k \|f\|_{H^k(\B_R)} \]
for all $f\in C^\infty(\overline{\B_R})$ and $k\in \N_0$. In
particular, $T$ extends to a bounded operator on $H^k(\B_R)$.
\end{lemma}

\begin{proof}
  For $x\in \overline{\B_R}\setminus\{0\}$ we have
\[ Tf(x)=\frac{1}{x^m}\int_0^x y^n \varphi(y)f(y)dy
=x^{n+1-m}\int_0^1 t^n \varphi(tx)f(tx)dt \]
and this shows that $T$ maps $C^\infty(\overline{\B_R})$ to itself.
By the Leibniz rule we infer
\begin{align*} |(Tf)^{(j)}(x)|
&\leq C_k \sum_{\ell=0}^j \int_0^1 |f^{(\ell)}(tx)|dt 
\leq C_k \sum_{\ell=0}^j \frac{1}{x}\int_0^x |f^{(\ell)}(y)|dy
\end{align*}
for all $j=0,1,2,\dots,k$ and $x\in \B_R\setminus \{0\}$.
Consequently, Lemma \ref{lem:HardyR} yields the bound
\[ \|(Tf)^{(j)}\|_{L^2(\B_R)}\leq C_k \sum_{\ell=0}^j
\|f^{(\ell)}\|_{L^2(\B_R)}\leq C_k \|f\|_{H^k(\B_R)}
 \]
for all $j=0,1,2,\dots,k$.
\end{proof}

\section{Proof of Proposition \ref{prop:D}}
\label{sec:proofD}
\noindent Recall that $\mb D_5 \mb f=\widehat{\mb D}_5 \mb E_1 \mb f$, where $\mb
E_1\mb f(\eta)=\mb f(\eta e_1)$ and
\[ \widehat{\mb D}_5 \begin{pmatrix}
\widehat f_1 \\ \widehat f_2 \end{pmatrix}
=\begin{pmatrix}
a_{11} \widehat f_1' + a_{10}\widehat f_1+a_{20}\widehat f_2 \\
b_{12}\widehat f_1''
+b_{11}\widehat f_1'
+b_{10}\widehat f_1
+b_{21}\widehat f_2'
+b_{20}\widehat f_2
\end{pmatrix}. \]
In other words,
\begin{align*} 
&\mb D_5 \begin{pmatrix}
f_1 \\ f_2 \end{pmatrix}(\eta) \\
&=\begin{pmatrix}
a_{11}(\eta) \partial_1 f_1(\eta e_1) + a_{10}(\eta)f_1(\eta
e_1)+a_{20}(\eta) f_2(\eta e_1) \\
b_{12}(\eta) \partial_1^2 f_1(\eta e_1)
+b_{11}(\eta)\partial_1 f_1(\eta e_1)
+b_{10}(\eta) f_1(\eta e_1)
+b_{21}(\eta) \partial_1 f_2(\eta e_1)
+b_{20}(\eta) f_2(\eta e_1)
\end{pmatrix}. 
\end{align*}
The coefficients are of the form
\begin{align*}
  a_{11}(\eta)&=\eta^2 \varphi_\infty(\eta) 
& a_{10}(\eta)&=3\eta 
& a_{20}(\eta)&=\eta^3\varphi_\infty(\eta) \\
b_{12}(\eta)&=\eta^3\varphi_\infty(\eta) & 
b_{11}(\eta)&=\eta^2\varphi_\infty(\eta) &
b_{10}(\eta)&=-3\eta  \\
b_{21}(\eta)&=\eta^2\varphi_\infty(\eta) & b_{20}(\eta)&=\eta\varphi_\infty(\eta),
\end{align*}
where $\varphi_\infty\in C^\infty(\R)$ denotes a generic smooth and
even function with $\varphi_\infty(0)\not=0$.
From the form of the coefficients it is obvious that $\widehat{\mb
  D}_5$ maps even functions to odd functions.
Our first goal is to prove the bound
\[ \|\mb D_5 \mb f\|_{H^k(\B_R)\times H^{k-1}(\B_R)}\lesssim \|\mb
f\|_{H^{k+1}(\B_R^5)\times H^k(\B_R^5)}. \]
By Lemma \ref{lem:extension} it suffices to show
\begin{equation}
\label{eq:D5bound1}
 \|\mb D_5 \mb f\|_{H^k(\R)\times H^{k-1}(\R)}\lesssim \|\mb
f\|_{H^{k+1}(\R^5)\times H^k(\R^5)} 
\end{equation}
and we may assume that the coefficients of $\mb D_5$ have compact
support.
Hardy's inequality yields
\begin{align*}
  \|[\mb D_5\mb f]_1\|_{L^2(\R)}&\lesssim \||\cdot|^2 \partial_1 f_1((\cdot)e_1)\|_{L^2(\R)}
                                +\||\cdot|f_1((\cdot)e_1)\|_{L^2(\R)}+\||\cdot|^2 f_2((\cdot)e_1)\|_{L^2(\R)} \\
                              &\lesssim \|\nabla f_1\|_{L^2(\R^5)}+\||\cdot|^{-1}f_1\|_{L^2(\R^5)}
                                +\|f_2\|_{L^2(\R^5)} \\
                              &\lesssim \|f_1\|_{H^1(\R^5)}+\|f_2\|_{L^2(\R^5)}
\end{align*}
and from the Leibniz rule we infer
$\|[\mb D_5\mb f]_1\|_{H^k(\R)}\lesssim \|\mb f\|_{H^{k+1}(\R^5)\times
  H^k(\R^5)}$.
Analogously, we obtain
$\|[\mb D_5\mb f]_2\|_{H^{k-1}(\R^5)}\lesssim \|\mb
f\|_{H^{k+1}(\R^5)\times H^k(\R^5)}$ and this proves Eq.~\eqref{eq:D5bound1}.

Thus, it remains to show the more difficult reverse inequality
\[ \|\mb D_5\mb f\|_{H^k(\B_R)\times H^{k-1}(\B_R)}\gtrsim \|\mb
f\|_{H^{k+1}(\B^5_R)\times H^{k}(\B^5_R)} \]
or, by Corollary \ref{cor:Hk5}, the equivalent estimate
\begin{equation}
  \label{eq:D5bound2}
\left   \|\widehat{\mb D}_5 \widehat{\mb f} \right \|_{H^k(\B_R)\times
    H^{k-1}(\B_R)}\gtrsim \left \||\cdot|^2\, \widehat{\mb
    f}\right \|_{H^{k+1}(\B_R)\times H^k(\B_R)}.
\end{equation}
Let $\mb g:=\widehat{\mb D}_5\widehat{\mb f}$.  Then we have
\begin{align*}
  g_1&=a_{11}\widehat f_1'+a_{10} \widehat f_1+a_{20}\widehat f_2 \\
  g_2&=b_{12}\widehat f_1''+b_{11}\widehat f_1'+b_{10}\widehat f_1+b_{21}\widehat f_2'+b_{20}\widehat f_2 .
\end{align*}
We use the first equation to eliminate $\widehat f_2$ from the second
equation. This yields
\begin{equation}
  \label{eq:fhat}\widehat f_1''(\eta)+\left (\frac{6}{\eta}-\frac{h''(\eta)}{h'(\eta)}\right )\widehat f_1'(\eta)
  +3\frac{2h'(\eta)-\eta h''(\eta)}{\eta^2h'(\eta)}\widehat f_1(\eta)
  =\frac{a(\eta)}{\eta^2}g_1'(\eta) 
  +\frac{b(\eta)}{\eta}g_2(\eta),
\end{equation}
where 
\begin{align*}
  a(\eta)&=\frac{2\eta h'(\eta)-h(\eta)[1+h'(\eta)^2]}{\eta
           h'(\eta)-h(\eta)}, &
b(\eta)&=\frac{h'(\eta)}{\eta}\frac{1-h'(\eta)^2}{\eta h'(\eta)-h(\eta)}.
\end{align*}
Note that $a,b\in C^\infty(\R)$ are even and $a(0)\not=0$, $b(0)\not=0$.
If $g_1=g_2=0$, Eq.~\eqref{eq:fhat} has
the two solutions
\[ \phi(\eta)=\frac{\sqrt{2+\eta^2}-\sqrt 2}{\eta^3}=\frac{\varphi_\infty(\eta)}{\eta},\qquad
\psi(\eta)=\frac{1}{\eta^3} \] with Wronskian
\[ W(\eta)=\phi(\eta)\psi'(\eta)-\phi'(\eta)\psi(\eta)=-\frac{1}{\eta^5\sqrt{2+\eta^2}}. \]
Consequently, the variation of constants formula yields
\[ \widehat
f_1(\eta)=c_0\phi(\eta)+c_1\psi(\eta)-\phi(\eta)\int_0^\eta\frac{\psi(\eta')}{W(\eta')}g(\eta')
d\eta' +\psi(\eta)\int_0^\eta\frac{\phi(\eta')}{W(\eta')}g(\eta') d\eta' \]
where $c_0,c_1\in \R$ and $g$ denotes the right-hand side of
Eq.~\eqref{eq:fhat}.  Since $f_1(y)=\widehat f_1(|y|)$ belongs to
$H^3(\B_R^5)$, it follows that
$c_0=c_1=0$. In particular, this shows that the map $\mb D_5:
H^{k+1}_{\mathrm{rad}}(\B_R^5)\times H^k_{\mathrm{rad}}(\B_R^5) \to
H_-^k(\B_R)\times H_-^{k-1}(\B_R)$ is injective. 
Thus, we have $\widehat f_1=T_1 g_1+T_2 g_2$, where
\begin{align*}
  T_1 g_1(\eta)&:=\psi(\eta)\int_0^\eta
                 \frac{\phi(\eta')a(\eta')}{\eta'^2
                 W(\eta')}g_1'(\eta')d\eta' 
-\phi(\eta)\int_0^\eta
                 \frac{\psi(\eta')a(\eta')}{\eta'^2
                 W(\eta')}g_1'(\eta')d\eta' \\
  T_2 g_2(\eta)&:=\psi(\eta)\int_0^\eta
                 \frac{\phi(\eta')b(\eta')}{\eta'
                 W(\eta')}g_2(\eta')d\eta' 
-\phi(\eta)\int_0^\eta
            \frac{\psi(\eta')b(\eta')}{\eta'
                 W(\eta')}g_2(\eta')d\eta'.
\end{align*}
In view of Corollary \ref{cor:Hk5}, we have to prove the bounds
\begin{align*}
  \||\cdot|^2 T_1 g_1\|_{H^{k+1}(\B_R)}&\lesssim \|g_1\|_{H^k(\B_R)}  \\
  \||\cdot|^2 T_2 g_2\|_{H^{k+1}(\B_R)}&\lesssim \|g_2\|_{H^{k-1}(\B_R)}.
\end{align*}
An integration by parts using $g_1(0)=0$
yields
\begin{equation}
\begin{split}
\label{eq:T1g1}
 T_1 g_1(\eta)
&=\phi(\eta)\int_0^\eta
                 \partial_{\eta'}\left [\frac{\psi(\eta')a(\eta')}{\eta'^2
                 W(\eta')}\right ]g_1(\eta')d\eta' 
-\psi(\eta)\int_0^\eta
                 \partial_{\eta'}\left [\frac{\phi(\eta')a(\eta')}{\eta'^2
                 W(\eta')}\right ]g_1(\eta')d\eta' \\
         &=\frac{\varphi_\infty(\eta)}{\eta}\int_0^\eta \eta'\varphi_\infty(\eta')g_1(\eta')d\eta'
           +\frac{\varphi_\infty(\eta)}{\eta^3}\int_0^\eta \eta'\varphi_\infty(\eta')g_1(\eta')d\eta' 
\end{split}
\end{equation}
and for brevity we set
\begin{align*} T_{11}g_1(\eta)&:=\frac{\varphi_\infty(\eta)}{\eta}
\int_0^\eta
\eta'\varphi_\infty(\eta')g_1(\eta')d\eta' \\
  T_{12}g_1(\eta)&:=\frac{\varphi_\infty(\eta)}{\eta^3}\int_0^\eta
                \eta'\varphi_\infty(\eta')g_1(\eta')d\eta'.
\end{align*}
By Lemma \ref{lem:intop} we immediately obtain the bound 
$\||\cdot|^2T_{11}g_1\|_{L^2(\B_R)}\lesssim \|g_1\|_{L^2(\B_R)}$ and, since
\[ \partial_\eta \left [\eta^2 T_{11}g_1(\eta)\right
]=\varphi_\infty(\eta)\int_0^\eta
\eta'\varphi_\infty(\eta')g_1(\eta')d\eta'
+\eta^2 \varphi_\infty(\eta)g_1(\eta), \]
we infer $\||\cdot|^2 T_{11} g_1\|_{H^{k+1}(\B_R)}\lesssim \|g_1\|_{H^k(\B_R)}$.
Similarly, since
\[ \eta^2 T_{12}g_1(\eta)=\frac{\varphi_\infty(\eta)}{\eta}\int_0^\eta
                \eta'\varphi_\infty(\eta')g_1(\eta')d\eta' \]
and
\begin{align*}
  \partial_\eta [\eta^2 T_{12}g_1(\eta)]
&=\frac{\varphi_\infty(\eta)}{\eta^2}\int_0^\eta
  \eta'\varphi_\infty(\eta') g_1(\eta')d\eta'
+\varphi_\infty(\eta)g_1(\eta),
\end{align*}
Lemma \ref{lem:intop} yields $\|T_{12}g_1\|_{H^{k+1}(\B_R)}\lesssim \|g_1\|_{H^k(\B_R)}$.
Next, we turn to the operator $T_2$. We have
\[ \eta^2 T_2 g_2(\eta)=
\frac{\varphi_\infty(\eta)}{\eta}\int_0^\eta
\eta'^3\varphi_\infty(\eta')g_2(\eta')d\eta'
+\eta \varphi_\infty(\eta)\int_0^\eta \eta'
\varphi_\infty(\eta')g_2(\eta')d\eta' \]
and Lemma \ref{lem:intop} immediately yields the bound
$\||\cdot|^2T_2 g_2\|_{L^2(\B_R)}\lesssim \|g_2\|_{L^2(\B_R)}$.
Now we exploit the usual cancellation to obtain
\begin{align*}
 (T_2 g_2)'(\eta)
&=\psi'(\eta)\int_0^\eta\frac{\phi(\eta')b(\eta')}{\eta'W(\eta')}g_2(\eta')
  d\eta' 
-\phi'(\eta)
\int_0^\eta\frac{\psi(\eta')b(\eta')}{\eta'W(\eta')}g_2(\eta') d\eta' \\
&=\frac{\varphi_\infty(\eta)}{\eta^4}\int_0^\eta \eta'^3
  \varphi_\infty(\eta')g_2(\eta')d\eta'
+\frac{\varphi_\infty(\eta)}{\eta^2}\int_0^\eta \eta'
  \varphi_\infty(\eta')g_2(\eta')d\eta'
\end{align*}
and thus,
\begin{align*}
  \partial_\eta \left [\eta^2(T_2g_2)'(\eta)\right ]
&=\frac{\varphi_\infty(\eta)}{\eta^3}\int_0^\eta
  \eta'^3\varphi_\infty(\eta')g_2(\eta')d\eta' \\
&\quad +\eta\varphi_\infty(\eta)\int_0^\eta
  \eta'\varphi_\infty(\eta')g_2(\eta')d\eta'
+\eta\varphi_\infty(\eta)g_2(\eta),
\end{align*}
which yields the bound $\||\cdot|^2(T_2 g_2)'\|_{H^k(\B_R)}\lesssim
\|g_2\|_{H^{k-1}(\B_R)}$ by Lemma \ref{lem:intop}.
Analogously, we infer $\|(\cdot)T_2 g_2\|_{H^k(\B_R)}\lesssim \|g_2\|_{H^{k-1}(\B_R)}$
and in summary,
\begin{align*}
 \||\cdot|^2 T_2 g_2\|_{H^{k+1}(\B_R)}
&\lesssim \||\cdot|^2 T_2 g_2\|_{L^2(\B_R)}+\||\cdot|^2 (T_2 g_2)'\|_{H^k(\B_R)}
+\|(\cdot) T_2 g_2\|_{H^k(\B_R)} \\
&\lesssim \|g_2\|_{H^{k-1}(\B_R)},
\end{align*}
as desired.

Finally, we turn to $\widehat f_2$, which is given by
\begin{align*}
  \widehat f_2
&=\frac{g_1}{a_{20}}-\frac{a_{11}}{a_{20}}\widehat
  f_1'
-\frac{a_{10}}{a_{20}}\widehat f_1=S_1 g_1+S_2 g_2,
\end{align*}
where
\begin{align*}
  S_1 g_1
&:=\frac{g_1}{a_{20}}-\frac{a_{11}}{a_{20}}(T_1
  g_1)'-\frac{a_{10}}{a_{20}}T_1 g_1 \\
S_2 g_2&:=-\frac{a_{11}}{a_{20}}(T_2
               g_2)'-\frac{a_{10}}{a_{20}}T_2 g_2.
\end{align*}
We have to show the bounds 
\begin{align*}
\||\cdot|^2 S_1 g_1\|_{H^k(\B_R)}&\lesssim
\|g_1\|_{H^k(\B_R)} \\
\||\cdot|^2 S_2 g_2\|_{H^k(\B_R)}&\lesssim \|g_2\|_{H^{k-1}(\B_R)}.
\end{align*}
For the bound on $S_1$ we exploit some subtle cancellations. 
From Eq.~\eqref{eq:T1g1} we obtain
\begin{align*}
  (T_1 g_1)'(\eta)
&=\left [ \phi(\eta)\partial_\eta \frac{\psi(\eta)a(\eta)}{\eta^2
  W(\eta)}-\psi(\eta)\partial_\eta\frac{\phi(\eta)a(\eta)}{\eta^2
  W(\eta)}\right ]g_1(\eta) \\
&\quad +\phi'(\eta)\int_0^\eta
                 \partial_{\eta'}\left [\frac{\psi(\eta')a(\eta')}{\eta'^2
                 W(\eta')}\right ]g_1(\eta')d\eta' 
-\psi'(\eta)\int_0^\eta
                 \partial_{\eta'}\left [\frac{\phi(\eta')a(\eta')}{\eta'^2
                 W(\eta')}\right ]g_1(\eta')d\eta' \\
&=\left [\frac{1}{\eta^2}+\varphi_\infty(\eta)\right ]g_1(\eta) \\
&\quad +\phi'(\eta)\int_0^\eta
                 \partial_{\eta'}\left [\frac{\psi(\eta')a(\eta')}{\eta'^2
                 W(\eta')}\right ]g_1(\eta')d\eta' 
-\psi'(\eta)\int_0^\eta
                 \partial_{\eta'}\left [\frac{\phi(\eta')a(\eta')}{\eta'^2
                 W(\eta')}\right ]g_1(\eta')d\eta'
\end{align*}
and thus,
\begin{align*}
  a_{20}(\eta)S_1 g_1(\eta)
&=\left
  [1-\frac{a_{11}(\eta)}{\eta^2}+\eta^2\varphi_\infty(\eta)\right
  ]g_1(\eta) \\
&\quad -\left [a_{11}(\eta)\phi'(\eta)+a_{10}(\eta)\phi(\eta)\right ]\int_0^\eta
                 \partial_{\eta'}\left [\frac{\psi(\eta')a(\eta')}{\eta'^2
                 W(\eta')}\right ]g_1(\eta')d\eta' \\
&\quad +\left [a_{11}(\eta)\psi'(\eta)+a_{10}(\eta)\psi(\eta)\right ]\int_0^\eta
                 \partial_{\eta'}\left [\frac{\phi(\eta')a(\eta')}{\eta'^2
                 W(\eta')}\right ]g_1(\eta')d\eta' \\
&=\eta^2 \varphi_\infty(\eta)g_1(\eta)
+\varphi_\infty(\eta)\int_0^\eta \eta'
  \varphi_\infty(\eta')g_1(\eta')d\eta'
\end{align*}
since
$a_{11}(\eta)\psi'(\eta)+a_{10}(\eta)\psi(\eta)=\varphi_\infty(\eta)$.
We have $|\frac{a_{20}(\eta)}{\eta^3}|\gtrsim 1$ for all $\eta\in \B_R$
which implies
\[ S_1 g_1(\eta)=\frac{\varphi_\infty(\eta)}{\eta}g_1(\eta)
+\frac{\varphi_\infty(\eta)}{\eta^3}\int_0^\eta
\eta'\varphi_\infty(\eta')g_1(\eta')d\eta' \]
and Lemma \ref{lem:intop} yields the desired bound $\||\cdot|^2 S_1
g_1\|_{H^k(\B_R)}\lesssim \|g_1\|_{H^k(\B_R)}$.
Similarly, we have
\begin{align*}
  a_{20}(\eta)S_2 g_2(\eta)
&=-a_{11}(\eta)(T_2 g_2)'(\eta)-a_{10}(\eta)T_2 g_2(\eta) \\
&=\left [a_{11}(\eta)\phi'(\eta)+a_{10}(\eta)\phi(\eta)\right ]
\int_0^\eta \frac{\psi(\eta')b(\eta')}{\eta'W(\eta')}g_2(\eta')d\eta'
  \\
&\quad -\left [a_{11}(\eta)\psi'(\eta)+a_{10}(\eta)\psi(\eta)\right]
\int_0^\eta \frac{\phi(\eta')b(\eta')}{\eta'W(\eta')}g_2(\eta')d\eta'
  \\
&=\varphi_\infty(\eta)\int_0^\eta \eta'
  \varphi_\infty(\eta')g_2(\eta')d\eta'
+\varphi_\infty(\eta)\int_0^\eta \eta'^3 \varphi_\infty(\eta')g_2(\eta')d\eta'
\end{align*}
and thus,
\[ S_2 g_2(\eta)=\frac{\varphi_\infty(\eta)}{\eta^3}\int_0^\eta \eta'
  \varphi_\infty(\eta')g_2(\eta')d\eta'
+\frac{\varphi_\infty(\eta)}{\eta^3}\int_0^\eta \eta'^3
\varphi_\infty(\eta')g_2(\eta')d\eta'. \]
From Lemma \ref{lem:intop} we infer the bound $\||\cdot|^2 S_2
g_2\|_{L^2(\B_R)}\lesssim \|g_2\|_{L^2(\B_R)}$ and
\begin{align*}
 \partial_\eta \left [\eta^2 S_2 g_2(\eta)\right ]
&=\frac{\varphi_\infty(\eta)}{\eta^2}\int_0^\eta \eta'
  \varphi_\infty(\eta')g_2(\eta')d\eta'
+\varphi_\infty(\eta)g_2(\eta) 
\end{align*}
yields $\||\cdot|^2 S_2 g_2\|_{H^k(\B_R)}\lesssim \|g_2\|_{H^{k-1}(\B_R)}$, again
by Lemma \ref{lem:intop}.

\section{A monotonicity formula for the free wave equation}

\noindent In this section we prove a version of the energy identity
that is used for finite speed of propagation.

\begin{lemma}
\label{lem:apxdiv}
  Let $f\in C^1(\overline{\B^d})$ and $d\in \N$. Then we have
\[ \int_{\B^d}f(x)x^j\partial_{x^j} f(x)dx=-\frac{d}{2}\int_{\B^d}f(x)^2
dx+\frac12 \int_{\S^{d-1}}f(\omega)^2 d\sigma(\omega). \]
\end{lemma}

\begin{proof}
  The chain rule yields
\[ \partial_{x^j}\left [x^j f(x)^2\right]=df(x)^2+2f(x)x^j\partial_{x^j}
f(x) \]
and thus, from the divergence theorem we infer
\begin{align*}
  \int_{\B^d}f(x)x^j\partial_{x^j} f(x)dx
&=-\frac{d}{2}\int_{\B^d}f(x)^2 dx+\frac12 \int_{\B^d}
\partial_{x^j}\left [x^j f(x)^2\right]dx \\
&=-\frac{d}{2}\int_{\B^d}f(x)^2 dx+\frac12 \int_{\partial\B^d}\omega_j\omega^jf(\omega)^2d\sigma(\omega).
\end{align*}
\end{proof}

\begin{lemma}
\label{lem:apxenid}
Let $d\geq 3$ and $T>0$. For $u\in C^2(\R^{1,d})$ and $t\in [0,T)$ set 
\[ E_u(t):=\int_{\B_{T-t}^d}\partial_0 u(t,x)^2 dx+\int_{\B_{T-t}^d}\partial^j u(t,x)\partial_j
  u(t,x)dx
+\frac{1}{T-t}\int_{\partial
\B^d_{T-t}}u(t,\omega)^2 d\sigma(\omega). \]
If $\partial^\mu\partial_\mu u=0$, the function $E_u: [0,T)\to
[0,\infty)$ is monotonically decreasing.
\end{lemma}

\begin{proof}
  We write $E_u=E^0_u+E^1_u+B_u$, where
  \begin{align*}
    E^0_u(t)&=\int_{\B_{T-t}^d}\partial_0 u(t,x)^2 dx=(T-t)^d
              \int_{\B^d}\partial_0 u(t,(T-t)x)^2dx \\
    E^1_u(t)&=\int_{\B_{T-t}^d}\partial^j u(t,x)\partial_j
              u(t,x)dx=(T-t)^d \int_{\B^d}\partial^j u(t,(T-t)x)\partial_j
              u(t,(T-t)x)dx \\
    B_u(t)&=\frac{1}{T-t}\int_{\partial
            \B^d_{T-t}}u(t,\omega)^2 d\sigma(\omega)
            =(T-t)^{d-2}\int_{
            \S^{d-1}}u(t,(T-t)\omega)^2 d\sigma(\omega).
  \end{align*}
Using Lemma \ref{lem:apxdiv}, we compute
\begin{align*}
  \frac{d}{dt}E_u^0(t)
&=-\frac{d}{T-t}E_u^0(t) \\
&\quad +2(T-t)^d \int_{\B^d}\partial_0 u(t,(T-t)x)\left
  [\partial_0^2 u(t,(T-t)x)-x^j\partial_j 
\partial_0 u(t,(T-t)x) \right]dx \\
&=-\frac{d}{T-t}E_u^0(t) 
 +2(T-t)^d \int_{\B^d}\partial_0 u(t,(T-t)x)\partial_0^2
  u(t,(T-t)x)dx \\
&\quad -2(T-t)^{d-1}\int_{\B^d}\partial_0 u(t,(T-t)x)x^j\partial_{x^j} 
\partial_0 u(t,(T-t)x)dx \\
&=-\frac{d}{T-t}E_u^0(t) 
 +2(T-t)^d \int_{\B^d}\partial_0 u(t,(T-t)x)\partial_0^2
  u(t,(T-t)x)dx \\
&\quad +d(T-t)^{d-1}\int_{\B^d}\partial_0 u(t,(T-t)x)^2 dx
-(T-t)^{d-1}\int_{\S^{d-1}}\partial_0 u(t,(T-t)\omega)^2 d\sigma(\omega)
\end{align*}
and thus,
\begin{align*}
  \frac{d}{dt}E_u^0(t)
&=2\int_{\B^d_{T-t}}\partial_0
  u(t,x)\partial_0^2 u(t,x)dx-\int_{\partial\B^d_{T-t}}\partial_0 u(t,\omega)^2 d\sigma(\omega).
\end{align*}
Analogously, we obtain
\[ \frac{d}{dt}E_u^1(t)=2\int_{\B^d_{T-t}}\partial^j
u(t,x)\partial_j\partial_0
u(t,x)dx-\int_{\partial\B^d_{T-t}}\partial^j
u(t,\omega)\partial_j u(t,\omega)d\sigma(\omega) \]
and an integration by parts yields
\begin{align*}
  \frac{d}{dt}E_u^1(t)
  &=-2\int_{\B^d_{T-t}}\partial^j\partial_j
    u(t,x)\partial_0
    u(t,x)dx+2\int_{\partial\B^d_{T-t}}\frac{\omega^j}{|\omega|} \partial_j
    u(t,\omega)\partial_0
    u(t,\omega)d\sigma(\omega) \\
&\quad -\int_{\partial\B^d_{T-t}}\partial^j
u(t,\omega)\partial_j u(t,\omega)d\sigma(\omega).
\end{align*}
Finally,
\begin{align*}
  \frac{d}{dt}B_u(t)
&=-\frac{d-2}{T-t}B_u(t) \\
&\quad +2(T-t)^{d-2}\int_{\S^{d-1}}u(t,(T-t)\omega)\left [
\partial_0 u(t,(T-t)\omega)-\omega^j\partial_j u(t,(T-t)\omega)\right
  ]d\sigma(\omega) \\
&=-\frac{d-2}{(T-t)^2}\int_{\partial\B^d_{T-t}} u(t,\omega)^2
  d\sigma(\omega)
+\frac{2}{T-t}\int_{\partial\B_{T-t}^d}u(t,\omega)\partial_0
  u(t,\omega)d\sigma(\omega) \\
&\quad -\frac{2}{(T-t)^2}\int_{\partial\B_{T-t}^d}\omega^j\partial_j u(t,\omega)u(t,\omega)d\sigma(\omega).
\end{align*}
In summary, since $\partial_0^2 u-\partial^j\partial_j u=0$, we infer
\[ \frac{d}{dt}E_u(t)=\frac{d}{dt}E_u^0(t)+\frac{d}{dt}E_u^1(t)
+\frac{d}{dt}B_u(t)=\int_{\partial\B_{T-t}^d}A_u(t,\omega)d\sigma(\omega), \]
where
\begin{align*}
  A_u(t,\omega)
&=-\partial_0 u(t,\omega)^2-\partial^j u(t,\omega)\partial_j
  u(t,\omega)-(d-2)\frac{u(t,\omega)^2}{(T-t)^2} \\
&\quad +2\frac{\omega^j}{|\omega|}\partial_j u(t,\omega)\partial_0
  u(t,\omega)
+2\partial_0u(t,\omega)\frac{u(t,\omega)}{T-t}
-2\frac{\omega^j}{|\omega|}\partial_j
  u(t,\omega)\frac{u(t,\omega)}{T-t}. 
\end{align*}
Since $d\geq 3$ we obtain
\begin{align*}
A_u(t,\omega)&=-\left [\partial_0 u(t,\omega)-\frac{u(t,\omega)}{T-t}\right
  ]^2-\partial^j u(t,\omega)\partial_j
  u(t,\omega)-(d-3)\frac{u(t,\omega)^2}{(T-t)^2} \\
&\quad +2\frac{\omega^j}{|\omega|}\partial_j u(t,\omega)\left [
\partial_0 u(t,\omega)-\frac{u(t,\omega)}{T-t}\right ] \\
&\leq 0
\end{align*}
for all $t\in [0,T)$ and $\omega\in \partial\B_{T-t}^d$ by Cauchy's inequality.
\end{proof}

\bibliographystyle{plain} \bibliography{wmnew}
 
\end{document}